\documentclass[a4paper,reqno,11pt]{amsart}
\usepackage{amsfonts,amsmath,amsthm,amssymb,stmaryrd}
\usepackage{mathrsfs}

\numberwithin{equation}{section}
\usepackage{color}
\allowdisplaybreaks

\usepackage[left=1 in, right=1 in,top=1 in, bottom=1 in]{geometry}

\providecommand{\abs}[1]{\left\vert#1\right\vert}
\providecommand{\norm}[1]{\left\Vert#1\right\Vert}

\providecommand{\Rn}[1]{\mathbb{R}^{#1}}

\providecommand{\sd}[1]{\mathcal{D}_{#1}}
\providecommand{\se}[1]{\mathcal{E}_{#1}}
\providecommand{\sdb}[1]{\bar{\mathcal{D}}_{#1}}
\providecommand{\seb}[1]{\bar{\mathcal{E}}_{#1}}
\providecommand{\feb}[1]{\bar{\mathfrak{E}}_{#1}}
\providecommand{\fdb}[1]{\bar{\mathfrak{D}}_{#1}}

\providecommand{\ns}[1]{\norm{#1}^2}
\providecommand{\as}[1]{\abs{#1}^2}

\def\Lbrack{\left \llbracket}
\def\Rbrack{\right \rrbracket}

\DeclareMathOperator{\diverge}{div}

\providecommand{\Rn}[1]{\mathbb{R}^{#1}}
\providecommand{\norm}[1]{\left\Vert#1\right\Vert}
\def\ls{\lesssim}
\def\gss{\gtrsim}

\def\dt{\partial_t}
\def\H{H^1_0(\Omega)}
\def\Hs{H^1_{0,\sigma}(\Omega)}

\def\Hsd{(\Hs)^\ast}

\def\pa{\partial}

\def\rj{\Lbrack  {\rho} \Rbrack}
\providecommand{\jump}[1]{\left\llbracket #1 \right\rrbracket }

\def\RRvert2{\right \vert\! \right\vert}
\def\Lvert3{\left \vert\!\left\vert\!\left\vert}
\def\Rvert3{\right \vert\!\right\vert\!\right\vert}

\def\nab{\nabla}
\def\al{\alpha}
\def\dt{\partial_t}
\def\dtt{ \frac{d}{dt}}
\def\hal{\frac{1}{2}}

\def\ls{\lesssim}
\def\p{\partial}
\def\sg{\mathbb{D}}

\def\da{\Delta_{\mathcal{A}}}
\def\naba{\nab_{\mathcal{A}}}
\def\diva{\diverge_{\mathcal{A}}}
\def\Sa{S_{\mathcal{A}}}

\def\dis{\displaystyle}

\def\lam{\lambda}
\def\Lam{\lambda}
\def\mc{\mathcal{M}_c}

\def\a{\mathcal{A}}
\def\f{\mathcal{F}_{2N}}

\def\fj1{\mathcal{J}^{-1}}

\def\n{\mathcal{N}}

\def\x{\mathcal{X}}
\def\y{\mathcal{Y}}

\def\S{\mathbb{S}}

\def\fe{\mathfrak{E}_}
\def\fd{\mathfrak{D}_}

\newtheorem{lem}{Lemma}[section]

\newtheorem{prop}[lem]{Proposition}
\newtheorem{thm}[lem]{Theorem}
\newtheorem{remark}[lem]{Remark}

\title[Viscous non-resistive MHD internal waves]{Sharp nonlinear stability criterion of viscous non-resistive MHD internal waves in 3D}

\author{Yanjin Wang}
\address{
School of Mathematical Sciences\\
Xiamen University\\
Xiamen, Fujian 361005, China}
\email[Y. J. Wang]{yanjin$\_$wang@xmu.edu.cn}
\thanks{This work was supported by the Fujian Province
Natural Science Funds for Distinguished Young Scholar (No. 2015J06001), and A Foundation
for the Author of National Excellent Doctoral Dissertation of PR China (No. 201418).}


\begin{document}

\begin{abstract}
We consider the dynamics of two layers of incompressible electrically conducting fluid interacting with the magnetic field, which are confined within a 3D horizontally infinite slab and separated by a free internal interface. We assume that the upper fluid is heavier than the lower fluid so that the fluids are susceptible to the Rayleigh-Taylor instability. Yet, we show that the viscous and non-resistive problem around the equilibrium is nonlinearly stable provided that the strength of the vertical component of the steady magnetic field, $\abs{\bar B_3}$, is greater than the critical value, $\mathcal{M}_c$, which we identify explicitly. We also prove that the problem is nonlinearly unstable if $\abs{\bar B_3}<\mathcal{M}_c$. Our results indicate that the non-horizontal magnetic field has strong stabilizing effect on the Rayleigh-Taylor instability but the horizontal one does not have in 3D.
\end{abstract}

\maketitle


\section{Introduction}

\subsection{Eulerian formulation}

We consider two distinct, immiscible, incompressible electrically conducting fluids interacting with the magnetic field that evolve within the infinite slab $\Omega=\mathbb{R}^2\times(-m,\ell)$ with constants $m,\ell>0$. The fluids are separated by a free internal interface $\Sigma(t)$ that extends to infinity in every horizontal direction. The interface divides $\Omega$ into two time-dependent, disjoint, open subsets $\Omega_\pm(t)$, which are filled by the ``upper fluid" $(+)$ and  the ``lower fluid" $(-)$, respectively. The motions of the fluids are driven by the gravitational force and the induced Lorentz force \cite{Ca,Co,LL,GLL}. The two fluids are described by their velocity, pressure and magnetic field functions, which are given for each $t\ge0$ by, respectively,
\begin{equation}
(\tilde u_\pm, \tilde p_\pm, \tilde B_\pm)(t,\cdot):\Omega_\pm(t)\rightarrow
(\mathbb{R}^3,\mathbb{R},\mathbb{R}^3).
\end{equation}

We assume that the fluids are viscous and non-resistive. Then for each $t>0$ we require that $(\tilde u_\pm, \tilde p_\pm, \tilde B_\pm)$ satisfy the following equations of magnetohydrodynamics (MHD):
\begin{equation}\label{ns_euler0}
\begin{cases}
 {\rho}_\pm   (\partial_t\tilde{u}_\pm  +   \tilde{u}_\pm \cdot \nabla \tilde{u}_\pm ) +  \diverge \mathbb{T}_\pm  =-g {\rho}_\pm e_3 & \text{in } \Omega_\pm(t)
\\ \partial_t\tilde{B}_\pm  +   \tilde{u}_\pm \cdot \nabla \tilde{B}_\pm  =\tilde{B}_\pm \cdot \nabla \tilde{u}_\pm & \text{in } \Omega_\pm(t)
\\\diverge{\tilde u_\pm}=\diverge{\tilde B_\pm}=0 &\text{in } \Omega_\pm(t)
\\ \tilde V= (\tilde u \cdot \tilde n) \tilde n &\hbox{on }\Sigma(t)
\\ \tilde{u}_+=\tilde{u}_-,\quad \mathbb{T}_+ \tilde n=\mathbb{T}_- \tilde n &\hbox{on }\Sigma(t) \\
\tilde{u}_+(t,y_1,y_2,\ell)=\tilde{u}_-(t,y_1,y_2,-m)=0.
\end{cases}
\end{equation}
In the equations $-g  {\rho}_\pm e_3$ is the gravitational force with the constants $ {\rho}_\pm>0$ the densities of the two fluids, $g>0$ the acceleration of gravity and $e_3$ the vertical unit vector. The stress tensors $\mathbb{T}_\pm$ consist of the fluid part and the magnetic part \cite{PSO},
\begin{equation}
\mathbb{T}_\pm=-\mu_{\pm}\mathbb{D}(\tilde u_\pm)+ \tilde p_\pm I + \frac{ {|\tilde B_\pm|}^2}{2}I  -\tilde B_\pm\otimes \tilde B_\pm ,
\end{equation}
where $\mu_\pm$ denote the viscosities coefficients of the respective fluids and we have written $I$ for the $3\times3$ identity matrix and $\mathbb{D}(\tilde u_\pm)_{ij}=\partial_j\tilde u_{i,\pm} +\partial_i\tilde u_{j,\pm}$ for twice the velocity deformation tensor. The fourth equation in \eqref{ns_euler0} is called kinematic boundary condition since it implies that the free interface is advected with the fluids, where $\tilde V(t,y)$ is the normal velocity of the free interface at $y\in\Sigma(t)$ and $\tilde n(t,y)$ is the upward normal vector to $\Sigma(t)$ at $y$. Note that the continuity of velocity on $\Sigma(t)$, $\tilde{u}_+=\tilde{u}_-$, means that it is the common value of $\tilde u_\pm$ that advects the interface. Note that in the dynamic boundary condition on $\Sigma(t)$ the effect of surface tension is not taken into account, which implies the continuity of the normal stress on $\Sigma(t)$ \cite{WL}.

To complete the statement of the problem, we must specify the initial conditions. We suppose that the initial interface $\Sigma(0)$ is given, which yields the open sets $\Omega_\pm(0)$ on which we specify the initial data for the velocity, $\tilde{u}_\pm(0): \Omega_\pm(0) \rightarrow  \mathbb{R}^3$ and the magnetic field, $\tilde B_\pm(0): \Omega_\pm(0) \rightarrow  \mathbb{R}^3$.

\subsection{Lagrangian formulation}
The movement of the free interface $\Sigma(t)$ and the subsequent change of the domains $\Omega_\pm(t)$ create numerous mathematical difficulties. To circumvent these, we will switch to coordinates in which the interface and the domains stay fixed in time. Since we are interested in the nonlinear stability of the equilibrium state, we will use the equilibrium domains. To this end we define the fixed
domains $\Omega_+=\mathbb{R} \times(0,\ell)$ and
$\Omega_-=\mathbb{R} \times(-m,0)$, and we would like $\Omega_\pm$ to be the equilibrium domains. We shall write  $\Sigma:=\{x_3=0\}$ for the equilibrium interface, $\Sigma_{m}:=\{x_3=-m\}$ and $\Sigma_{\ell}:=\{x_3=\ell\}$ for the lower and upper boundaries, respectively.  We will also write $\Sigma_{m,\ell}= \Sigma_m \cup \Sigma_\ell$.

We assume that there exist invertible mappings
\begin{equation}
\eta_{0,\pm} :\Omega_\pm \rightarrow
\Omega_\pm(0)
\end{equation}
which are continuous across $\Sigma$ so that $\Sigma(0)=\eta_{0,\pm}(\Sigma)$, $\Sigma_{m}=\eta_{0,-}(\Sigma_{m})
$ and $\Sigma_{\ell}=\eta_{0,+}(\Sigma_{\ell})
$. The first condition means that $\Sigma(0)$ is parameterized by the
either of the mappings $\eta_\pm$ restricted to $\Sigma$ (which
one is irrelevant since they are continuous across $\Sigma$). Define the flow maps
$\eta_\pm$ as the solutions to
\begin{equation}
\begin{cases}
\partial_t\eta_\pm(t,x)=\tilde u_\pm(t,\eta_\pm(t,x)),\\
\eta_\pm(0,x)=\eta_{0,\pm}(x).
\end{cases}
\end{equation}
We think of the Eulerian coordinates as $(t,y)\in \mathbb{R}^+\times\Omega_\pm(t)$ with $y =
\eta_\pm(t,x)$, whereas we think of Lagrangian coordinates as the fixed
$(t,x)\in \mathbb{R}^+\times\Omega_\pm$. In order to switch back and
forth from Lagrangian to Eulerian coordinates we assume that $\eta_\pm(t,\cdot)$ are invertible and  $\Omega_\pm(t) = \eta_\pm(
t,\Omega_\pm)$. We also assume that $\Sigma(t)=\eta_{\pm}(t,\Sigma)$, $\Sigma_{m}=\eta_{-}(t,\Sigma_m)
$ and $\Sigma_{\ell}=\eta_{+}(t,\Sigma_\ell)$ (the first assumption means that we have assumed the continuity of $\eta_\pm$ across $\Sigma$); these follows by the initial assumption of $\eta_{0,\pm}$ and the boundary conditions of $u_\pm$.

If $\eta_\pm-Id $ are sufficiently small (in an appropriate Sobolev space), then the mappings $\eta_\pm$ are diffeomorphism.  This allows us to transform the problem \eqref{ns_euler0} to one in the fixed spatial domains $\Omega_\pm$. We define the Lagrangian unknowns on $\Omega_\pm$ by the compositions
\begin{equation}
(u_\pm,p_\pm,B_\pm)(t,x)=(\tilde u_\pm,\tilde p_\pm+\frac{ {|\tilde B_\pm|}^2}{2}+ g\rho_\pm y_3,\tilde B_\pm)(t,\eta_\pm(t,x)),\quad (t,x)\in \mathbb{R}^+\times\Omega_\pm.
\end{equation}
Here we have defined the modified pressure $p_\pm$ so that it will be more convenient for our stability analysis. Since the domains $\Omega_\pm$ are now fixed, we henceforth consolidate notation by writing $f$ to refer to $f_\pm$ except when necessary to distinguish the two; when we write an equation for $f$ we assume that the equation holds with the subscripts added on the domains $\Omega_\pm$. If $f$ appears in an equation on the interface $\Sigma$, implicitly it is continuous across $\Sigma$. To write the jump conditions on $\Sigma$, for a quantity $f=f_\pm$, we define the interfacial jump as
\begin{equation}
\jump{f} := f_+ \vert_{\Sigma} - f_- \vert_{\Sigma}.
\end{equation}
Then in Lagrangian coordinates, the PDEs \eqref{ns_euler0} becomes the following system for $(\eta,u,p,B)$:
\begin{equation}\label{lagrangian}
\begin{cases}
\partial_t\eta=u & \text{in }
\Omega
 \\\rho\partial_t u+\diva \S_\a(p,u)=B\cdot\nabla_\a B & \text{in }
\Omega
 \\ \partial_t B=B\cdot\nabla_\a u& \text{in }
\Omega
\\ \diva u=\diva B=0& \text{in }
\Omega
\\  \jump{u}=0,\quad \jump{ \S_\a(p,u)}  \n=\jump{ B\cdot\n B}+g\rj \eta_3 \n &\hbox{on }\Sigma \\
 {u}=0 &\text{on }\Sigma_{m,\ell}
 \\ (\eta,u)\mid_{t=0}=(\eta_0,u_0).
\end{cases}
\end{equation}
Here $\a=((\nabla\eta)^{-1})^T$ and we have written the differential operators $\nabla_\a,\diverge_\a$ with their actions given by
$(\naba f)_i := \a_{ij} \p_j f$, $\diva X := \a_{ij}\p_j X_i $
for appropriate $f$ and $X$.  We have also written
\begin{equation}\label{n_def}
\n=\left.\p_1\eta\times\p_2\eta\right|_\Sigma=\left. J\a e_3 \right|_\Sigma
\end{equation}
for the non-unit normal to $\Sigma(t)$ and
\begin{equation}
\S_\a(p,u): =-\mu  \sg_{\a} u+ p I,\   (\sg_{\a} u)_{ij} =  \a_{ik} \p_k u_j + \a_{jk} \p_k u_i.
\end{equation}
Note that if we extend $\diva$ to act on symmetric tensors in the natural way, then $\diva \Sa(p,u) = \naba p -\mu \da u$ for vector fields satisfying $\diva u=0$, where $\da f := \diva \naba f$.
Note that the kinematic boundary condition, the fourth equation in \eqref{ns_euler0}, is automatically satisfied by the first equation in \eqref{lagrangian}.
Recall that $\a$ is determined by $\eta$. This means that all of the differential operators are connected to $\eta$, and hence to the geometry of the free interface.

\subsection{Steady states, conserved quantities and reformulation}
The system \eqref{lagrangian} admits the steady solution with $\eta=Id, \ u=0,\ p=\hbox{const.}$ and $B=\bar B$, where $\bar B$ is a uniform magnetic field. Notice that the densities of the two fluids generally are different. The jump of the density $\rj=\rho_+-\rho_- $ is of fundamental importance in the analysis of solutions to \eqref{lagrangian} (equivalently, \eqref{ns_euler0}) near the equilibrium.  Indeed, if $\rj > 0$, then the upper fluid is heavier than the lower fluid, and the fluids are susceptible to the well-known Rayleigh-Taylor gravitational instability \cite{Ra,Ta}. We assume that $\rj > 0$ in this paper.

It is a key to find out the conserved quantities for the system \eqref{lagrangian}; these will help us reformulate the system in a proper way, and the reformulation will be more suitable for our nonlinear stability analysis. Indeed, these conserved quantities indicate the conditions which are needed to be imposed on the initial data if one wants to show the asymptotic stability of the equilibrium. To begin with, we denote $J={\rm det} (\nabla\eta)$, the Jacobian of the coordinate transformation. First, direct computation, together with the incompressiblity of the fluids, yields that
\begin{equation}\label{jequ}
\partial_t J =J\diverge_\a    u=0 .
\end{equation}
Next, applying $\a^T $ to the magnetic equation, we obtain
\begin{equation}
\a_{ji}\partial_tB_j=\a_{ji}B_k\a_{kl} \partial_lu_j =\a_{ji}B_k\a_{kl}\partial_t(\partial_l\eta_j)
=- \partial_t\a_{ji}B_k\a_{kl}\partial_l\eta_j=-B_j\partial_t\a_{ji}.
\end{equation}
This implies that
\begin{equation}\label{re00}
\partial_t (\a^T B)=0.
\end{equation}
It then follows from \eqref{jequ}, \eqref{re00} and \eqref{n_def} that
\begin{equation}
\partial_t (\diva B)= J^{-1}\partial_t \diverge(J\a^T B)=0
\end{equation}
and
\begin{equation}
\partial_t(B_\pm\cdot\n) =\partial_t(B_\pm\cdot  J_\pm \a_\pm e_3 )=\partial_t(J_\pm B_\pm^T  \a_\pm e_3 )=0\text{ on }\Sigma.
\end{equation}
Here we have used the well-known geometric identity $\partial_j(J\a_{ij})=0$.
Finally,
\begin{equation}\label{etaeq}
\dt \eta =0 \text{ on } \Sigma_{m,\ell}\text{ and }\dt \jump{\eta} =0 \text{ on } \Sigma.
\end{equation}
Hence, by \eqref{jequ}, \eqref{re00} and \eqref{etaeq}, we have
\begin{equation}\label{conserv}
 J =1,\ \a^T B= \bar B \text{ in }\Omega, \ \eta =Id \text{ on } \Sigma_{m,\ell}\text{ and }\jump{\eta} =0 \text{ on } \Sigma,
\end{equation}
if we have assumed that the initial data has the same quantities as the equilibrium. Note then that $\diva B = \diverge\bar B=0$ in $\Omega$ and $B_\pm\cdot\n=\bar B_3$ on $\Sigma$.

The conservation analysis above reveals that in order to have our nonlinear stability, the magnetic field $B$ should have certain relations with the flow map $\eta$. In turn, this motivates us to eliminate the magnetic field $B$ from the system \eqref{lagrangian}. Indeed, since $B=\bar B\cdot\nabla\eta$ by the second identity in \eqref{conserv}, we may rewrite the Lorentz force term:
\begin{equation}
B\cdot\nabla_\a B=B_j\a_{jk}\partial_kB=\bar{B}_k\partial_k(\bar{B}_m\partial_m\eta)
 =(\bar B\cdot\nabla)^2\eta.
\end{equation}
Then the system \eqref{lagrangian} can be reformulated as a Navier-Stokes system with forcing terms induced by the flow map:
\begin{equation}\label{reformulationic}
\begin{cases}
\partial_t\eta=u & \text{in }
\Omega
 \\\rho\partial_t u+\diva S_\a(p,u)=(\bar B\cdot\nabla)^2\eta & \text{in }
\Omega
\\ \diva u=0 & \text{in }
\Omega
\\ \jump{u}= 0,\quad \jump{ S_\a(p,u)}  \n=\jump{ \bar B_3 (\bar B\cdot\nabla) \eta}+g\rj \eta_3 \n &\hbox{on }\Sigma \\
u= 0 &\text{on }\Sigma_{m,\ell}
 \\ (\eta,u)\mid_{t=0}=(\eta_0,u_0).
\end{cases}
\end{equation}
Here we have shifted $\eta\rightarrow Id+\eta$, and hence $\a=((I+\nabla\eta)^{-1})^T$, etc.. We may also record those conserved quantities from \eqref{conserv} correspondingly:
\begin{equation}\label{imp1}
{\rm det} (I+\nabla\eta) =1\text{ in }\Omega, \ \eta =0 \text{ on } \Sigma_{m,\ell}\text{ and }\jump{\eta} =0 \text{ on } \Sigma.
\end{equation}

\subsection{Previous works of Rayleigh-Taylor problems}
The Rayleigh-Taylor instability, one of the classic examples of hydrodynamic instability, is an interfacial instability between two fluids of different densities that occurs when a heavy fluid lies above a lighter one in a gravitational field. The instability is well-known since the classical work of Rayleigh \cite{Ra} and of Taylor \cite{Ta}, and it is one of the fundamental problems in fluid dynamics. A general discussion of the physics related to this topic can be found, for example, in \cite{3K}.

The Rayleigh-Taylor problem has received a lot of attention in the mathematics community due both to its physical importance and to the mathematical challenges it offers.  The linear stability and instability of the Rayleigh-Taylor problem is extensively studied (see, for instance, Chandrasekhar's book \cite{3C}). We are mainly concerned with the nonlinear theory. For the Euler Rayleigh-Taylor problem without surface tension, Ebin \cite{3E} proved the nonlinear ill-posedness of the problem for incompressible fluids, and Guo and Tice \cite{3GT1} showed an analogous result for compressible fluids as \cite{3E}.  For the Navier-Stokes Rayleigh-Taylor problem, Pr$\ddot{\text{u}}$ss and Simonett \cite{PS2} proved the nonlinear instability for incompressible fluids with surface tension in the whole space, Wang and Tice \cite{WT} and Wang, Tice and Kim \cite{WTK} established the sharp nonlinear stability criteria for incompressible fluids with or without surface tension in the horizontally periodic slab, and Jang, Tice and Wang \cite{JTW_NRT,JTW_GWP} showed an analogous result for compressible fluids as \cite{WT,WTK} (the linear instability was previously shown by Guo and Tice \cite{3GT2}).

Note that for the horizontally infinite setting, the surface tension, no matter how large, cannot prevent the Rayleigh-Taylor instability \cite{PS2,3GT2}. It is very interesting to find stabilizing mechanisms that can prevent the Rayleigh-Taylor instability. In \cite{W}, we studied the stabilizing effect of the magnetic field on the Rayleigh-Taylor instability for the linearized system of \eqref{reformulationic} when $\bar B$ is either vertical or horizontal. We identified the critical magnetic number $|B|_c$ so that for the case $|\bar{B}|\ge |B|_c$ and when the magnetic field $\bar{B}$ is vertical in 2D or 3D or when $\bar{B}$ is horizontal in 2D we proved that the problem is linearly stable; for the rest cases the problem is linearly unstable. This shows the remarkable stabilizing effect of the magnetic field on the Rayleigh-Taylor instability in the linear sense, however, the nonlinear theory is left open. The main difficulty of the nonlinear problem lies in the derivation of the energy estimates, which is essentially due to the weak dissipation of $\eta$ in \eqref{reformulationic} when $\bar B\neq 0$ or lack of dissipation when $\bar B=0$ \cite{Lin}. In this paper we consider the nonlinear problem \eqref{reformulationic} in 3D, and the results are twofold. First, we characterize a stability criterion in terms of $\bar B_3$ so that it allows for $\bar B$ in any direction. Second, we completely solve the nonlinear stability and instability.  We remark that a weaker nonlinear instability result for $\bar B_3=0$ was obtained by Jiang, Jiang and Wang \cite{JJW}.

Finally, we mention that for the incompressible inhomogeneous problem in a fixed domain, one can construct the steady density which is continuous and increasing with height in certain region; this also leads to the Rayleigh-Taylor instability phenomenon. Hwang and Guo \cite{hw_guo} proved the nonlinear instability for the Euler equations, Jiang and Jiang \cite{JJ} proved the nonlinear instability for the Navier-Stokes equations in a bounded domain, and Jiang and Jiang \cite{JJ2} proved the linear stability and instability for the viscous non-resistive MHD equations in a bounded domain.

\section{Main results}

Before we state our results, let us first mention the issue of compatibility conditions for the initial data $(\eta_0,u_0)$ since our problem \eqref{reformulationic} is considered in a domain with boundary. We will work in a high-regularity context, essentially with regularity up to $2N$ temporal derivatives for $N \ge 4$ an integer. This requires us to use $(\eta_0,u_0)$ to construct the initial data $\partial_t^j \eta(0)$ and $\partial_t^j u(0)$ for $j=1,\dotsc,2N$ and $\partial_t^j  p(0)$ for $j = 0,\dotsc, 2N-1$, inductively. These data must satisfy various conditions (essentially what one gets by applying $\partial_t^j $ to \eqref{reformulationic} and then setting $t=0$), which in turn require $(\eta_0,u_0)$ to satisfy $2N$ compatibility conditions; these are natural for solutions to \eqref{reformulationic}
in our functional framework. Since they are similar as those of \cite{GT_lwp,WTK}, we shall neglect to record them in this
paper and refer to them as the necessary compatibility conditions.

We denote $H^k(\Omega_\pm)$ with $k\ge 0$ and $H^s(\Sigma)$ with $s \in \Rn{}$ for the usual Sobolev spaces. If we write $f \in H^k(\Omega)$, the understanding is that $f$ represents the pair $f_\pm$ defined on $\Omega_\pm$ respectively, and that $f_\pm \in H^k(\Omega_\pm)$. We will avoid writing $H^k(\Omega)$ or $H^s(\Sigma)$ in our norms and write
\begin{equation}
 \ns{f}_k = \ns{f_+}_{H^k(\Omega_+)} + \ns{f_-}_{H^k(\Omega_-)}   \text{ and }  \abs{f}_s = \ns{f }_{H^s(\Sigma )} .
\end{equation}
We introduce the following anisotropic Sobolev norm defined on $\Omega$:
\begin{equation}
\norm{f}_{k,l}:=\sum_{ \al_1+\al_2 \le l}\norm{\p_1^{\al_1}\p_2^{\al_2} f}_k.
\end{equation}

For a given jump value in the density $\rj>0$, we define the critical value
\begin{equation}\label{mc}
\mathcal{M}_c:= \sqrt{\frac{\rj g}{\frac{1}{\ell} +\frac{1}{m}  } }.
\end{equation}
Note that the definition of $\mc$ is independent of $\mu$.
We first state our stability result for the system \eqref{reformulationic} when $\abs{\bar B_3}> \mathcal{M}_c$. For this, we define some energy functionals. For a generic integer $n\ge 3$, we define the energy as
\begin{equation}\label{p_energy_def}
 \se{n} := \sum_{j=0}^{n}  \ns{\dt^j u}_{2n-2j} + \sum_{j=0}^{n-1}  \ns{\nabla \dt^j p}_{2n-2j-2}+ \sum_{j=0}^{n-1}  \as{  \jump{ \dt^jp}}_{2n-2j-3/2}
+\ns{ \eta}_{1,2n }+\ns{ \eta}_{2n }
\end{equation}
and the dissipation as
\begin{equation}\label{p_dissipation_def}
\begin{split}
 \sd{n} :=& \ns{ u}_{1,2n }+ \ns{ u}_{2n }+ \sum_{j=1}^{n}  \ns{\dt^j u}_{2n-2j+1}
+\ns{ \nabla p}_{2n-2 }+ \sum_{j=1}^{n-1}  \ns{\nabla \dt^j p}_{2n-2j-1}
\\&+\as{ \jump{ p}}_{2n-3/2 }+ \sum_{j=1}^{n-1}  \as{  \jump{ \dt^j p}}_{2n-2j-1/2}+\ns{(\bar B\cdot\nabla)\eta}_{0,2n }+\ns{ \eta}_{2n }.
\end{split}
\end{equation}
We will consider both $n=2N$ and $n=N+2$ for the integer $N\ge 4$.
We also define
\begin{equation}\label{fff}
\f:=  \ns{\eta}_{4N+1}\text{ and }\mathcal{J}_{2N}:=\ns{u}_{4N+1}+\ns{ \nabla p}_{4N-1}+\as{ \jump{ p}}_{4N-1/2 }.
\end{equation}
Finally, we define
\begin{equation}\label{G_def}
\begin{split}
\mathcal{G}_{2N}(t) :=& \sup_{0 \le r \le t} \se{2N} (r) + \int_0^t \sd{2N} (r) dr + \sup_{0 \le r \le t} (1+r)^{2N-4} \se{N+2} (r)
\\&+ \sup_{0 \le r \le t} \f(r)  + \int_0^t \frac{  \mathcal{J}_{2N}(r)}{(1+r)^{1+\vartheta}}dr
\end{split}
\end{equation}
for any fixed $0<\vartheta\le N-3$ (this requires that $N\ge 4$). Our global well-posedness result of \eqref{reformulationic}, which in particular implies the stability, is stated as follows.
\begin{thm}\label{thic}
Assume $\abs{\bar B_3}> \mathcal{M}_c$. Let $N\ge 4$ be an integer. Assume that $u_0\in H^{4N}(\Omega)$ and $\eta_0\in H^{4N+1}(\Omega)$ satisfy the necessary compatibility conditions of \eqref{reformulationic} and that $\eta_0$ satisfies
\begin{equation}\label{eta00}
 {\rm det} (I+\nabla\eta_0)=1  \text{ in }\Omega,\
\eta_0=0 \text{ on }\Sigma_{m,\ell}\text{ and }
\jump{\eta_0}=0 \text{ on }\Sigma.
\end{equation}
There exists a universal constant $\varepsilon_0>0$ such that if $\se{2N} (0) + \f(0) \le \varepsilon_0$, then there exists a global unique solution $(\eta,u,p)$ solving  \eqref{reformulationic} on $[0,\infty)$. Moreover, there exists a universal constant $C>0$ such that
\begin{equation}
\mathcal{G}_{2N}(\infty) \le C( \se{2N} (0) + \f(0)).
\end{equation}
\end{thm}

We now state our instability result for the system \eqref{reformulationic} when $\abs{\bar B_3}< \mathcal{M}_c$.
\begin{thm}\label{maintheorem}
Assume $\abs{\bar B_3}< \mathcal{M}_c$. Let $N\ge 4$ be an integer. There exist universal constants $\theta_0>0$  and $C>0$ such that for any sufficiently small $0< \iota<\theta_0$ there exist solutions $(\eta^\iota, u^\iota, p^\iota  )$ to \eqref{reformulationic} such that
 \begin{equation}
(\se{2N}  + \f) ( \eta^\iota  , u^\iota  , p^\iota)(0)  \le C\iota,\hbox{ but }  \abs{ \eta_3^\iota(T^\iota)}_{0}\ge \frac{\theta_0}{2}.
 \end{equation}
Here the escape time $T^\iota>0$ is
 \begin{equation}\label{escape_time}
T^\iota:=\frac{1}{\Lam}\log\frac{\theta_0}{\iota},
\end{equation}
 where $\Lam$ is the sharp linear growth rate.
\end{thm}

\begin{remark}
Theorems \ref{thic} and \ref{maintheorem} establish the sharp nonlinear stability criteria for the equilibrium in the reformulated problem \eqref{reformulationic}. Note that in Theorem \ref{thic} the bound of $\mathcal{G}_{2N}(\infty)$ implies that $\se{N+2} (t) \ls (1+t)^{-2N+4} $; since $N$ may be taken to be arbitrarily large, this decay result can be regarded as an ``almost exponential" decay rate. Theorem \ref{maintheorem} shows that the onset of the instability occurs in $\eta_3$ at the internal interface. Note that our results do not cover the critical case: $\abs{\bar B_3}=\mc$; we only know that the problem is locally well-posed, but it is not clear to us what the stability of the system should be.
\end{remark}

\begin{remark}
With the solution $(\eta,u,p)$ to \eqref{reformulationic} in hand, by shifting $Id+\eta\rightarrow\eta$ and then defining $B:=\bar B\cdot\nabla\eta$, we have that $(\eta,u,p, B)$ solve the original problem \eqref{lagrangian}. Hence, Theorem \ref{thic} produces a global-in-time, decaying solution to \eqref{lagrangian} for the initial data near the equilibrium that satisfies \eqref{conserv} initially; Theorem \ref{maintheorem} implies the nonlinear instability of the equilibrium for \eqref{lagrangian}. By further changing coordinates back to $y\in\Omega_\pm(t)$, we conclude the nonlinear stability and instability of the incompressible viscous non-resistive MHD internal wave problem \eqref{ns_euler0}.
\end{remark}

\begin{remark}
One of crucial points in our analysis is the finite depth of the two fluids. Formally, if $\ell=m=\infty$, then the critical value $\mc$, defined by \eqref{mc}, is infinite; this would suggest that the incompressible viscous non-resistive MHD internal wave problem in the whole space is unstable for any $\bar{B}$. Indeed, the linear growing normal mode solution was constructed in \cite{3C}.
\end{remark}

Since within a local time interval the $\eta$ terms can be easily controlled by the viscosity term, the local well-posedness of the system \eqref{reformulationic} in our functional framework can be established similarly as \cite{WTK} for the incompressible viscous surface-internal wave problem, which is motivated by \cite{GT_lwp} for the incompressible viscous surface wave problem. So we may omit the proof and refer to \cite{WTK,GT_lwp} for the construction of local solutions. Therefore, the main part of proving Theorems \ref{thic} and \ref{maintheorem} is to derive the a priori estimates. Our basic ingredient of showing the stability and instability of \eqref{reformulationic} is the natural energy identity:
\begin{equation} \label{identity000}
 \hal  \frac{d}{dt} \left(\int_\Omega \left( \rho\abs{  u}^2+\abs{(\bar B\cdot\nabla) \eta}^2  \right)- \int_{\Sigma} \rj g \abs{ \eta_3}^2 \right)
+  \int_\Omega  \frac{\mu}{2} \abs{ \sg u}^2=h.o.t..
\end{equation}
Here $h.o.t.$ denotes the higher order terms. Since $\rj>0$, the energy may not be positively definite, which is the cause of the instability. Our observation is that when $\bar B_3\neq 0$, there is a competition between the magnetic energy and the gravity potential energy. This leads to our definition \eqref{mc} of the critical value $\mc$. Indeed, by a variational argument, we will show that when $\abs{\bar B_3}\ge \mc$ the energy is non-negative and that when $\abs{\bar B_3}<\mc$ we can construct functions so that the energy is negative. This should suggest, at least, the linear stability and instability. Indeed, if $\abs{\bar B_3}<\mc$, then we will construct growing mode solutions, which grow as $e^{\lambda t}$ with $\lambda>0$, to the linearized system of \eqref{reformulationic} (cf. Theorem \ref{growingmode}). It is known that the construction of growing mode solutions is reduced to the solvability of an eigenvalue problem with the eigenvalue $\lambda$. However, the viscosity destroys the variational structure of the reduced eigenvalue problem; we will resort to the framework developed by \cite{3GT2} for the compressible viscous internal wave problem to restore the ability to use variational methods. These analysis of the linear stability and instability will be carried out in Section \ref{linear theory}.

In Section \ref{nonlinear stability}, we will prove the nonlinear stability as stated in Theorem \ref{thic} when $\abs{\bar B_3}>\mc$. The basic strategy in the energy method is to use first the energy-dissipation structure of \eqref{reformulationic}. Note that besides \eqref{identity000} there is another interactive energy-dissipation structure resulting from testing \eqref{reformulationic} by $\eta$ (guaranteed by the boundary conditions of $\eta$):
 \begin{equation}\label{intt}
\int_\Omega \rho \dt      u \cdot  \eta
  +\frac{1}{2}\frac{d}{dt} \int_\Omega \frac{\mu}{2}   \abs{\sg   \eta}^2  +\int_{\Omega}  \abs{(\bar B\cdot \nabla)  \eta }^2- \int_{\Sigma} \rj g \abs{  \eta_3}^2
= h.o.t..
\end{equation}
The novelty of \eqref{intt} is the recovery of the dissipation estimates of $\eta$ from the magnetic effect (and gravity) and the energy estimates of $\eta$ from the viscous effect. Note that by Poincar\'e's inequality in $\Omega$, which does not hold for the whole space, the first term in \eqref{intt} can be handled by the integration by parts in time. As we need to work with the higher order energy functionals to control the nonlinear terms, we apply derivatives to \eqref{reformulationic} and then use the energy-dissipation structures \eqref{identity000} and \eqref{intt} to get the energy evolution estimates of $(\eta,u)$ as well its temporal and horizontal spatial derivatives that preserve the boundary conditions. Note that it is crucial to employ the structure of the nonlinear terms of $\diverge u$ and $\diverge \eta$ (guaranteed by \eqref{imp1}) since we can not get any estimates of the pressure $p$ without spatial derivatives. The main conclusion of the energy evolution is that
\begin{equation}\label{energyevolution}
\dtt \seb{n}+\sdb{n}\le \mathcal{N}_n,
\end{equation}
where $\seb{n}$ and $\sdb{n}$ represent the ``horizontal" counterparts of $\se{n}$ and $\sd{n}$, respectively, and $\mathcal{N}_n$ represents the nonlinear estimates.

The next step is to use the structure of the equations to improve the energy evolution estimates \eqref{energyevolution}. Due to the presence of the pressure $p$, the only way to improve the estimates is to use the elliptic regularity theory of the Stokes system. However, due to the presence of $\eta$ terms, the procedure is much more delicate than that of the incompressible Navier-Stokes equations in a fixed domain \cite{L,T}. First, as we will see, we can not complete the improvement of energy estimates untill we have completed the improvement of dissipation estimates. Second, to improve the dissipation estimates, by $\dt\eta=u$, we may use the two-phase Stokes regularity for $(u,p)$ by regarding $\eta$ terms as forcing terms to estimate the time derivatives of $(u,p)$ in $\sd{n}$ by $\norm{u}_{2n-1}^2+\sdb{n}+\mathcal{N}_n$. Note that the term $\norm{u}_{2n-1}^2$ can be absorbed by the Sobolev interpolation since we will control $\norm{u}_{2n}^2$ in $\sd{n}$. However, we can not estimate $(u,p)$ without time derivatives in the same way since $\sdb{n}$ only controls $\norm{\eta}_{1,2n-1}^2$ (since $\bar B_3\neq0$), which is not regular enough to control the $\eta$ terms. We have two observations to get around this obstacle. The first observation is that we may write $(\bar B\cdot \nabla)^2\eta=\bar B_3^2\Delta\eta-\bar B_3^2\Delta_\ast\eta+(\bar B_\ast\cdot\nabla_\ast)^2\eta+2  (\bar B_3\pa_3) (\bar B_\ast\cdot\nabla_\ast)\eta$ and we have certain control of last three terms in $\sdb{n}$; this motivates us to consider the Stokes system for $(w,p)$ with introducing the quantity $w=  u+   \frac{ \bar B_3^2}{\mu}  \eta$. However, we can not use the two-phase Stokes regularity as before since the two viscosities $\mu_\pm$ generally are different and the difference would prevent us from obtaining a ``good" jump boundary condition for $w$. Our second observation is that $\bar{\mathcal{D}}_{n}$ has the certain control of $w$ on the boundary $\Sigma$ due to the flatness of $\Sigma$. The idea is then to apply the one-phase Stokes regularity with  Dirichlet boundary conditions to the domains $\Omega_\pm$ respectively, interwinding between vertical derivatives of $\eta$ and horizontal derivatives, to deduce the estimates of $\norm{w}_{2n}^2$ bounded by $(\norm{\eta}_{1,2n-1}^2+)\norm{\dt u}_{2n-2}^2+\sdb{n}+\mathcal{N}_n$. The key point is that, since $\dt\eta=u$, $\norm{w}_{2n}^2\simeq \frac{d}{dt} \norm{\eta}_{2n}^2+\norm{\eta}_{2n}^2+\norm{u}_{2n}^2$; this yields not only the dissipation estimates of $(\eta,u,p)$ but also the energy estimates of $\eta$. Finally, note that it is just the energy estimates of $\eta$ that allows us to employ the two-phase Stokes regularity for $(u,p)$ to deduce the energy estimates of $(u,p)$. The conclusion of the improved estimates is that
\begin{equation}
\dtt \se{n}+\sd{n}\le \mathcal{N}_n.
\end{equation}

Now the remaining is to estimate $\mathcal{N}_n$, and the basic goal is $\mathcal{N}_n\ls \sqrt{\se{n}}\sd{n}$; this would then close the estimates in a small-energy regime. Unfortunately, this is not true; we need to resort to $\mathcal{F}_n$ and $\mathcal{J}_n$ to control some troubling terms. The control of $\mathcal{F}_n$ and $\mathcal{J}_{n} $ is through the following, by estimating $\norm{w}_{2n+1}^2$,
\begin{equation}\label{ine}
 \dtt {\mathcal{F}}_{n}
+  {\mathcal{F}}_{n} +\mathcal{J}_{n}  \ls    \norm{\eta}_{1,2n}^2+ \sd{n}+\mathcal{N}_n.
\end{equation}
Note that $\norm{\eta}_{1,2n}^2$ can only be controlled by $\se{n}$ (indeed, $\seb{n}$) but not by $\sd{n}$, and hence ${\mathcal{F}}_{n} $ and $\mathcal{J}_{n} $ are not included in the dissipation. This would be harmful for the energy method. Our solution to this problem is to implement the two-tier energy method \cite{GT_per,GT_inf}.  The idea is to employ two tiers of energies and dissipations, $\se{N+2}$, $\sd{N+2}$, $\se{2N}$, and $\sd{2N}$. We then control the troubling terms in $\mathcal{N}_n$ by $\sqrt{\se{N+2}}(\f+\mathcal{J}_{2N})$ when $n=2N$ and by $\sqrt{\se{2N}}\sd{N+2}$ when $n=N+2$. This leads to
\begin{equation}\label{conclu2}
 \frac{d}{dt} \se{2N} + \sd{2N} \ls  \sqrt{\se{N+2}}(\f+\mathcal{J}_{2N})
\end{equation}
and
\begin{equation}\label{conclu23}
\frac{d}{dt} \se{N+2} + \sd{N+2} \le 0.
\end{equation}
To control the right hand side of \eqref{conclu2}, a time weighted analysis on \eqref{ine} with $n=2N$ leads to the boundedness of $\f$ and $
  \int_0^t \frac{\mathcal{J}_{2N}}{(1+r)^{1+\vartheta}}dr $ for any $\vartheta>0.$ Hence, if $\se{N+2}$ decays at a sufficiently fast rate, then the estimates \eqref{conclu2} close. This can be achieved by using \eqref{conclu23};
although we do not have that $  \se{N+2}\ls \sd{N+2}$, which rules out the exponential decay, we can use an interpolation argument as \cite{RG,GT_per,GT_inf} to bound $\se{N+2} \ls (\se{2N})^{1-\theta} (\sd{N+2})^{\theta}$ for $\theta=(2N-4)/(2N-3)$.  Plugging this in \eqref{conclu23} leads to an algebraic decay estimate for $\se{N+2}$ with the rate $(1+t)^{-2N+4}$. Consequently, this scheme of the a priori estimates closes by choosing $0<\vartheta\le N-3$ for $N\ge 4$, and hence the proof of Theorem \ref{thic} is completed.

In Section \ref{nonlinear instability}, we will prove the nonlinear instability as stated in Theorem \ref{maintheorem} when $\abs{\bar B_3}<\mc$. Since linear instability has been established in Section \ref{linear theory}, the heart of the proof is then to derive the energy estimates, which allows us to employ the bootstrap argument developed by Guo and Strauss \cite{GS} to passage from linear instability to nonlinear instability.
The natural way of showing nonlinear instability is to consider the difference between the solution to the nonlinear problem and the growing mode solution to the linear problem; if the difference is relatively small with respect to the linear growing mode solution (in a small-energy regime), then the solution to the nonlinear problem behaviors as the linear growing mode solution, and hence nonlinear instability follows. In order to estimate the difference, in spirit of Duhamel's principle, we need to estimate the growth of the solution operator for the linearized problem. Because the spectrum of the linear solution operator is complicated, we can only derive the largest growth rate, indeed $\lambda$, for the linearized problem by using careful energy estimates and the variational character of $\lambda$; this is in the context of strong solutions,
which requires the initial data to satisfy the linear compatibility conditions. Such
estimates would not be applicable to the nonlinear problem by directly employing
Duhamel's principle. To get around this issue, we provide the estimates for the
growth in time of arbitrary solutions of the linear inhomogeneous equations (cf. Theorem \ref{lineargrownth}); clearly, the estimates can be applied directly to the nonlinear problem.

When applying Theorem \ref{lineargrownth} to the nonlinear problem, although we estimate the difference in the lower-order regularity norm, it is typical that the control of the nonlinear terms requires the control of the higher-order regularity norm of the solution to the nonlinear problem due to that the unboundedness of the nonlinear part usually yields a loss in derivatives. In order to close the analysis, motivated by Guo-Strauss's approach \cite{GS}, we will derive the energy estimates for the nonlinear problem which shows that on the time scale of the instability the higher-order regularity norm of the solution is actually bounded by the growth of low-order regularity norm. The energy estimates will be recorded in Theorem \ref{engver} for the case $\bar B_3\neq 0$ and Theorem \ref{enghor} for the case $\bar B_3= 0$. We employ a variant of the strategy of the stable case $\abs{\bar B_3}>\mc$. We first derive the modified energy evolution estimates by shifting the negative gravity potential energy onto the right hand side of the estimates, which would imply that the growth of ``horizontal" energy $\seb{2N}$ is bounded by the growth of $\abs{\eta_3}_0^2$. Next, to improve the estimates, we need to employ different arguments for the cases $\bar B_3\neq0$ and $\bar B_3=0$. For the case $\bar B_3\neq 0$, we can mostly follow that of the stable case to conclude that the growth of  $\se{2N}+\f$ is bounded by the growth of $\seb{2N}$ and hence by $\abs{\eta_3}_0^2$. For the case $\bar B_3= 0$, the situation is much more delicate; we can not improve the estimates of the solution without time derivatives by using the Stokes system for $(w,p)$. We will make the interplay  between the control of $u$ through $\eta$ by using the Stokes system for $(u,p)$ and the control of $\eta$ through $u$ by $\dt \eta=u$ and interwind between vertical derivatives of $\eta$ and horizontal derivatives; the conclusion is also that the growth of  $\se{2N}+\f$ is bounded by the growth of $\abs{\eta_3}_0^2$.

Lastly, we need to employ an argument from \cite{JT} that uses the linear growing mode to construct initial data for the nonlinear problem so that the compatibility conditions are satisfied, which are required for the local well-posedness of the nonlinear problem. Hence, starting from this initial data that is close to the linear growing mode, employing the linear growth estimates and the bootstrap nonlinear energy estimates, we can complete the proof of Theorem \ref{maintheorem}.

At the end of the paper we present Appendix A, where we record various analytic
tools that are useful throughout the paper.

\medskip

\hspace{-13pt}{\bf Notation.}
We now set the conventions for our notation. The Einstein convention of summing over repeated indices is used. Throughout the paper $C>0$ will denote a generic constant that does not depend on the initial data and time, but can depend on $N$, $\Omega_\pm$, the steady states, or any of the parameters of the problem (e.g. $g$, $\rho_\pm$, $\mu_\pm$).  We refer to such constants as ``universal."   We employ the notation $A \ls B$ to mean that $A \le C B$ for a universal constant $C>0$, and similarly, $A \gss B$ and $A \simeq B$. We will also write $\dt A+B\ls D$ for $\dt A+C^{-1}B\le CD$. Universal constants are allowed to change from line to line. When a constant depends on a quantity $z$ we will write $C  = C_z$ to indicate this. To indicate some constants in some places so that they can be referred to later, we will denote them in particular by $C_1,C_2$, etc.

We will write $\mathbb{N} = \{ 0,1,2,\dotsc\}$ for the collection of non-negative integers.  When using space-time differential multi-indices, we will write $\mathbb{N}^{1+m} = \{ \alpha = (\alpha_0,\alpha_1,\dotsc,\alpha_m) \}$ to emphasize that the $0-$index term is related to temporal derivatives.  For just spatial derivatives we write $\mathbb{N}^m$.  For $\alpha \in \mathbb{N}^{1+m}$ we write $\partial^\alpha = \dt^{\alpha_0} \p_1^{\alpha_1}\cdots \p_m^{\alpha_m}.$ We define the parabolic counting of such multi-indices by writing $\abs{\alpha} = 2 \alpha_0 + \alpha_1 + \cdots + \alpha_m.$  We will write $\nabla_\ast$ for the horizontal gradient and $\Delta_\ast$ for the horizontal Laplace operator. For a vector $v=(v_1,v_2,v_3)$, we write $v_\ast=(v_1,v_2)$ for the horizontal component. Finally, for a \textit{given norm} $\norm{\cdot}$ and an integer $k\ge 0$, we introduce the following notation for sums of derivatives:
\begin{equation} \nonumber
\norm{\bar{\nab}^k_0 f}^2 := \sum_{\substack{\alpha \in \mathbb{N}^{1+3}, \abs{\alpha}\le k} } \norm{\pa^\al  f}^2 \text{ and }
\norm{\bar{\nab}^{\ k}_{\ast 0} f}^2 := \sum_{\substack{\alpha \in \mathbb{N}^{1+2}, \abs{\alpha}\le k} } \norm{\pa^\al  f}^2.
\end{equation}

\section{Linear theory}\label{linear theory}

In this section, we consider the stability and instability for the linearization of \eqref{reformulationic}:
\begin{equation}\label{perturb_linear}
\begin{cases}
 \dt\eta=u&\hbox{in }\Omega
\\\rho\partial_t u-\mu\Delta u+\nabla p- (\bar B\cdot\nabla)^2 \eta=0\quad&\hbox{in }\Omega
\\ \diverge u=0&\hbox{in }\Omega
\\ \Lbrack u\Rbrack=0,\quad \Lbrack pI-\mu\mathbb{D}u\Rbrack e_3-\jump{\bar B_3(\bar B\cdot\nabla)\eta}= \rj g\eta_3  e_3&\hbox{on }\Sigma
\\ u  =0 &\hbox{on }\Sigma_{m,\ell}.
\end{cases}
\end{equation}

As standard in the stability theory \cite{3C}, we will study the growing mode solutions to \eqref{perturb_linear} of the following form:
\begin{equation}\label{ansatz}
\eta(t,x)=w(x){\rm e}^{\lam t},\
u(t,x)=v(x){\rm e}^{\lam t},\ p(t,x)=q(x){\rm e}^{\lam t},
\end{equation}
for some $\lam>0$ (the same in the upper and lower fluids). Substituting the ansatz \eqref{ansatz} into \eqref{perturb_linear}, we find that
\begin{equation}\label{ansatz1}
v=\lam w.
\end{equation}
By using \eqref{ansatz1}, we can eliminate $v$ and arrive at the following time-invariant system for  $(w,q)$:
\begin{equation}\label{perturb_linear steady}
\begin{cases}
\lam^2\rho w-\lam \mu\Delta w+\nabla q- (\bar B\cdot\nabla)^2 w=0\quad&\hbox{in }\Omega
\\ \diverge w=0&\hbox{in }\Omega
\\ \Lbrack w\Rbrack=0,\quad \Lbrack qI-\lam\mu\mathbb{D}w\Rbrack e_3-\jump{\bar B_3(\bar B\cdot\nabla)w}= \rj gw_3  e_3&\hbox{on }\Sigma
\\ w =0 &\hbox{on }\Sigma_{m,\ell}.
\end{cases}
\end{equation}

It is not trivial at all to construct solutions to \eqref{perturb_linear steady} by utilizing variational methods
since $\lam$ appears both linearly and quadratically. In order to circumvent this problem
and restore the ability to use variational methods, as \cite{3GT2} we artificially remove the linear
dependence on $\lam$ in \eqref{perturb_linear steady} by introducing an arbitrary parameter $s>0$. This results
in a family $(s>0)$ of modified problems:
\begin{equation}\label{perturb_linear steady s}
\begin{cases}
\lam^2\rho w- s \mu\Delta w+\nabla q- (\bar B\cdot\nabla)^2 w=0\quad&\hbox{in }\Omega
\\ \diverge w=0&\hbox{in }\Omega
\\ \Lbrack w\Rbrack=0,\quad \Lbrack qI-s\mu\mathbb{D}w\Rbrack e_3-\jump{\bar B_3(\bar B\cdot\nabla)w}= \rj gw_3  e_3&\hbox{on }\Sigma
\\ w =0 &\hbox{on }\Sigma_{m,\ell}.
\end{cases}
\end{equation}

A solution to the modified problem \eqref{perturb_linear steady s} with $\lam = s$ corresponds to a solution to the original problem \eqref{perturb_linear steady}.
Note that for any fixed $s>0$, \eqref{perturb_linear steady s} can be viewed as an eigenvalue problem for $-\lam^2$, which has a natural variational structure that allows us to use variational methods to construct solutions. In order to understand $\lam$ in a variational framework, we consider the energy functional
\begin{equation}\label{E_def}
E(w;s) := E_0(w) + s E_1(w)
\end{equation}
with
\begin{equation}\label{E0_def}
E_0(w) :=\hal\left(\int_\Omega \abs{(\bar B\cdot \nabla )w}^2- \int_\Sigma  \rj g w_3^2\right)
\end{equation}
and
\begin{equation}\label{E1_def}
E_1(w):=\hal \int_\Omega \frac{\mu}{2}\abs{\sg w}^2,
\end{equation}
which are all well-defined on the space $w\in H_0^1(\Omega)$. Here $\H$ denotes for the usual Sobolev space on $\Omega$. We also introduce $\Hs=\{v\in\H\,\vert\,\diverge w=0\}$.
Consider the admissible set
\begin{equation}
 \mathfrak{S} = \{w\in \Hs\;\vert\;   J(w):=\hal \int_\Omega  \rho  w^2=1  \}.
\end{equation}
Notice that $E_0(w)$ is not positively definite for $\jump{\rho} >0$.

The first proposition asserts the existence of the minimizer of $E$ in \eqref{E_def} over $\mathfrak{S}$ and hence the solvability of the problem \eqref{perturb_linear steady s}.
\begin{prop}\label{propro1} Let $s>0$ be fixed. Then the following hold:
\begin{enumerate}
\item $E$ achieves its infimum over $\mathfrak{S}$.
\item Let $w$ be a minimizer and $q$ be the associated Lagrange multiplier. Then $w\in H^k(\Omega)$ and $\nabla q\in H^{k-2}(\Omega)$ (and $\jump{q}\in H^{k-3/2}(\Sigma)$) for any $k\ge 2$ and solve the problem \eqref{perturb_linear steady s} with $\lam^2$ given by
\begin{equation}\label{mu_def}
-\lam^2=\alpha(s) := \inf_{w\in \mathfrak{S}}E(w;s).
\end{equation}
\end{enumerate}
\end{prop}
\begin{proof}
By the trace estimates \eqref{tra2} of Remark \ref{traceth} and Korn's inequality \eqref{korneq} of Lemma \ref{korn}, we have that for $w\in \mathfrak{S}$,
\begin{equation}\label{lower bound}
\int_\Sigma     w_3^2 \ls \norm{w}_0 \norm{w}_1 \ls \sqrt{J(w)}\sqrt{E_1(w)}=\sqrt{E_1(w)}.
\end{equation}
Then we see that, by Cauchy's inequality,
\begin{equation}\label{11}
E(w;s)\ge    s E_1(w)- \int_\Sigma  \rj g w_3^2  \ge  s E_1(w)-C\sqrt{E_1(w)}\gss -s^{-1}.
\end{equation}
This shows that $E$ is bounded below on $\mathfrak{S}$. Then the existence of the minimizer of $E$ in \eqref{E_def} follow from  standard compactness arguments, which shows the assertion $(1)$.

To prove $(2)$, we let $w$ be a minimizer, then the variational principle for Euler-Langrange equations shows that
\begin{equation}\label{cS0d1}
 \lam^2\int_\Omega \rho w\cdot \varphi + s \int_\Omega \frac{\mu}{2} \mathbb{D} w: \mathbb{D}\varphi+\int_\Omega (\bar B\cdot\nabla)  w\cdot(\bar B\cdot\nabla)  \varphi=0\hbox{ for all }\varphi\in\Hs.
\end{equation}
We then introduce the pressure $q$ as a Lagrange multiplier. For this, we define $\Lambda \in \Hsd$ so that $\Lambda (\varphi)$ equals the left hand side of \eqref{cS0d1}. Then $\Lambda=0$ on $\Hs$, and hence according to Lemma \ref{Pressure} there exists a unique $q \in L_{loc}^2(\Omega)$ (up to constants) so that $(q ,\diverge{\varphi})=\Lambda(\varphi)$ for all $\varphi\in \H$. This implies that $(w,q)$ is a weak solution to the problem \eqref{perturb_linear steady s} with $\lam^2$ given by \eqref{mu_def}.
The regularity of $(w,q)$ follows by a similar argument in the proof of energy estimates of our main theorems. We may only sketch the strategy here. First, we will use the weak formulation and employ the difference quotient to control the $H^1$-regularity of horizontal derivatives of $w$. By using the trace estimates and the flatness of $\Sigma$, we get also the higher regularity of $w$ on $\Sigma$. Then we will use the classical regularity for the Stokes problem with Dirichlet boundary conditions to get the desired regularity of $w$ and $q$ on $\Omega_\pm$, respectively. Finally, the second jump condition on $\Sigma$ follows by taking test functions which is compactly supported near $\Sigma$. The boundary estimates of $\jump{q}$ also follows.
\end{proof}
\begin{remark}
If we recover the dependence of constants $C$ on $\mu$ and $\bar B$, \eqref{11} shows the upper bound $\lam^2\ls ((\mu_++\mu_-)s)^{-1} $, which is independent of $\bar B$. If $\bar B_3\neq 0$, we can also show an upper bound of $\lam^2$ independent of $\mu$. Indeed, if $\bar B_3\neq 0$,  by the trace estimates \eqref{tra2b} of Lemma \ref{tracethb}, we have that for $w\in \mathfrak{S}$,
\begin{equation}\label{lower bound22}
\int_\Sigma     w_3^2 \ls \frac{1}{\abs{\bar B_3}}\norm{(\bar B\cdot\nabla)w}_0\norm{w}_0 \ls \frac{1}{\abs{\bar B_3}}\sqrt{E_0(w)}\sqrt{J(w)}=\frac{1}{\abs{\bar B_3}}\sqrt{E_0(w)},
\end{equation}
which implies
\begin{equation}\label{1122}
E(w;s)\ge     E_0(w)- \int_\Sigma  \rj g w_3^2  \ge    E_0(w)-\frac{C}{\abs{\bar B_3}}\sqrt{E_0(w)}\gss - \bar B_3^{-2}.
\end{equation}
This yields that $\lam^2\ls  \bar B_3^{-2} $, which is independent of $\mu$. These two imply that in the presence of the viscosity or the non-horizontal magnetic field, the growing modes, if it could be constructed, cannot grow arbitrarily fast. This suggests that the presence of the viscosity or the non-horizontal magnetic field may prevent the ill-posedness for the Euler Rayleigh-Taylor problem (see \cite{3E,3GT1}).
\end{remark}

In order to construct the growing mode solution to the original problem \eqref{perturb_linear} we first need to clarify the sign of the infimum \eqref{mu_def}; if $E(w;s)$ is always non-negative, then no growing mode solutions to \eqref{perturb_linear} can be constructed, which may suggest that the system \eqref{perturb_linear} is linearly stable. The possibility of negativity of $E(w;s)$ is essentially that of $E_0(w)$ since $s>0$ can be chosen arbitrarily small. To clarify the sign of $E_0(w)$, we first show the following variational problem.

\begin{lem}\label{criticalb}
For any constant vector $\bar B \in \mathbb{R}^3$, it holds that
\begin{equation}\label{va}
\inf_{\substack{f\in H^1_0(\Omega)\\  \int_\Sigma   f^2 =1}} \int_\Omega \abs{(\bar B\cdot \nabla )f}^2 = {\bar B_3^2\left(\dis\frac{1}{\ell}+\frac{1}{m}\right)}.
\end{equation}
\end{lem}
\begin{proof}
First, for the horizontal $\bar B$, that is, the case $\bar B_3=0$, by rotation, it suffices to prove \eqref{va} for the case $\bar B=(\bar B_1,0,0)$. We take a sequence of test functions $\varphi_k(x)=\frac{1}{k}\varphi ( \frac{x_1}{k},x_2,x_3  )$ for $\varphi\in C_0^\infty(\Omega)$ with $\int_\Sigma   \varphi^2 =1$. Then $\varphi_k\in C_0^\infty(\Omega)\subset H^1_0(\Omega)$ and
\begin{equation}
\int_\Sigma   \varphi_k^2 =1 \text{ and } \int_\Omega \abs{(\bar B\cdot \nabla )\varphi_k}^2= \frac{\bar B_1^2}{k^2}\int_\Sigma   \varphi_k^2=\frac{\bar B_1^2}{k^2}.
\end{equation}
Letting $k\rightarrow 0$ proves \eqref{va} for the case $\bar B_3= 0$.

Now we assume $\bar B_3\neq 0$. For any $f\in H^1_0(\Omega)$, taking $x_3=0$ in \eqref{pbbb}, we obtain
\begin{equation}
\int_{\Sigma}  f^2=\int_{\mathbb{R}^2}  f(x_\ast,0)^2  \le  \frac{\ell}{\bar B_3^2} \int_{\Omega_+} \abs{(\bar B\cdot\nabla)f }^2,
\end{equation}
and also
\begin{equation}
\int_{\Sigma}  f^2  \le  \frac{m}{\bar B_3^2} \int_{\Omega_-} \abs{(\bar B\cdot\nabla)f }^2.
\end{equation}
We then have
\begin{equation}\label{infbound}
 \int_\Omega \abs{(\bar B\cdot \nabla )f}^2 \ge  {\bar B_3^2\left(\frac{1}{\ell}+\frac{1}{m}\right)}\int_\Sigma   f ^2 \text{ for any }f\in H^1_0(\Omega).
\end{equation}
On the other hand, we take a sequence of test functions $f_k(x )=\varphi_k( x_\ast )\psi(x_3) $ with $\varphi_k( x_\ast)=\frac{1}{k}\varphi ( \frac{x_\ast}{k}  )$ for $\varphi\in C_0^\infty(\mathbb{R}^2)$ and $\psi(x_3)$ defined by
\begin{equation}\label{pside}
\psi(x_3)=\begin{cases}
1-\frac{x_3}{\ell},\quad x_3\in [0,\ell]
\\ 1+\frac{x_3}{m},\quad x_3\in [-m,0).
\end{cases}
\end{equation}
Then $f_k \in H^1_0(\Omega)$ and
\begin{equation}
\begin{split}
&\int_\Omega \abs{(\bar B\cdot \nabla )f_k}^2=\int_\Omega \abs{(\bar B_\ast\cdot \nabla_\ast )f_k}^2+
\bar B_3^2\abs{ \pa_3f_k}^2+2(\bar B_\ast\cdot \nabla_\ast )f_k \bar B_3\pa_3f_k
\\&\quad=\int_{\mathbb{R}^2} \abs{(\bar B_\ast\cdot \nabla_\ast )\varphi_k}^2\int_{-m}^\ell \psi^2+
\int_{\mathbb{R}^2} \bar B_3^2\abs{ \varphi_k}^2\int_{-m}^\ell \psi'^2
+\hal\int_{\mathbb{R}^2} (\bar B_\ast\cdot \nabla_\ast )\varphi_k^2 \bar B_3 \int_{-m}^\ell(\psi^2)'
\\&\quad=\frac{1}{k^2}\int_{\mathbb{R}^2} \abs{(\bar B_\ast\cdot \nabla_\ast )\varphi}^2\int_{-m}^\ell \psi^2+
  \bar B_3^2\int_{\mathbb{R}^2}\abs{ \varphi}^2\int_{-m}^\ell \psi'^2.
\end{split}
\end{equation}
Note that
\begin{equation}
 \int_\Sigma   f_k ^2= \int_{\mathbb{R}^2}\abs{ \varphi}^2\text{ and }\int_{-m}^\ell \psi'^2 =\frac{1}{\ell}+\frac{1}{m}.
\end{equation}
This implies that
\begin{equation}\label{limitvalue}
\lim_{k\rightarrow\infty}\frac{\dis\int_\Omega \abs{(\bar B\cdot \nabla )f_k}^2}{\dis\int_\Sigma   f_k ^2 }={\bar B_3^2\left(\frac{1}{\ell}+\frac{1}{m}\right)}.
\end{equation}
Combining \eqref{infbound} and \eqref{limitvalue} shows \eqref{va}.
\end{proof}

Now we can clarify the sign of $E_0(w)$.
\begin{lem}\label{signsign}
For any constant vector $\bar B \in \mathbb{R}^3$, we have the following assertions:
\begin{enumerate}
\item
If $\abs{\bar B_3}\ge  \mc$, then for any $w\in H_0^1(\Omega)$, $E_0(w)\ge 0$. Moreover,
\begin{equation}
  \int_\Omega  \abs{(\bar B\cdot \nabla) w}^2- \int_{\Sigma} \rj g \abs{  w_3}^2 \gss \left(1-  \frac{\mc^2}{\abs{\bar B_3}^2}\right)\int_\Omega  \abs{(\bar B\cdot \nabla) w}^2 .
\end{equation}

\item If $\abs{\bar B_3}<\mc$, then there exists $w\in \Hs$ such that $E_0(w)< 0$.
\end{enumerate}
\end{lem}
\begin{proof}
The assertion $(1)$ follows by applying Lemma \ref{criticalb} to $f=w_3$ and noting the definition \eqref{mc} of $\mc$.

For the assertion $(2)$, by rotation, it suffices to prove it for the case $\bar B=(\bar B_1,0,\bar B_3)$. We then take functions $w_{k,n}=\left(0, \varphi(\frac{x_1}{k}) \phi_n(\frac{x_2}{k})\psi^\prime(x_3), \varphi(\frac{x_1}{k})\frac{1}{k}\phi_n^\prime(\frac{x_2}{k})\psi(x_3) \right)$ with $\varphi\in C_0^\infty(\mathbb{R})$, $\phi_n(z)= n^{-1/4}e^{-n  z^2}$ and $\psi\in H_0^1((-m,\ell))$ defined by \eqref{pside}. Then we have $w_{k,n}\in \Hs$ and
\begin{equation}
\begin{split}
&\int_\Omega \abs{(\bar B\cdot \nabla )w_{k,n}}^2 = \bar B_1 ^2\int_\Omega \abs{\pa_1 w_{k,n}}^2 +
  \bar B_3^2\int_\Omega \abs{\pa_3 w_{k,n}}^2
  \\& \quad=
  \frac{\bar B_1 ^2}{k^2}  \int_{\mathbb{R} } \abs{ \varphi^\prime\left(\frac{x_1}{k}\right)}^2 \int_{\mathbb{R} } \abs{\phi_n\left(\frac{x_2}{k}\right)}^2
    \int_{-m}^\ell \psi'^2+
  \bar B_3^2   \int_{\mathbb{R} } \abs{ \varphi\left(\frac{x_1}{k}\right)}^2 \int_{\mathbb{R} } \abs{\phi_n\left(\frac{x_2}{k}\right)}^2 \int_{-m}^\ell \psi''^2
  \\&\qquad+ \frac{\bar B_1 ^2}{k^4}  \int_{\mathbb{R} } \abs{ \varphi^\prime\left(\frac{x_1}{k}\right)}^2 \int_{\mathbb{R} } \abs{\phi_n^\prime\left(\frac{x_2}{k}\right)}^2\int_{-m}^\ell \psi^2
  +
  \frac{\bar B_3^2}{k^2}\int_{\mathbb{R} } \abs{ \varphi\left(\frac{x_1}{k}\right)}^2 \int_{\mathbb{R} } \abs{\phi_n^\prime\left(\frac{x_2}{k}\right)}^2\int_{-m}^\ell \psi'^2
    \\& \quad=
  {\bar B_1 ^2}   \int_{\mathbb{R} } \abs{ \varphi^\prime }^2 \int_{\mathbb{R} } \abs{\phi_n }^2
    \int_{-m}^\ell \psi'^2+
  k^2\bar B_3^2   \int_{\mathbb{R} } \abs{ \varphi }^2 \int_{\mathbb{R} } \abs{\phi_n }^2 \int_{-m}^\ell \psi''^2
  \\&\qquad+ \frac{\bar B_1 ^2}{k^2}   \int_{\mathbb{R} } \abs{ \varphi^\prime }^2 \int_{\mathbb{R} } \abs{\phi_n^\prime }^2\int_{-m}^\ell \psi^2
  +
  {\bar B_3^2} \int_{\mathbb{R} } \abs{ \varphi }^2 \int_{\mathbb{R} } \abs{\phi_n^\prime }^2\int_{-m}^\ell \psi'^2.
\end{split}
\end{equation}
Note that
\begin{equation}
{\bar B_3^2} \int_{\mathbb{R} } \abs{ \varphi }^2 \int_{\mathbb{R} } \abs{\phi_n^\prime }^2\int_{-m}^\ell \psi'^2
={\bar B_3^2\left(\frac{1}{\ell}+\frac{1}{m}\right)}\int_\Sigma   w_{k,n} ^2
\end{equation}
and
\begin{equation}
\int_{\mathbb{R} } \abs{\phi_n  }^2 =O(n^{-1})\text{ and }\int_{\mathbb{R} } \abs{\phi_n^\prime }^2 =O(1).
\end{equation}
Since $\abs{\bar B_3}<\mc$ and hence ${\bar B_3^2\left(\frac{1}{\ell}+\frac{1}{m}\right)}<\rj g$, then for sufficiently large $k$ and $n$ we have $E_0(w_{k,n})< 0$. This proves the assertion $(2)$.
\end{proof}

\begin{remark}
It follows from Lemmas \ref{criticalb} and \ref{signsign} that for any constant vector $\bar B \in \mathbb{R}^3$,
\begin{equation}\label{va12}
\inf_{\substack{w\in \Hs\\  \int_\Sigma   \abs{w}^2 =1}} \int_\Omega \abs{(\bar B\cdot \nabla )w}^2 = {\bar B_3^2\left(\dis\frac{1}{\ell}+\frac{1}{m}\right)}.
\end{equation}
\end{remark}

By Lemma \ref{signsign}, we see that if $\abs{\bar B_3}\ge  \mc$, then $E(w;s)\ge E_0(w)\ge 0$ and hence no growing mode solutions can be constructed; if $\abs{\bar B_3}<  \mc$, then $E(w;s)$ has the possibility to be negative, and we record this in the following lemma.

\begin{lem}\label{negativein} If $\abs{\bar B_3}<\mc$, then there exists $s_0>0$ depending on $g,\rho,\bar B,m,\ell,\mu$ such that for $0<s\leq s_0$ it holds that $\alpha(s)<0$.
\end{lem}

\begin{proof}
Since $E$ and $J$ have the same homogeneity, we may reduce to constructing any function $w\in \Hs$ such that $E(w;s) <0$.  Since $\abs{\bar B_3}<\mc$, we know from the assertion $(2)$ of Lemma \ref{signsign} that
there exists $ \widetilde{w}$ such that
$E_0(\widetilde{w})<0$. Obviously, we have
\begin{equation}\label{ineq12}
E(\widetilde{w};s)= E_0(\widetilde{w})+sE_1(\widetilde{w})\le
E_0(\widetilde{w})+sC_1
\end{equation} for a
constant $C_1$ depending on
$g,\rho,\bar B,m,\ell,\mu$.  Then
there exists $s_0>0$ depending on  these parameters such that for
$s\le s_0$ it holds that
$E(\widetilde{w};s)<0$.  Thus  $\alpha(s)<0$ for $s \le s_0$.
\end{proof}

\begin{remark}
For a  minimizer $w \in \mathfrak{S}$ we have
\begin{equation}
- \int_\Sigma  \rj g w_3^2\le \al(s)  ,
\end{equation}
which in particular requires that $w_3\not\equiv 0$ on $\Sigma$ if $\al(s) <0$.
\end{remark}

We now study the behavior of $\alpha(s)$ as a function of $s\ge 0$.
\begin{lem}\label{behavior}We have the following statements.

\begin{enumerate}
\item

$\alpha(s)$ is strictly increasing and $\alpha\in C_{loc}^{0,1}((0,\infty))\cap C^0((0,\infty))$.

\item There exist constants
 $C_1,C_2,C_3,C_4>0$ depending on $g,\rho,\bar B,m,\ell,\mu$ so that
\begin{equation}
-C_3s^{-1}+sC_4\le \alpha(s)\le-C_2+sC_1.\end{equation}
\end{enumerate}
\end{lem}
\begin{proof} Recall the energy decomposition \eqref{E_def}
along with \eqref{E0_def} and \eqref{E1_def}. It has the same form as in Proposition 3.6 of
\cite{3GT2}, hence the assertion $(1)$ follows in the same way.

For $(2)$, the second inequality follows by taking $C_2=-E_0( \tilde w)>0$ in
\eqref{ineq12}. On the other hand, \eqref{11} implies that
\begin{equation}
\alpha(s)   \ge \inf_{w\in \mathfrak{S}}\left( s E_1(w)-C\sqrt{E_1(w)}\right)\ge s \inf_{w\in \mathfrak{S}} E_1(w) -C_3s^{-1}.
\end{equation}
We denote by $C_4$ this positive infimum, then the first inequality follows.
\end{proof}

We may now define the open set
\begin{equation}\mathcal{S}=\al^{-1}((-\infty,0))\subset(0,\infty).\end{equation}
Note that by Lemma \ref{negativein} $\mathcal{S}$ is non-empty and allows us to define
$\lam(s)=\sqrt{-\alpha(s)}$ for $s\in\mathcal{S}$. We will then make a fixed-point argument to find $s\in \mathcal{S}$
such that $s=\lam(s)$.

\begin{lem}\label{lambdas}
There exists a unique $s\in \mathcal{S}$ so that
$\lam(s)=\sqrt{-\al( s)}>0$ and
\begin{equation}\label{fixedpoint}s=\lam(s).\end{equation}
\end{lem}
\begin{proof}
By Lemma \ref{behavior} $(1)$, there exists $s_\ast>0$
such that
\begin{equation}\mathcal{S}=\al^{-1}((-\infty,0))=(0,s_\ast),\end{equation}
and if we define
function $\Phi: \mathcal{S}=(0,s_\ast)\rightarrow(0,\infty)$ by
\begin{equation}\Phi(s)=s/\lam(s),\end{equation}
then $\Phi(s)$ is continuous and strictly increasing in $s$. Moreover, by Lemma \ref{behavior} $(2)$, we have that
$\lim_{s\rightarrow0}\Phi(s)=0$ and $\lim_{s\rightarrow
s_\ast}\Phi(s)=+\infty$. Hence there is unique $s\in (0,s_\ast)$ so
that $\Phi(s)=1$, which gives \eqref{fixedpoint}.
\end{proof}

In conclusion, we now have the existence of solutions to the
original problem \eqref{perturb_linear steady}.

\begin{prop}\label{w_soln_2}
If $\abs{\bar B_3}<\mc$, then there exists a solution $(w,q)$, and $\lam>0$  to \eqref{perturb_linear steady} so that  $w_3\not\equiv 0$.
The solutions are smooth when restricted to $\Omega_+$ or $\Omega_-$.
\end{prop}
\begin{proof}
For $\lam>0$ given in Lemma \ref{lambdas}, we define $w$, a minimizer of \eqref{mu_def}, and $q$, the associated Lagrange multiplier produced by Proposition \ref{propro1}, which solve \eqref{perturb_linear steady s}, with $s=\lam$.  This gives a solution to \eqref{perturb_linear steady}.
\end{proof}

We may now construct a growing mode solution to the linearized problem \eqref{perturb_linear}.

\begin{thm}\label{growingmode}
Assume $\abs{\bar B_3}<\mc$. Let $\lam>0$ be given in Lemma \ref{lambdas}. Then there is a growing mode solution to \eqref{perturb_linear} so that
\begin{equation}\label{gro1}
\eta(t)=e^{\Lam t} \eta(0),\ u(t)=e^{\Lam t} u(0),\ p(t)=e^{\Lam t} p(0).
\end{equation}
\end{thm}

\begin{proof}
 Let $(w,q)$ be the solution to \eqref{perturb_linear steady} with $\lam>0$ as stated in Proposition \ref{w_soln_2}. We then define $\eta$, $u$, and $p$ according to \eqref{ansatz} and \eqref{ansatz1}. Then we have that $(\eta, u, p)$ solve the linearized problem \eqref{perturb_linear}.  Moreover, $(\eta, u, p)$ satisfy \eqref{gro1}.
\end{proof}

\section{Nonlinear stability for $\abs{\bar B_3}>\mathcal{M}_c$}\label{nonlinear stability}

In this section, we will prove Theorem \ref{thic} for $\abs{\bar B_3}>\mathcal{M}_c$. The main part of the proof is to derive a priori estimates for solutions $(\eta,u,p)$ to \eqref{reformulationic} in our functional framework, i.e. for solutions satisfying $\se{2N}$, $\sd{2N}$, $\mathcal{J}_{2N}$, $\f<\infty$.  Throughout this section we will assume that
\begin{equation}
\mathcal{E}_{2N}(t)\le \mathcal{G}_{2N}(T)\le \delta^2\le 1
\end{equation}
for some sufficiently small $\delta>0$ and for all $t \in [0,T]$ where $T>0$ is given. We will implicitly allow $\delta$ to be made smaller in each result, but we will reiterate the smallness of $\delta$ in our main result. Here is the main result of this section.
\begin{thm}\label{Ap}
There exists a universal constant $0 < \delta < 1$ so that if $
\mathcal{G}_{2N} (T) \le \delta^2$, then
\begin{equation}\label{Apriori}
 \mathcal{G}_{2N} (t) \ls\mathcal{E}_{2N} (0)+ \mathcal{F}_{2N}(0) \text{ for all }0 \le t \le
 T.
\end{equation}
\end{thm}

We may first present the

\begin{proof}[Proof of Theorem \ref{thic}]
Theorem \ref{thic} follows from the local well-posedness for the initial data satisfying $\se{2N}(0)+\f(0)$ sufficiently small and the a priori estimates of Theorem \ref{Ap}, by a continuity argument.
\end{proof}

The rest of this section is devoted to proving Theorem \ref{Ap}.

\subsection{Energy evolution}\label{stability evolution}

In this subsection we derive energy evolution estimates for temporal and horizontal spatial derivatives by using the energy-dissipation structure of \eqref{reformulationic}.

\subsubsection{Energy evolution of time derivatives}

For the temporal derivatives, it is a key to use the original geometric formulation \eqref{reformulationic}. As well explained by \cite{GT_per} in the study of the incompressible viscous surface wave problem, the reason is that if we attempted to use the linear perturbed formulation \eqref{perturb}, there would be too many time derivatives of $p$ and $ u$ appearing in the nonlinear estimates that are out of control. Applying $\dt^j$ for $j=0,\dots,n$ to \eqref{reformulationic}, we find that
\begin{equation}\label{linear_geometric}
\begin{cases}
  \dt (\dt^j \eta) =\dt^j u  & \text{in } \Omega
\\  \rho\dt (\dt^j u) +\diverge_\a S_{\mathcal{A}}(\dt^j p,\dt^j u)- (\bar B\cdot\nabla)^2(\dt^j \eta) = F^{1,j}  & \text{in } \Omega \\
 \diva (\dt^j u) = F^{2,j} & \text{in } \Omega
\\ \Lbrack \dt^j u\Rbrack=0,\quad \Lbrack S_{\mathcal{A}}(\dt^j p,\dt^j u )\Rbrack \mathcal{N} -\jump{ \bar B_3 (\bar B\cdot\nabla) (\dt^j \eta)}= \rj g\dt^j \eta_{3 }  e_3+F^{3,j} &\text{on }\Sigma \\
\dt^j u  =0 & \text{on } \Sigma_{m,\ell},
\end{cases}
\end{equation}
where
\begin{equation}\label{F_def_start}
\begin{split}
& F_i^{1,j}  = \sum_{0 < \ell \le j}  C_j^\ell\left\{\mu
\mathcal{A}_{lk} \p_k (\dt^\ell  \mathcal{A}_{im} \dt^{j-\ell} \p_m u_l  +\dt^\ell  \mathcal{A}_{lm} \dt^{j-\ell}\p_m u_i)
  \right.
\\
 &\left.\qquad\qquad + \mu  \dt^\ell \mathcal{A}_{lk}\dt^{j-\ell}\p_k (\mathcal{A}_{im} \p_m u_l + \mathcal{A}_{lm}\p_m u_i) -  \dt^\ell  \mathcal{A}_{ik} \dt^{j-\ell}\p_k p\right\},\ i=1,2,3,
 \end{split}
\end{equation}
 \begin{equation}\label{i_F2u_def}
 F^{2,j} = - \sum_{0 < \ell \le j}  C_j^\ell \dt^\ell  \mathcal{A}_{lk} \dt^{j-\ell}\p_k u_l
,
\end{equation}
and
\begin{equation}\label{i_F3_def2}
\begin{split}
&
 F_{i}^{3,j} = \sum_{0 < \ell \le j}  C_j^\ell  \left\{  \jump{\mu
  \dt^\ell ( \n_{l} \mathcal{A}_{ik} ) \dt^{j - \ell} \p_k u_l}+\jump{\mu \dt^\ell  ( \n_{l} \mathcal{A}_{lk} ) \dt^{j - \ell} \p_k u_i}\right.
\\
 &\left.\qquad\qquad-\dt^\ell  \n_{i} \dt^{j - \ell} \jump{ p} \right\}+\rj g\dt^j\left(\eta_3(\mathcal{N}_{i}-\delta_{i3})  \right),\ i=1,2,3.
    \end{split}
 \end{equation}
Since we can not hope to get any estimates of $p$ ($i.e.$ $\partial_t^j p$) without spatial derivatives, we have to pay more attention on the expression of $F^{2,j}$.  We will need some structural
conditions on $F^{2,j}$ which allow us to integrate by parts in the interaction between $\partial_t^j p$ and $F^{2,j}$. Indeed, note that $ \diva u=\p_i(\a_{mi}u_m)=0$. Then we have
\begin{equation}\label{Qstr}
F^{2,j}= {\rm div}Q^{2,j}\text{ with }Q_i^{2,j}=
 -\sum_{0<\ell\le j}C_j^\ell\partial_t^\ell \mathcal{A}_{mi} \partial_t^{j-\ell}u_m,\ i=1,2,3.
 \end{equation}
Since $ \mathcal{A}_{m3}$  and $u$ are continuous across $\Sigma$, and $u=0$ on $\Sigma_{m,\ell}$, we have
\begin{equation}\label{Qstr1}
\Lbrack Q_3^{2,j}\Rbrack=0\hbox{ on }\Sigma\text{ and } Q^{2,j}=0\hbox{ on }\Sigma_{m,\ell}.\end{equation}
These facts are important for handling the pressure term.

We record the estimates of these nonlinear terms $F^{i,j}$ and $Q^{2,j}$ in the following lemma.
\begin{lem}\label{p_F_estimates}
For $n=N+2$ or $n=2N$, it holds that
\begin{equation}\label{p_F_e_01}
 \ns{F^{1,j} }_{0}+ \ns{F^{2,j}}_{2}+ \ns{\dt F^{2,j}}_{0}+\ns{   Q^{2,j} }_{2}+ \as{F^{3,j}}_{3/2}  \ls \se{N+2} \sd{n},
\end{equation}
\begin{equation}\label{p_F_e_011}
 \ns{ \dt Q^{2,n} }_{0} +\as{\dt Q^{2,n}_3 }_{-1/2} \ls \se{N+2} \sd{n}
\end{equation}
and
\begin{equation}\label{p_F_e_02}
\norm{Q^{2,n} }_2^2   \ls \se{N+2}\se{n} .
\end{equation}
\end{lem}
\begin{proof}
Note that all terms in the definitions of $F^{i,j}$ and $Q^{2,j}$ (and hence $\dt F^{2,j}$ and $\dt Q^{2,n}$) are at least quadratic; each term can be written in the form $X Y$, where $X$ involves fewer derivative counts than $Y$. We then estimate $Y$ in $H^k$ ($k = 0, 2$ or $3/2$, respectively) and $X$ in $H^m$ for $m$ depending on $k$, using Lemma \ref{sobolev}, trace
theory along with the definitions of $\se{n}$ and $\sd{n} $, to bound $\norm{X}_{m}^2\ls  \se{N+2} $ and $\norm{Y}_{k}^2\ls   \sd{n} $.
Then $\norm{XY}_k^2\le \norm{X}_{m}^2\norm{Y}_{k}^2\ls   \se{N+2} \sd{n} $, and the estimate \eqref{p_F_e_01} follows.

 The estimate of $\ns{ \dt Q^{2,n} }_{0}$ in \eqref{p_F_e_011} follows in the same way as \eqref{p_F_e_01}. To estimate the other one, we need to use Lemma \ref{h-1/2} and \eqref{Qstr}:
\begin{equation}
  \as{\dt Q^{2,n}_3 }_{-1/2}  \ls \ns{\dt Q^{2,n}  }_{0}+\ns{\dt {\rm div }Q^{2,n}  }_{0}=\ns{\dt Q^{2,n} }_{0}+\ns{\dt F^{n,j} }_{0}   .
\end{equation}
Then the estimate follows by \eqref{p_F_e_01}.

The estimate \eqref{p_F_e_02} follows similarly as \eqref{p_F_e_01} with instead bounding $\norm{Y}_{2}^2\ls   \se{n} $.
\end{proof}

For a generic integer $n\ge 3$, we define the energy involving only temporal derivatives by
\begin{equation}
 \bar{\mathcal{E}}^{t}_{n} =
 \sum_{j=0}^{n} \hal\left( \int_\Omega  \rho\abs{\dt^j u}^2+ \abs{(\bar B\cdot\nabla)\dt^j\eta}^2  - \int_{\Sigma } \rj g \abs{\dt^j\eta_3}^2
 \right)
\end{equation}
and the corresponding dissipation by
\begin{equation}
 \bar{\mathcal{D}}_n^{t} = \sum_{j=0}^{ n}  \int_\Omega\frac{ \mu}{2} \abs{ \sg \dt^j u}^2.
\end{equation}
Then we have the following energy evolution.
\begin{prop}\label{i_temporal_evolution  N}
For $n=N+2$ or $n=2N$, it holds that
\begin{equation} \label{i_te_0}
  \frac{d}{dt}\left(\bar{\mathcal{E}}_{n}^{t}-\mathfrak{B}_n\right)
+ \bar{\mathcal{D}}_{n}^{t}
  \ls   \sqrt{\se{N+2} } \sd{ n}
\end{equation}
and
\begin{equation}\label{bnbn}
 \mathfrak{B}_n:= \int_\Omega  \nabla \dt^{n-1} p  Q^{2,n}+ \int_\Sigma \jump{\dt^{n-1} p} Q^{2,n}_3.
\end{equation}
\end{prop}

\begin{proof}
We let $n$ denote either $2N$ or $N+2$ throughout the proof. Taking the dot product of the second equation of $\eqref{linear_geometric}$ with $\dt^j u$, $j=0,\dots,n,$ and then integrating by parts, using the third to fifth equations, we obtain
\begin{equation}\label{i_ge_ev_0}
\begin{split}
&\hal  \frac{d}{dt} \int_\Omega  \rho \abs{\dt^j u}^2
+ \int_\Omega \frac{\mu}{2} \abs{ \sg_{\mathcal{A}} \dt^j u}^2+\int_\Omega (\bar B\cdot\nabla) \dt^j\eta\cdot (\bar B\cdot\nabla) \dt^j u
\\&\quad= \int_\Omega   (  \dt^j u\cdot F^{1,j}+  \dt^j p  \diverge_\a (\dt^j u))
 +\int_\Sigma \left(\jump{\S_{\a}(\dt^j p,\dt^j u)}\n -\jump{\bar B_3(\bar B\cdot\nabla)(\dt^j\eta)}\right)\cdot \dt^j u
\\&\quad= \int_\Omega   (   \dt^j u\cdot F^{1,j}+  \dt^j p  F^{2,j})
+\int_{\Sigma}  F^{3,j}\cdot \dt^j u
+\int_{\Sigma} \rj g\dt^j\eta_{3 }    \dt^j u_{3}.
\end{split}
\end{equation}
By the first equation, we get
\begin{equation}
 \int_\Omega (\bar B\cdot\nabla)  \dt^j\eta\cdot (\bar B\cdot\nabla) \dt^j u =  \int_\Omega (\bar B\cdot\nabla)\dt^j\eta\cdot (\bar B\cdot\nabla) \dt^{j+1}\eta=\hal\frac{d}{dt}\int_\Omega \abs{(\bar B\cdot\nabla) \dt^j\eta}^2
\end{equation}
and
\begin{equation}
 \int_{\Sigma} \rj g\dt^j\eta_{3 }  \cdot \dt^j u_3
 = \hal\frac{d}{dt}\int_{\Sigma}  \rj g\abs{\dt^j\eta_3}^2 .
\end{equation}
Hence, we have
\begin{equation} \label{identity1}
\begin{split}
&\hal  \frac{d}{dt} \left(\int_\Omega  \rho\abs{\dt^j u}^2+\abs{(\bar B\cdot\nabla)\dt^j\eta}^2 - \int_{\Sigma} \rj g \abs{\dt^j\eta_3}^2 \right)
+   \int_\Omega  \frac{\mu}{2} \abs{ \sg_{\mathcal{A}} \dt^j u}^2
\\&\quad= \int_\Omega   (   \dt^j u\cdot F^{1,j}+  \dt^j p  F^{2,j})
+\int_\Sigma  \dt^j u\cdot F^{3,j} .
\end{split}
\end{equation}

We now estimate the right hand side of \eqref{identity1}. For the $F^{1,j}$ term, by \eqref{p_F_e_01}, we may bound
\begin{equation}\label{i_te_2}
\int_\Omega  \dt^j u\cdot F^{1,j} \le   \norm{\dt^j u}_{0}   \norm{F^{1,j}}_0 \ls  \sqrt{\sd{n} } \sqrt{\se{N+2} \sd{n}}
 .
\end{equation}
For the $F^{3,j}$ term, by \eqref{p_F_e_01} and the trace theory, we have
\begin{equation}\label{i_te_3}
 \int_\Sigma  \dt^j u\cdot F^{3,j} \ls  \abs{\dt^{j} u}_{0}  \abs{F^{3,j}}_{0}
\ls    \norm{\dt^{j} u}_{1} \sqrt{\se{N+2}\sd{n}} \le  \sqrt{\sd{n} } \sqrt{\se{N+2} \sd{n}}
.
\end{equation}
For the $F^{2,j}$ term, we need much more care. First, since we can not get any estimates of $\partial_t^j p$ without spatial derivatives, we need to use the structure of $F^{2,j}$ and employ an integration by parts in space. Indeed, by \eqref{Qstr} and \eqref{Qstr1}, we deduce
\begin{equation}
\int_\Omega  \dt^j p  F^{2,j} =\int_\Omega  \dt^j p  {\rm div}Q^{2,j}
=-\int_\Sigma \jump{\dt^j p} Q^{2,j}_3-\int_\Omega  \nabla \dt^j p   Q^{2,j}
\end{equation}
Second, we need to consider the case $j< n$ and $j= n$ separately. We first deal with the integral in the domain $\Omega$. For $j< n$, by \eqref{p_F_e_01} we have
\begin{equation}\label{i_te_4}
-\int_\Omega  \nabla \dt^j p  Q^{2,j}   \le   \norm{\nabla \dt^j p}_{0}   \norm{Q^{2,j} }_0 \ls  \sqrt{\sd{n}} \sqrt{\se{N+2} \sd{n}}.
\end{equation}
The case $j= n$ is much more involved since we can not control $\nabla\dt^{ n}p$. We are then forced to integrate by parts in time:
\begin{equation}
 -\int_\Omega   \nabla \dt^n p  Q^{2,n}= -\frac{d}{dt}\int_\Omega  \nabla \dt^{n-1} p  Q^{2,n}+\int_\Omega     \nabla \dt^{n-1} p \dt Q^{2,n}   .
\end{equation}
By \eqref{p_F_e_011}, we may bound
\begin{equation}\label{i_te_5}
 \int_\Omega  \nabla \dt^{n-1} p \dt Q^{2,n} \ls   \norm{\nabla \dt^{n-1} p }_{0} \norm{\dt   Q^{2,n}}_{0} \ls   \sqrt{\sd{n}} \sqrt{\se{N+2} \sd{n}} .
\end{equation}
We next deal with the integral on the boundary $\Sigma$. For $j< n$, by \eqref{p_F_e_01} and the trace theory we have
\begin{equation}\label{i_te_400}
-\int_\Sigma \jump{\dt^j p} Q^{2,j}_3 \le   \abs{\jump{\dt^j p}}_{0}   \abs{Q^{2,j}_3 }_{0} \ls  \abs{\jump{\dt^j p}}_{0}   \norm{Q^{2,j}_3 }_{1} \ls\sqrt{\sd{n}} \sqrt{\se{N+2} \sd{n}}.
\end{equation}
For $j= n$, we integrate by parts in time:
\begin{equation}
-\int_\Sigma \jump{\dt^n p} Q^{2,n}_3=- \frac{d}{dt}\int_\Sigma \jump{\dt^{n-1} p} Q^{2,n}_3+\int_\Sigma     \jump{\dt^{n-1} p}\dt Q^{2,n}_3   .
\end{equation}
By \eqref{p_F_e_011}, we may bound
\begin{equation}\label{i_te_500}
\int_\Sigma     \jump{\dt^{n-1} p}\dt Q^{2,n}_3 \ls   \abs{ \jump{\dt^{n-1} p}}_{1/2} \abs{\dt   Q^{2,n}_3}_{-1/2} \ls   \sqrt{\sd{n}} \sqrt{\se{N+2} \sd{n}} .
\end{equation}

Now we combine \eqref{i_te_2}--\eqref{i_te_500} to deduce from \eqref{identity1} that, summing over $j$,
\begin{equation} \label{iiii}
  \frac{d}{dt}\left(\bar{\mathcal{E}}_{n}^{t}- \mathfrak{B}_n\right)
+  \sum_{j=0}^{ n} \int_\Omega \frac{\mu}{2} \abs{ \sg_{\mathcal{A}} \dt^j u}^2
  \ls   \sqrt{\se{N+2} } \sd{ n},
\end{equation}
where $\mathfrak{B}_n$ is defined by \eqref{bnbn}.
We then seek to replace $ \abs{ \sg_{\mathcal{A}} \dt^{j} u}^2$ with $\abs{\sg \dt^{j} u}^2$
in \eqref{iiii}.  To this end we write
\begin{equation}\label{i_te_8}
 \abs{ \sg_{\mathcal{A}} \dt^{j} u}^2 = \abs{\sg \dt^{j} u}^2 +   \left(\sg_{\mathcal{A}} \dt^{j} u + \sg \dt^{j} u\right): \left(\sg_{\mathcal{A}} \dt^{j} u - \sg \dt^{j} u\right)
\end{equation}
and estimate the last three terms on the right side.  Note that
\begin{equation}
 \sg_{\mathcal{A}} \dt^{j} u \pm \sg \dt^{j} u  = (\mathcal{A}_{ik} \pm \delta_{ik})\p_k \dt^{j} u_l + (\mathcal{A}_{lk} \pm \delta_{lk})\p_k \dt^{j} u_i.
\end{equation}
Sobolev embeddings provide the bounds
\begin{equation}
 \abs{ \sg_{\mathcal{A}} \dt^{j} u - \sg \dt^{j} u } \ls \sqrt{\se{N+2}} \abs{\nab \dt^{j} u}
\text{ and }
 \abs{ \sg_{\mathcal{A}} \dt^{j} u + \sg \dt^{j} u } \ls (1+ \sqrt{\se{N+2}}) \abs{\nab \dt^{j} u}.
\end{equation}
We then get
\begin{equation}\label{i_te_9}
 \int_\Omega \mu \abs{ \left(\sg_{\mathcal{A}} \dt^{j} u + \sg \dt^{j} u\right): \left(\sg_{\mathcal{A}} \dt^{j} u - \sg \dt^{j} u\right)}
\ls  (\sqrt{\se{N+2}} + \se{N+2}) \int_\Omega \abs{\nab \dt^{j} u}^2  \ls  \sqrt{\se{N+2}} \sd{n}.
\end{equation}
We may then use \eqref{i_te_8} and \eqref{i_te_9} to replace in \eqref{iiii} and derive \eqref{i_te_0}.
\end{proof}

\subsubsection{Energy evolution of horizontal spatial derivatives}

For the horizontal spatial derivatives, it turns out to be convenient to rewrite the system \eqref{reformulationic} in a linear form
such that the coefficients get fixed; the elliptic regularity can be also readily adapted
in later sections. We shall use the following linear perturbed formulation of \eqref{reformulationic}:
\begin{equation}\label{perturb}
\begin{cases}
 \dt\eta=u&\hbox{in }\Omega
\\\rho\partial_t u-\mu\Delta u+\nabla p- (\bar B\cdot\nabla)^2 \eta=G^1\quad&\hbox{in }\Omega
\\ \diverge u=G^2&\hbox{in }\Omega
\\ \Lbrack u\Rbrack=0,\quad \Lbrack pI-\mu\mathbb{D}u\Rbrack e_3-\jump{\bar B_3(\bar B\cdot\nabla)\eta}= \rj g\eta_3  e_3+G^3 &\hbox{on }\Sigma
\\ u  =0 &\hbox{on }\Sigma_{m,\ell},
\end{cases}
\end{equation}
where
\begin{equation}\label{G1_def}
 G^1 =   \mu(\Delta_\a-\Delta)u-(\nabla_\a-\nabla)p,
 \end{equation}
\begin{equation}\label{G2_def}
G^2=(\diverge_\a-\diverge)u
 \end{equation}
 and
\begin{equation}\label{G3_def}
G^3= \jump{\mu\left(\mathbb{D}_{\mathcal{A}}u \mathcal{N}-\mathbb{D} u e_3 \right)}+(g\rj \eta_3-\jump{p})(\mathcal{N}-e_3).
 \end{equation}

We record the estimates of the nonlinear terms $G^1, G^2$ and $G^3$ in the following lemma.
\begin{lem}\label{p_G_estimates}
For $n=N+2$ or $n=2N$, it holds that
 \begin{equation}\label{p_G_e_0}
  \ns{ \bar{\nab}_0^{2n-2} G^1}_{0}  +  \ns{ \bar{\nab}_0^{2n-2}  G^2}_{1} +
 \as{ \bar{\nab}_{\ast 0}^{\ 2n-2}  G^3}_{1/2}    \ls  \se{N+2}\se{n} .
 \end{equation}
We also have
\begin{equation}\label{p_G_e_001}
\ns{ \bar{\nab}_0^{ 4N-1} G^1}_{0} +  \ns{ \bar{\nab}_0^{ 4N-1}  G^2}_{1}  +
 \as{\bar{\nab}_{\ast 0}^{\   4N-1} G^3}_{1/2}  \ls  \se{N+2}(\sd{2N}+\mathcal{J}_{2N}+ \f )
\end{equation}
and
\begin{equation}\label{p_G_e_002}
\ns{ \bar{\nab}_0^{ 2(N+2)-1} G^1}_{1} +  \ns{ \bar{\nab}_0^{ 2(N+2)-1}  G^2}_{2} +
 \as{\bar{\nab}_{\ast 0}^{\   2(N+2)-1} G^3}_{3/2}  \ls \se{2N}\sd{N+2} .
\end{equation}
\end{lem}
\begin{proof}
Note that all terms in the definitions of $G^i$ are at least quadratic. We apply these
space-time differential operators to $G^i$ and then expand using the Leibniz rule; each product in
the resulting sum is also at least quadratic. Then the estimate \eqref{p_G_e_0} follows in the same way as Lemma \ref{p_F_estimates}.

The last two terms in the right hand side of \eqref{p_G_e_001} is due to the control of the highest spatial derivatives in some products, which is not controlled by $\sd{2N}$ but rather by $\mathcal{J}_{2N}+ \f$. Note that the other factors in such products are of low derivatives and hence can be easily controlled by $\se{N+2}$. Then the estimate \eqref{p_G_e_001} follows.

The proof of the estimate \eqref{p_G_e_002} is somewhat easier. Indeed, we may write each term in the form $X Y$, where $X$ involves fewer derivative counts than $Y$; then we simply bound the various norms of $Y$ by $ \se{2N} $ and the various norms of $X$ by $\sd{N+2}$. Then the estimate \eqref{p_G_e_002} follows.
\end{proof}

For a generic integer $n\ge 3$, we define the energy involving only horizontal spatial derivatives by
\begin{equation}
 \bar{\mathcal{E}}^{\ast}_{n} =
 \sum_{\substack{{\al\in\mathbb{N}^{2}}\\ {\al_1+\al_2\ge 1,|\al|\le 2n}}} \hal\left(  \int_\Omega \rho \abs{\p^\al  u }^2  +\abs{(\bar B\cdot \nabla)\partial^\alpha  \eta }^2- \int_{\Sigma} \rj g \abs{\partial^\alpha  \eta_3}^2\right)
\end{equation}
and the corresponding dissipation by
\begin{equation}
 \bar{\mathcal{D}}_n^{\ast} =
 \sum_{\substack{\al\in\mathbb{N}^{2}\\ \al_1+\al_2\ge 1,|\al|\le 2n}} \int_\Omega\frac{\mu}{2} \abs{\sg \p^\al  u }^2.
\end{equation}
Then we have the following energy evolution.

\begin{prop}\label{p_upper_evolution  N12}
It holds that
\begin{equation}\label{p_u_e_00}
\frac{d}{dt}\bar{\mathcal{E}}^{\ast}_{2N}+\bar{\mathcal{D}}_{2N}^{\ast}\ls\sqrt{ \se{N+2}  } ( \sd{2N}+\mathcal{J}_{2N} +\f)
\end{equation}
and
\begin{equation}\label{p_u_e_00''}
\frac{d}{dt}\bar{\mathcal{E}}^\ast_{N+2}+\bar{\mathcal{D}}_{N+2}^\ast\ls\sqrt{ \se{2N}   } \sd{N+2}.
\end{equation}
\end{prop}
\begin{proof}
We let $n$ denote either $2N$ or $N+2$ throughout the proof. We take $\al\in \mathbb{N}^{2}$ so that $\al_1+\al_2\ge 1$ and $|\al|\le 2n$. Applying $\p^\al $ to the second equation of \eqref{perturb} and then taking the dot product with $\p^\al u$, using the other equations as in Proposition \ref{i_temporal_evolution  N}, we find that
\begin{equation} \label{p1p11110}
\begin{split}
 &\hal \frac{d}{dt}\left(  \int_\Omega \rho \abs{\p^\al  u }^2  +\abs{(\bar B\cdot \nabla)\partial^\alpha  \eta }^2- \int_{\Sigma} \rj g \abs{\partial^\alpha  \eta_3}^2\right)+ \int_\Omega \frac{\mu }{2} \abs{\sg \p^\al  u }^2
 \\&\quad= \int_\Omega   \p^\al  u  \cdot \p^\al( G^1-\mu\nabla G^2) +\int_\Omega   \p^\al  p  \p^\al  G^2+ \int_\Sigma \p^\al u \cdot \p^\al G^3 .
\end{split}
\end{equation}

We now estimate the terms on the right hand side of \eqref{p1p11110}. Since $\al_1+\al_2\ge 1$, we may write $\alpha = \gamma +(\alpha-\gamma)$ for some $\gamma \in \mathbb{N}^2$ with $\abs{\gamma}=1$. We first consider the case $n=2N$. We can then integrate by parts and use \eqref{p_G_e_001} to have
\begin{equation}\label{i_de_14}
\begin{split}
  &\int_\Omega     \p^\al   u \cdot \p^\al( G^1-\mu\nabla G^2) =- \int_\Omega     \p^{\alpha+\gamma}   u \cdot   \p^{\alpha-\gamma}  ( G^1-\mu\nabla G^2)
 \\&\quad\le \norm{\p^{\alpha+\gamma}   u}_{0} \left( \norm{ \p^{\alpha-\gamma}   G^1 }_{0}+ \norm{ \p^{\alpha-\gamma} \nabla  G^2 }_{0}\right)
\le \norm{\p^{\alpha}   u}_{1} \left( \norm{ G^1 }_{4N-1}+ \norm{   G^2 }_{4N}\right)
\\&\quad\ls \sqrt{ \sd{2N} } \sqrt{ \se{N+2}(\sd{2N}+\mathcal{J}_{2N} +\f)  } .
\end{split}
\end{equation}
Similarly, by using additionally the trace theory, we have
\begin{equation}\label{i_de_16}
\begin{split}
 &\int_\Sigma     \p^\al  u \cdot   \p^\al G^3   =  -\int_\Sigma  \p^{\alpha+\gamma}  u \cdot   \p^{\alpha-\gamma} G^3
\\& \quad\le \abs{\p^{\alpha+\gamma}  u}_{-1/2}  \abs{ \p^{\alpha-\gamma} G^3 }_{1/2}  \le \abs{\p^{\alpha}  u}_{1/2}   \abs{  G^3 }_{4N-1/2}
 \le \norm{\p^{\alpha}  u}_{1} \abs{  G^3 }_{4N-1/2}
  \\&\quad\ls \sqrt{ \sd{2N} } \sqrt{ \se{N+2}(\sd{2N}+\mathcal{J}_{2N} +\f)  }.
\end{split}
\end{equation}
For the $G^2$ term we do not need to (and we can not) integrate by parts:
\begin{equation}\label{i_de_15}
\begin{split}
  \int_\Omega  \p^\al   p \p^\al  G^2
&\le \norm{\p^{\alpha-\gamma}\p^\gamma  p}_{0}  \norm{\p^{\alpha} G^2 }_{0}
\le \norm{ \p^\gamma p}_{4N-1}  \norm{  G^2 }_{4N}
  \\&\ls \sqrt{\mathcal{J}_{2N} }\sqrt{ \se{N+2}(\sd{2N}+\mathcal{J}_{2N} +\f)  } .
\end{split}
\end{equation}
Hence, by \eqref{i_de_14}--\eqref{i_de_15}, we deduce from \eqref{p1p11110} that for all $\abs{\alpha} \le 4N$ with $\al_1+\al_2\ge 1$,
\begin{equation} \label{7878}
\begin{split}
 &\hal \frac{d}{dt}\left(  \int_\Omega \rho \abs{\p^\al  u }^2  +\abs{(\bar B\cdot \nabla)\partial^\alpha  \eta }^2- \int_{\Sigma} \rj g \abs{\partial^\alpha  \eta_3}^2\right)+ \int_\Omega \frac{\mu}{2}  \abs{\sg \p^\al  u }^2
 \\&\quad\ls  \sqrt{ \se{N+2}  } ( \sd{2N}+\mathcal{J}_{2N} +\f).
\end{split}
\end{equation}
The estimate \eqref{p_u_e_00} then follows from \eqref{7878} by summing over such $\alpha$.

We now consider the case $n=N+2$. By \eqref{p_G_e_002}, we have
\begin{equation} \label{i_de_4'}
\begin{split}
   \int_\Omega   \p^\al u  \cdot  \p^\al ( G^1-\mu\nabla G^2) &\le \norm{\p^\al u}_0\left( \norm{ \p^{\alpha }   G^1 }_{0}+ \norm{ \p^{\alpha } \nabla  G^2 }_{0}\right)
  \\& \le \norm{\p^\al u}_0
   \left(\norm{G^1}_{2(N+2)}+\norm{G^2}_{2(N+2)+1}\right)
 \\&\ls \sqrt{ \sd{N+2} } \sqrt{ \se{2N}\sd{N+2} }
\end{split}
\end{equation}
and
\begin{equation}\label{i_de_16;'}
\begin{split}
 \int_\Sigma     \p^\al  u \cdot   \p^\al G^3   &=  -\int_\Sigma  \p^{\alpha-\gamma}  u \cdot   \p^{\alpha+\gamma} G^3
 \\
&  \le \abs{\p^{\alpha-\gamma}  u}_{1/2}  \abs{ \p^{\alpha+\gamma} G^3 }_{-1/2} \le \norm{\p^{\alpha-\gamma}  u}_{1} \abs{  G^3 }_{4N+1/2}
  \\&\ls \sqrt{ \sd{2N} } \sqrt{ \se{N+2}(\sd{2N}+\mathcal{J}_{2N} +\f)  }.
\end{split}
\end{equation}
For the $G^2$ term we need a bit more care. If $|\al|\le 2(N+2)-1$, then
\begin{equation}\label{i_de_5'0}
\begin{split}
  \int_\Omega  \p^\al p \p^\al G^2
&\le \norm{\p^{\alpha-\gamma}\p^\gamma p}_{0}  \norm{ \p^{\alpha }  G^2 }_{0}
\le \norm{ \p^\gamma p}_{2(N+2)-2}  \norm{   G^2 }_{2(N+2)-1}
  \\&\ls \sqrt{ \sd{N+2} }\sqrt{ \se{2N}\sd{N+2}}  .
\end{split}
\end{equation}
If $|\al|= 2(N+2) $,  then we have that $\al_1+\al_2\ge 2$. We may then write $\alpha-\gamma = \beta +(\alpha-\beta-\gamma)$ for some $\beta \in \mathbb{N}^2$ with $\abs{\beta}=1$. Then we integrate by parts to have
\begin{equation}\label{i_de_5'}
\begin{split}
  \int_\Omega  \p^\al p \p^\al G^2&=-\int_\Omega  \p^{\al-\beta-\gamma}\p^\gamma p \p^{\al+\beta}  G^2
 \\&\le \norm{\p^{\al-\beta-\gamma}\p^\gamma p}_{0}  \norm{ \p^{\al+\beta}  G^2}_{0}
\le \norm{ \p^\gamma p}_{2(N+2)-2} \norm{G^2}_{2(N+2)+1}\\&\ls \sqrt{ \sd{N+2} }\sqrt{ \se{2N}\sd{N+2}}  .
\end{split}
\end{equation}
Hence, by \eqref{i_de_4'}--\eqref{i_de_5'}, we deduce from \eqref{p1p11110} that for all $\abs{\alpha} \le 2(N+2)$ with $\al_1+\al_2\ge 1$,
\begin{equation} \label{7878'}
\hal \frac{d}{dt}\left(  \int_\Omega \rho \abs{\p^\al  u }^2  +\abs{(\bar B\cdot \nabla)\partial^\alpha  \eta }^2- \int_{\Sigma} \rj g \abs{\partial^\alpha  \eta_3}^2\right)+ \int_\Omega \frac{\mu}{2}  \abs{\sg \p^\al  u }^2
\ls  \sqrt{ \se{2N}   } \sd{N+2}.
\end{equation}
The estimate \eqref{p_u_e_00''} then follows from \eqref{7878'} by summing over such $\alpha$.
\end{proof}

\subsubsection{Energy evolution recovering $\eta$}

Note that the dissipation estimates in $\bar{\mathcal{D}}_n^t$ of Proposition \ref{i_temporal_evolution  N} and $\bar{\mathcal{D}}_n^\ast$ of Proposition \ref{p_upper_evolution  N12} only contain $u$, we will now recover dissipation estimates of $\eta$ due to the magnetic effect and gravity; the energy estimates of $\eta$ due to the viscosity will be also recovered. Since $\dt \eta =u$, we then only need to estimate for $\eta$ without time derivatives. We recall the structure of $\eta$ from \eqref{imp1}, which follows from the assumption \eqref{eta00} of $\eta_0$. The estimates result from testing the linear perturbed formulation \eqref{perturb} by $\eta$. Note that the boundary conditions of $
\eta=0\text{ on }\Sigma_{m,\ell}$  and  $\jump{\eta} =0\text{ on }\Sigma
$ guarantee this. Moreover, the Jacobian identity ${\rm det} (I+\nabla\eta)=1$ gives the control of ${\rm div}\eta$;
indeed,
\begin{equation}\label{phhicon}
{\rm div}\eta=\Phi,\quad \Phi=-({\rm det} (I+\nabla\eta)-1-{\rm div}\eta)=O(\nabla\eta\nabla^2\eta).
\end{equation}
Again, as for $F^{2,j}$, we will need some structural
conditions on $\Phi$ which allow us to integrate by parts in the interaction between $p$ and $\Phi$ (without spatial derivatives). This is not apparent, we need to do some lengthy but straightforward computations. Indeed, we expand the expression of $\Phi$ to conclude that
 \begin{equation}\label{Qstr2}
\Phi= {\rm div}\Psi \text{ with }
\Psi= \left(\begin{array}{ccc}\eta_1\p_2\eta_2+\eta_1\p_3\eta_3+\eta_1(\p_2\eta_2\p_3\eta_3-\p_3\eta_2\p_2\eta_3)
\\-\eta_1\p_1\eta_2+\eta_2\p_3\eta_3+\eta_1(\p_3\eta_2\p_1\eta_3-\p_1\eta_2\p_3\eta_3)
\\-\eta_1\p_1\eta_3-\eta_2\p_2\eta_3+\eta_1(\p_1\eta_2\p_2\eta_3-\p_2\eta_2\p_1\eta_3)
\end{array}\right).
 \end{equation}
Moreover, since $
\eta=0\text{ on }\Sigma_{m,\ell}$  and  $\jump{\eta} =0\text{ on }\Sigma
$, we have
\begin{equation}\label{Qstr111}
\Psi=0\text{ on }\Sigma_{m,\ell}\text{ and }\jump{\Psi_3} =0\text{ on }\Sigma
.
\end{equation}

We record some estimates of $\Phi$ and $\Psi$ in the following lemma.
\begin{lem}\label{p_G_estimates''}
It holds that
\begin{equation}\label{p_G_e_001''}
 \ns{ \Phi}_{4N} \ls  \se{N+2}(\sd{2N}+\mathcal{J}_{2N}+ \f ),
\end{equation}
\begin{equation}\label{p_G_e_002''}
 \ns{\Phi}_{2(N+2)+1} \ls \se{2N}\sd{N+2}
\end{equation}
and
\begin{equation}\label{Phe_0}
  \ns{ \Psi}_1 \ls \mathcal{E}_{3} \mathcal{D}_{3} .
\end{equation}
\end{lem}
\begin{proof}
The proof proceeds similarly as Lemma \ref{p_G_estimates}.
\end{proof}

For a generic integer $n\ge 3$, we define the recovering energy by
\begin{equation}
 \bar{\mathcal{E}}^\sharp_{n} =\hal\sum_{\substack{\al\in \mathbb{N}^2\\ |\al|\le 2n}}\int_\Omega \frac{\mu}{2}   \abs{\sg  \p^\al  \eta}^2
\end{equation}
and the corresponding dissipation by
\begin{equation}
 \bar{\mathcal{D}}_n^\sharp =\sum_{\substack{\al\in \mathbb{N}^2\\ |\al|\le 2n}} \left(  \int_\Omega  \abs{(\bar B\cdot \nabla)\partial^\alpha  \eta }^2- \int_{\Sigma} \rj g \abs{\partial^\alpha  \eta_3}^2\right).
\end{equation}
Then we have the following energy evolution.

\begin{prop}\label{p_upper_evolution  N'132}
It holds that
\begin{equation}\label{p_u_e_00'132}
\begin{split}
&\frac{d}{dt}\left(\bar{\mathcal{E}}^\sharp_{2N}+ \sum_{\substack{\al\in \mathbb{N}^2\\  |\al|\le 4N}}\int_\Omega\rho \partial^\alpha    u \cdot \partial^\alpha  \eta\right)+\bar{\mathcal{D}}_{2N}^\sharp
\\&\quad\ls \sqrt{ \se{N+2}  }( \sd{2N} + \mathcal{J}_{2N} +\f)+\bar{\mathcal{D}}_{2N}^t+\bar{\mathcal{D}}_{2N}^{\ast}
\end{split}
\end{equation}
and
\begin{equation}\label{p_u_e_00'12132}
\begin{split}
&\frac{d}{dt}\left(\bar{\mathcal{E}}^\sharp_{N+2}+ \sum_{\substack{\al\in \mathbb{N}^2\\  |\al|\le 2(N+2)}}\int_\Omega \rho\partial^\alpha   u  \cdot \partial^\alpha  \eta\right)+\bar{\mathcal{D}}_{N+2}^\sharp
\\&\quad\ls\sqrt{ \se{2N}   } \sd{N+2} +\bar{\mathcal{D}}_{N+2}^t +\bar{\mathcal{D}}_{N+2}^{\ast}.
\end{split}
\end{equation}
\end{prop}
\begin{proof}
We let $n$ denote either $2N$ or $N+2$ throughout the proof.
Applying $\partial^\alpha$ with $\al\in \mathbb{N}^{2}$ so that $ |\al|\le 2n$ to the second equation of \eqref{perturb} and then taking the dot product with $\p^\al \eta$, since $
\eta=0\text{ on }\Sigma_{m,\ell}$  and  $\jump{\eta} =0\text{ on }\Sigma
$, by \eqref{phhicon}, we find that
\begin{equation} \label{p1p1111224}
\begin{split}
 &\int_\Omega \rho \dt \partial^\alpha    u \cdot \partial^\alpha \eta
  + \int_{\Omega}\frac{\mu}{2} \mathbb{D} \partial^\alpha u
:\mathbb{D} \partial^\alpha  \eta  +\int_{\Omega}  \abs{(\bar B\cdot \nabla)\partial^\alpha  \eta }^2- \int_{\Sigma} \rj g \abs{\partial^\alpha  \eta_3}^2
 \\&\quad=  \int_\Omega   \p^\al \eta  \cdot  \p^\al (G^1-\mu\nabla G^2) +\int_\Omega   \p^\al p  \p^\al \Phi
 +\int_\Sigma \p^\al \eta \cdot \p^\al G^3   .
\end{split}
\end{equation}
Since $\dt\eta=u$, we have
\begin{equation}
\int_\Omega  \rho \dt(\partial^\alpha  u)\cdot \partial^\alpha  \eta=\frac{d}{dt}\int_\Omega  \rho   \partial^\alpha  u \cdot \partial^\alpha  \eta-\int_\Omega  \rho  \partial^\alpha  u \cdot \dt \partial^\alpha  \eta=\frac{d}{dt}\int_\Omega  \rho  \partial^\alpha  u \cdot \partial^\alpha  \eta-\int_\Omega  \rho  \abs{\partial^\alpha  u}^2
\end{equation}
and
\begin{equation}
   \int_\Omega \frac{\mu }{2}  \sg  \p^\al u:  \sg \p^\al \eta =   \int_\Omega \frac{\mu  }{2} \sg \p^\al \dt\eta:  \sg  \p^\al \eta
 =   \frac{1}{2}\frac{d}{dt} \int_\Omega \frac{\mu }{2}  \abs{\sg  \p^\al  \eta}^2.
\end{equation}
Note that
\begin{equation}
 \int_\Omega  \rho  \abs{\partial^\alpha  u}^2\ls \sdb{n}^t+\sdb{n}^\ast.
\end{equation}

We now estimate the terms on the right hand side of \eqref{p1p1111224}. For $\al=0$, we easily have
\begin{equation}
   \int_\Omega  \eta  \cdot   (G^1-\mu\nabla G^2)+\int_\Sigma  \eta \cdot  G^3
 \ls \sqrt{ \sd{3} } \sqrt{ \se{3}\sd{3}  } .
\end{equation}
The pressure term is needed much more care; by \eqref{Qstr2}, \eqref{Qstr111} and \eqref{Phe_0}, we obtain
\begin{equation}
\begin{split}
\int_\Omega     p   \Phi&= \int_\Omega     p   {\rm div}\Psi  = -\int_\Sigma  \jump{  p}   \Psi_3-\int_\Omega  \nabla  p\cdot   \Psi
\\&\le \abs{\jump{  p}}_0\abs{\Psi_3}_0+\norm{\nabla p}_0\norm{\Psi}_0\ls \sqrt{ \sd{3} } \sqrt{ \se{3}\sd{3}  } .
\end{split}
\end{equation}
We then turn to the case $\al\neq 0$. We first consider the case $n=2N$. Similarly as \eqref{i_de_14}--\eqref{i_de_16},
\begin{equation}
\begin{split}
\int_\Omega   \p^\al \eta \cdot  \p^\al (G^1-\mu\nabla G^2)+\int_\Sigma \p^\al \eta \cdot \p^\al G^3     &\ls  \norm{\p^{\alpha}  \eta}_{1} \left( \norm{ G^1 }_{4N-1}+ \norm{   G^2 }_{4N}+ \abs{   G^3 }_{4N-1/2}\right)
 \\&\ls \sqrt{ \f } \sqrt{ \se{N+2}(\sd{2N}+\mathcal{J}_{2N} +\f)  }.
 \end{split}
\end{equation}
Similarly as \eqref{i_de_15}, by using instead \eqref{p_G_e_001''},
\begin{equation}
  \int_\Omega  \p^\al p \p^\al \Phi
\le \norm{\nabla_\ast p}_{4N-1}  \norm{    \Phi }_{4N}
 \ls \sqrt{ \mathcal{J}_{2N} }\sqrt{ \se{N+2}(\sd{2N}+\mathcal{J}_{2N} +\f)  } .
\end{equation}
We now consider the case $n=N+2$. Similarly as \eqref{i_de_4'}--\eqref{i_de_16;'},
\begin{equation}
\begin{split}
\int_\Omega   \p^\al \eta \cdot  \p^\al (G^1-\mu\nabla G^2)+\int_\Sigma \p^\al \eta \cdot \p^\al G^3     &\ls  \norm{  \eta}_{2(N+2)} \left( \norm{ G^1 }_{4N }+ \norm{   G^2 }_{4N+1}+ \abs{   G^3 }_{4N+1/2}\right)
 \\&\ls \sqrt{ \sd{N+2}} \sqrt{ \se{2N}\sd{N+2} }.
 \end{split}
\end{equation}
For the pressure term, similarly as \eqref{i_de_5'0}--\eqref{i_de_5'}, by using instead \eqref{p_G_e_002''}, we obtain
\begin{equation}
  \int_\Omega  \p^\al p \p^\al G^2
\ls \norm{ \nabla_\ast p}_{2(N+2)-2}  \norm{\Phi }_{2(N+2)+1}
 \ls \sqrt{ \sd{N+2} }\sqrt{ \se{2N}\sd{N+2}}  .
\end{equation}

Consequently, \eqref{p_u_e_00'132} and \eqref{p_u_e_00'12132} follow by collecting the estimates, summing over such $\alpha$.
\end{proof}

\subsubsection{Conclusion}

Note that the previous energy evolution is derived for any $\bar B$. Now assuming $\abs{\bar B_3}>\mc$, we can combine these estimates to conclude the following proposition.
\begin{prop}\label{conclusion}
Assume $\abs{\bar B_3}>\mc$. For $n=N+2$ or $2N$, there exist an energy $\seb{n}$ with
\begin{equation}\label{ebarn}
 \seb{n}\simeq \sum_{j=1}^{n}\left(\norm{\dt^j u}_{0}^2+\norm{(\bar B\cdot\nabla)\dt^j\eta}_{0}^2\right)
+\norm{\nabla_\ast u}_{0,2n-1 }^2+\norm{(\bar B\cdot\nabla) \nabla_\ast \eta}_{0,2n-1 }^2+ \norm{\eta}_{1,2n}^2
\end{equation}
and the corresponding dissipation
\begin{equation}
 \sdb{n}:=\sum_{j=0}^{n}\norm{    \dt^j u}_{1}^2+\norm{\nabla_\ast u}_{1,2n-1 }^2+\norm{(\bar B\cdot \nabla)  \eta }_{0,2n}^2
\end{equation}
such that
\begin{equation}\label{conclusion2N}
 \frac{d}{dt}\left(\bar{\mathcal{E}}_{2N}-  \mathfrak{B}_{2N}\right)+\bar{\mathcal{D}}_{2N}
 \ls \sqrt{ \se{N+2}  }( \sd{2N} + \mathcal{J}_{2N} +\f)
\end{equation}
and
\begin{equation}\label{conclusionN+2}
 \frac{d}{dt}\left(\bar{\mathcal{E}}_{N+2}-  \mathfrak{B}_{N+2}\right)+\bar{\mathcal{D}}_{N+2}
 \ls \sqrt{ \se{2N}  } \sd{N+2}.
\end{equation}
\end{prop}
\begin{proof}
We let $n$ denote either $2N$ or $N+2$ through the proof, and we use the compact notation
\begin{equation}\label{zn}
\mathcal{Z}_n \text{ with }\mathcal{Z}_{2N}:=\sqrt{ \se{N+2}  }(\sd{ 2N}+  \mathcal{J}_{2N} +\f)\text{ and }\mathcal{Z}_{N+2}:=\sqrt{ \se{2N}  }\sd{ N+2} .
\end{equation}

We deduce from Propositions \ref{i_temporal_evolution  N}, \ref{p_upper_evolution  N12} and \ref{p_upper_evolution  N'132} that for any $\epsilon>0$,
\begin{equation}
\begin{split}
 & \frac{d}{dt}\left( \bar{\mathcal{E}}_{n}^t-  \mathfrak{B}_n+ \bar{\mathcal{E}}_{n}^*+\epsilon\left( \bar{\mathcal{E}}_{n}^\sharp+ \sum_{\substack{\al\in \mathbb{N}^2\\  |\al|\le 2n}}\int_\Omega \rho\partial^\alpha u  \cdot \partial^\alpha  \eta\right) \right) +  \bar{\mathcal{D}}_{n}^t+\bar{\mathcal{D}}_{n}^*+\epsilon \bar{\mathcal{D}}_{n}^\sharp
 \\&\quad \ls     \mathcal{Z}_n +\epsilon(\bar{\mathcal{D}}_{n}^t+\bar{\mathcal{D}}_{n}^{\ast} ).
  \end{split}
\end{equation}
Then for sufficiently small $ \epsilon>0$, the last two terms on the right hand side can be absorbed by the left hand side; if we define
\begin{equation}
 \bar{\mathcal{E}}_{n}:= \bar{\mathcal{E}}_{n}^t + \bar{\mathcal{E}}_{n}^*+\epsilon\left( \bar{\mathcal{E}}_{n}^\sharp+ \sum_{\substack{\al\in \mathbb{N}^2\\  |\al|\le 2n}}\int_\Omega \rho\partial^\alpha u  \cdot \partial^\alpha  \eta\right),
\end{equation}
then we conclude that
\begin{equation} \label{llaf}
\frac{d}{dt}\left( \bar{\mathcal{E}}_{n} -  \mathfrak{B}_n  \right) +  \bar{\mathcal{D}}_{n}^t+\bar{\mathcal{D}}_{n}^*+\epsilon \bar{\mathcal{D}}_{n}^\sharp
  \ls     \mathcal{Z}_n .
\end{equation}
Applying Lemma \ref{signsign} $(1)$ for $\abs{\bar B_3}>\mc$, together with Poincar\'e's and Korn's inequalities, we have that $\bar{\mathcal{D}}_{n}^t+\bar{\mathcal{D}}_{n}^*+\epsilon \bar{\mathcal{D}}_{n}^\sharp\simeq \bar{\mathcal{D}}_{n}$ and that $\bar{\mathcal{E}}_{n}$ satisfies \eqref{ebarn} by further choosing $\epsilon$ smaller if necessary. Consequently, \eqref{llaf} implies \eqref{conclusion2N} for $n=2N$ and \eqref{conclusionN+2} for $n=N+2$ by recalling \eqref{zn}.
\end{proof}
\subsection{Estimates via Stokes regularity}
Now we combine the previous estimates with the elliptic regularity theory of certain Stokes problems to improve the energy-dissipation estimates.

\subsubsection{Dissipation improvement}

We first consider the improvement of the dissipation estimates; the energy estimates of $\eta$ will be improved along the way.
\begin{prop}\label{p_upper_evolution  N'}
For $n\ge 3$, there exists an energy $ \mathrm{E}_n \simeq\norm{\eta}_{2n}^2$ such that
\begin{equation}\label{d2n}
\frac{d}{dt}\mathrm{E}_{2N}+\mathcal{D}_{2N}\ls  { \se{N+2}  }(\sd{2N} +  \mathcal{J}_{2N} +\f)+\bar{\mathcal{D}}_{2N}
\end{equation}
and
\begin{equation}\label{d2n+2}
\frac{d}{dt}\mathrm{E}_{N+2} +\mathcal{D}_{N+2}\ls { \se{2N}   } \sd{N+2}+\bar{\mathcal{D}}_{N+2} .
\end{equation}
\end{prop}
\begin{proof}
We let $n$ denote either $2N$ or $N+2$ throughout the proof, and we compactly write
\begin{equation}\label{n1}
 \mathcal{Y}_n =  \ns{ \bar{\nab}^{2n-1}_0 G^1}_{0}  +  \ns{ \bar{\nab}^{2n-1}_0  G^2}_{1}+  \as{ \bar{\nab}_{\ast0}^{\ 2n-1}  G^3}_{1/2} +  \ns{   \Phi}_{2n-1}  .
\end{equation}

We divide the proof into several steps.

{\bf Step 1: control terms with time derivatives}

Applying the time derivatives $\dt^{j},\ j=1,\dots,n-1$ to the equations \eqref{perturb}, we find that
\begin{equation}\label{stokesp1}
\begin{cases}
-\mu\Delta \dt^{j} u+\nabla \dt^{j} p =-\rho\dt^{j+1} u+(\bar B\cdot\nabla)^2 \dt^{j} \eta+\dt^{j} G^1& \text{in }
\Omega
\\ \diverge \dt^{j} u=\dt^{j} G^2 & \text{in }
\Omega
\\ \Lbrack \dt^j u\Rbrack =0,\ \Lbrack \dt^j pI-\mu\mathbb{D}(\dt^j u)\Rbrack e_3=\jump{\bar B_3(\bar B\cdot\nabla)\dt^{j} \eta}+\rj g\dt^{j} \eta_3 e_3+\dt^j G^3 &\hbox{on }\Sigma
\\ \dt^{j} u=0 &\text{on }\Sigma_{m,\ell}.
\end{cases}
\end{equation}
Applying the elliptic estimates \eqref{cSresult} of Lemma \ref{cStheorem} with $r=2n-2j+1 \ge 3$ to the problem \eqref{stokesp1} for $j=1,\dots,n-1$, by the trace theory, we obtain
\begin{equation}\label{ffes1}
\begin{split}
 &\norm{\dt^j  u  }_{2n-2j+1}^2 + \norm{\nab \dt^j  p  }_{2n-2j-1}^2+ \as{  \jump{\dt^j  p } }_{2n-2j-1/2}
 \\  &\quad\ls
\norm{\dt^{j+1} u   }_{2n-2j-1 }^2+\norm{\nabla^2 \dt^{j }\eta}_{2n-2j-1 }^2+ \norm{ \dt^j G^1   }_{2n-2j-1}^2
+ \norm{\dt^j  G^2  }_{2n-2j}^2
\\  &\qquad
+\as{\nabla\dt^{j} \eta    }_{2n-2j-1/2}+\as{\dt^{j}\eta_3   }_{2n-2j-1/2}+ \as{ \dt^j G^3   }_{2n-2j-1/2}
\\&\quad\ls   \norm{\dt^{j+1} u   }_{2n-2(j+1)+1 }^2 +\norm{\dt^{j }\eta  }_{2n-2 j +1 }^2 +  \y_n.
\end{split}
\end{equation}
A simple induction on \eqref{ffes1} yields, since $\dt \eta=u$
\begin{equation}\label{eses110}
\begin{split}
 & \sum_{j=1}^{n}\norm{ \dt^j u  }_{2n-2j+1}^2+  \sum_{j=1}^{n-1}\norm{\nab \dt^j  p  }_{2n-2j-1}^2+\sum_{j=1}^{n-1}\as{ \jump{ \dt^j  p } }_{2n-2j-1/2}
 \\&\quad \ls \ns{\dt^{n} u}_{1}+ \sum_{j=1}^{n-1}\norm{\dt^{j }\eta  }_{2n-2 j +1 }^2    +\y_n
\\& \quad  =  \ns{\dt^{n} u}_{1}+ \sum_{j=1}^{n-1}\norm{\dt^{j-1} u   }_{2n-2(j-1)-1 }^2 +\y_n
\\& \quad = \ns{\dt^{n} u}_{1}+ \sum_{j=0}^{n-2}\norm{\dt^{j } u   }_{2n-2j-1 }^2     + \y_n.
\end{split}
\end{equation}
Using the Sobolev interpolation and Young's inequality, we may improve \eqref{eses110} to be
\begin{equation}\label{eses11}
\begin{split}
 & \sum_{j=1}^{n}\norm{ \dt^j u  }_{2n-2j+1}^2+  \sum_{j=1}^{n-1}\norm{\nab \dt^j  p  }_{2n-2j-1}^2+\sum_{j=1}^{n-1}\as{  \jump{\dt^j  p } }_{2n-2j-1/2}
 \\&\quad \ls   \ns{\dt^{n} u}_{1}+  \norm{  u   }_{2n-1 }^2+ \sum_{j=1}^{n-2}\norm{\dt^{j } u   }_{0 }^2     + \y_n
    \\&\quad\ls\norm{  u   }_{2n-1 }^2+\bar{\mathcal{D}}_{n}  + \y_n .
\end{split}
\end{equation}

{\bf  Step 2: control terms without time derivatives}

Note that we can not use the Stokes problem \eqref{stokesp1} with $j=0$ as above since we have not controlled $(\bar B\cdot\nabla)^2\eta$ yet. Our observation is that we have certain control of the horizontal derivatives of $\eta$ in $\bar{\mathcal{D}}_n^\sharp$; we may write $(\bar B\cdot \nabla)^2\eta=\bar B_3^2\Delta\eta-\bar B_3^2\Delta_\ast\eta+(\bar B_\ast\cdot\nabla_\ast)^2\eta+2  (\bar B_3\pa_3) (\bar B_\ast\cdot\nabla_\ast)\eta$. This motivates us to introduce the quantity $w=  u+  \frac{ \bar B_3^2}{\mu} \eta$ and we deduce from \eqref{perturb} that, using \eqref{phhicon},
\begin{equation}\label{stokesp2}
\begin{cases}
-\mu\Delta w+\nabla p= -\bar B_3^2\Delta_\ast\eta+(\bar B_\ast\cdot\nabla_\ast)^2\eta+2  (\bar B_3\pa_3) (\bar B_\ast\cdot\nabla_\ast)\eta-\rho \partial_t u+G^1& \text{in }
\Omega
\\ \diverge w=  G^2+ \frac{ \bar B_3^2}{\mu}  \Phi& \text{in }
\Omega
\\ w=0 &\text{on }\Sigma_{m,\ell}.
\end{cases}
\end{equation}
However, since the two viscosities $\mu_\pm$ generally are different and the difference would prevent us from obtaining a ``good" jump boundary conditions for $w$ on $\Sigma$. Our second observation is that we can get higher regularity estimates of $u$ and $\eta$ and hence $w$ on the boundary $\Sigma$ from $\bar{\mathcal{D}}_{n}$. Indeed, since $\Sigma$ is flat, we may use the definition of Sobolev norms on $\Sigma$, the trace theory and Korn's inequality to obtain, by the definitions of $\sdb{n}$,
\begin{equation}\label{boure1}
\as{  u}_{2n+1/2}   \lesssim
 \ns{  u   }_{1,2n}
 \lesssim  \bar{\mathcal{D}}_{n} .
\end{equation}
Similarly, we may use instead Poincar\'e's inequality to have
\begin{equation}\label{boure2}
\as{ \eta}_{2n}  \ls \ns{(\bar B\cdot \nabla)   \eta   }_{0,2n} \lesssim \bar{\mathcal{D}}_{n}.
\end{equation}
This motivates us to use the elliptic estimates with Dirichlet boundary conditions.

For $j=2,\dots,2n$, applying $\p^\al$ with $\al\in \mathbb{N}^2$ so that $|\al|\le 2n-j$ to the problem \eqref{stokesp2}, and then applying the elliptic estimates \eqref{stokes es} of Lemma \ref{i_linear_elliptic2} with $j\ge 2$ to the resulting problems in $\Omega_+$ and $\Omega_-$ separately and using \eqref{boure1} and \eqref{boure2}, summing over such $\al$, we obtain
\begin{equation}\label{hihoh}
\begin{split}
&\norm{w}_{j,2n-j}^2+\norm{\nabla p}_{j-2,2n-j}^2
\\&\quad\ls \norm{ \nabla_\ast^2\eta}_{j-2,2n-j }^2+\norm{\pa_3\nabla_\ast\eta}_{j-2,2n-j}^2+\norm{\p_t u}_{j-2,2n-j}^2
\\&\qquad+\norm{G^1}_{j-2,2n-j}^2+\norm{G^2}_{j-1,2n-j}^2+\norm{\Phi}_{j-1,2n-j}^2+\as{w}_{ j-1/2+2n-j}
\\&\quad\ls  \norm{\eta}_{j-1,2n-j+1}^2+\norm{\p_t u}_{2n-2}^2 +\norm{G^1}_{2n-2}^2+\norm{G^2}_{2n-1}^2
 +\norm{\Phi}_{2n-1}^2
+\as{w}_{  2n-1/2}
\\&\quad\ls  \norm{ \eta}_{j-1,2n-j+1}^2+\norm{\p_t u}_{2n-2}^2 +\as{u}_{  2n-1/2} +\as{\eta}_{  2n-1/2}  +\y_n.
\end{split}
\end{equation}
It is a key to note that
\begin{equation}
 \norm{w}_{j,2n-j}^2 = \norm{     u+     \frac{  \bar B_3^2}{\mu} \eta  }_{j,2n-j}^2
 = \norm{      u  }_{j,2n-j}^2+ \frac{ \bar B_3^4}{\mu^2}\norm{       \eta  }_{j,2n-j}^2+ \frac{  \bar B_3^2}{\mu} \frac{d}{dt}\norm{     \eta  }_{j,2n-j}^2.
\end{equation}
Therefore, we deduce that for $j=2,\dots, n$,
\begin{equation}\label{llalfl}
\begin{split}
& \frac{d}{dt}  \norm{     \eta  }_{j,2n-j}^2+ \norm{        \eta  }_{j,2n-j}^2+\norm{      u  }_{j,2n-j}^2+\norm{\nabla p}_{j-2,2n-j}^2
\\&\quad\ls \norm{ \eta}_{j-1,2n-(j-1)}^2+\norm{\p_t u}_{2n-2}^2+\as{u}_{  2n-1/2} +\as{\eta}_{  2n-1/2}   +\y_n.
\end{split}
\end{equation}
Multiplying \eqref{llalfl} by $\epsilon^{j}$ with $0<\epsilon\ll 1$ and summing over  $ j=2,\dots, n$, we obtain
\begin{equation}\label{q1345}
\begin{split}
& \frac{d}{dt}\sum_{j=2}^{2n}\epsilon^{j }\norm{     \eta  }_{j,2n-j}^2+\sum_{j=2}^{2n}\epsilon^{j }\left(\norm{        \eta  }_{j,2n-j}^2+\norm{      u  }_{j,2n-j}^2+\norm{\nabla p}_{j-2,2n-j}^2\right)
\\&\quad\ls \sum_{j=2}^{2n}\epsilon^{j } \norm{ \eta}_{j-1,2n-(j-1)}^2+\norm{\p_t u}_{2n-2}^2+\as{u}_{  2n-1/2} +\as{\eta}_{  2n-1/2}     +\y_n
\\&\quad=\epsilon \sum_{j=2}^{2n}\epsilon^{j-1} \norm{ \eta}_{j-1,2n-(j-1)}^2+\norm{\p_t u}_{2n-2}^2+\as{u}_{  2n-1/2} +\as{\eta}_{  2n-1/2}    +\y_n
\\&\quad=\epsilon \sum_{j=1}^{2n-1}\epsilon^{j } \norm{ \eta}_{j ,2n- j }^2+\norm{\p_t u}_{2n-2}^2+\as{u}_{  2n-1/2} +\as{\eta}_{  2n-1/2}     +\y_n.
\end{split}
\end{equation}
Hence if we define
\begin{equation}
\mathrm{E}_n:=\sum_{j=2}^{2n}\epsilon^{j}\norm{     \eta  }_{j,2n-j}^2,
\end{equation}
then $\mathrm{E}_n\simeq\norm{\eta}_{2n}^2$. For sufficiently small $\epsilon>0$, \eqref{q1345} implies in particular that
\begin{equation}
\begin{split}
& \frac{d}{dt}\mathrm{E}_n
+ \norm{    \eta  }_{2n}^2+\norm{    u  }_{2n}^2+\norm{\nabla p}_{2n-2}^2
\\&\quad \ls \norm{ \eta}_{1,2n-1}^2+\norm{\p_t u}_{2n-2}^2  +\as{u}_{  2n-1/2} +\as{\eta}_{  2n-1/2} +\y_n.
\end{split}
\end{equation}
Note then that we can also get the boundary regularity of $\jump{p}$. For this, we use the third component of the dynamic boundary conditions in \eqref{perturb}:
\begin{equation} \label{pb2}
\Lbrack   p \Rbrack  =2\Lbrack   \mu\p_3u_3\Rbrack+\jump{ \bar B_3  (\bar B\cdot\nabla) \eta_3}+ \rj g  \eta_3 + G^3_{3}  \text{ on } \Sigma.
\end{equation}
Then we have
\begin{equation}\label{pp2}
\begin{split}
\as{   \jump{  p } }_{2n -3/2}
&\lesssim \as{\partial_3  u_3}_{2n -3/2} + \as{\nabla  \eta }_{2n -3/2}+ \as{ \eta_3}_{2n -3/2}  +\as{  G^3_3}_{2n -3/2}
\\
&\lesssim  \norm{   u_3}_{2n }^2+  \norm{   \eta }_{2n }^2 + \y_n .
\end{split}
\end{equation}
Hence, we obtain
\begin{equation} \label{eses22}
\begin{split}
& \frac{d}{dt}\mathrm{E}_n
+ \norm{    \eta  }_{2n}^2+\norm{    u  }_{2n}^2+\norm{\nabla p}_{2n-2}^2+\as{   \jump{  p } }_{2n -3/2}
\\&\quad \ls \norm{ \eta}_{1,2n-1}^2+\norm{\p_t u}_{2n-2}^2  +\as{u}_{  2n-1/2} +\as{\eta}_{  2n-1/2} +\y_n.
\end{split}
\end{equation}

Finally, note that
\begin{equation} \label{eses2290}
\begin{split}
  \norm{ \eta}_{1,2n-1}^2&\ls \norm{ \eta}_{0,2n}^2 +\norm{\pa_3 \eta}_{0,2n-1}^2
  =\norm{ \eta}_{0,2n}^2 +\frac{1}{\bar B_3^2}\norm{\bar B_3\pa_3 \eta}_{0,2n-1}^2
  \\& = \norm{ \eta}_{0,2n}^2 +\frac{1}{\bar B_3^2}\norm{(\bar B\cdot\nabla) \eta-(\bar B_\ast\cdot\nabla_\ast) \eta}_{0,2n-1}^2
  \\& \ls  \norm{ \eta}_{0,2n}^2 + \norm{(\bar B\cdot\nabla) \eta }_{0,2n-1}^2\ls \bar{\mathcal{D}}_n .
    \end{split}
\end{equation}
This together with \eqref{boure1}--\eqref{boure2} allows us to deduce from \eqref{eses22} that
\begin{equation} \label{eses229012}
\frac{d}{dt}\mathrm{E}_n
+ \norm{    \eta  }_{2n}^2+\norm{    u  }_{2n}^2+\norm{\nabla p}_{2n-2}^2
  \ls  \norm{\p_t u}_{2n-2}^2+\bar{\mathcal{D}}_n    +\y_n.
\end{equation}

{\bf  Step 3: combine the estimates}

We may now combine the estimates \eqref{eses11} and \eqref{eses229012} to get
\begin{equation}\label{ffes2}
\begin{split}
&
\frac{d}{dt}\mathrm{E}_n
+ \norm{    \eta  }_{2n}^2+  \norm{    u  }_{2n}^2+\norm{\nabla p}_{2n-2}^2+\as{   \jump{  p } }_{2n -3/2}
\\&\quad+ \sum_{j=1}^{n}\norm{ \dt^j u  }_{2n-2j+1}^2+  \sum_{j=1}^{n-1}\norm{\nab \dt^j  p  }_{2n-2j-1}^2 +\sum_{j=1}^{n-1}\as{  \jump{ \dt^j p } }_{2n-2j-1/2}
  \\&\quad\ls   \norm{  u   }_{2n -1 }^2   +\bar{\mathcal{D}}_n +\y_n.
  \end{split}
\end{equation}
Using the Sobolev interpolation and Young's inequality, we may improve \eqref{ffes2} to be
\begin{equation}\label{claim2}
\begin{split}
&
\frac{d}{dt}\mathrm{E}_n
+   \norm{    u  }_{2n}^2+\norm{    \eta  }_{2n}^2+\norm{\nabla p}_{2n-2}^2+\as{   \jump{  p } }_{2n -3/2}
\\&\quad + \sum_{j=1}^{n}\norm{ \dt^j u  }_{2n-2j+1}^2+  \sum_{j=1}^{n-1}\norm{\nab \dt^j  p  }_{2n-2j-1}^2+\sum_{j=1}^{n-1}\as{  \jump{\dt^j  p } }_{2n-2j-1/2}
 \\&\quad \ls   \norm{  u   }_{0}^2  +\bar{\mathcal{D}}_n  +\y_n\ls  \bar{\mathcal{D}}_n   +\y_n.
  \end{split}
\end{equation}

Adding $\bar{\mathcal{D}}_n$ to both sides of \eqref{claim2} implies that
\begin{equation}\label{dth_7}
\frac{d}{dt}\mathrm{E}_n
+  \mathcal{D}_n   \ls  \bar{\mathcal{D}}_n  +\y_n.
\end{equation}
Using \eqref{p_G_e_001} and \eqref{p_G_e_001''} to estimate
$\mathcal{Y}_{2N}\lesssim  { \se{N+2}  }(\sd{2N} +  \mathcal{J}_{2N} +\f)$, we obtain \eqref{d2n} from \eqref{dth_7} with $n=2N$; using \eqref{p_G_e_002} and \eqref{p_G_e_002''} to estimate $\mathcal{Y}_{N+2}\lesssim   {\mathcal{E}}_{2N}  {\mathcal{D}}_{N+2} $, we obtain \eqref{d2n+2} from \eqref{dth_7} with $n=N+2$.
\end{proof}

\subsubsection{Energy improvement}

Now we improve the energy estimates.

\begin{prop}\label{e2nic}
For $n=2N$ or $N+2$, it holds that
\begin{equation}\label{e2n}
{\mathcal{E}}_{n}  \lesssim
\bar{\mathcal{E}}_n  +\mathrm{E}_{n}+ \se{N+2}\se{n}.
\end{equation}
\end{prop}
\begin{proof}
We let $n$ denote either $2N$ or $N+2$ throughout the proof, and we compactly write
\begin{equation}\label{n112}
 \mathcal{X}_n =   \ns{ \bar{\nab}^{2n-2}_0  G^1}_{0} +\ns{ \bar{\nab}^{2n-2}_0  G^2}_{1}   +\as{ \bar{\nab}^{\ 2n-2}_{\ast 0}  G^3}_{1/2} .
\end{equation}

Applying the elliptic estimates \eqref{cSresult} of Lemma \ref{cStheorem} with $r=2n-2j \ge 2$ to the problem \eqref{stokesp1} for $j=0,\dots,n-1$, by the trace theory, we obtain
\begin{equation}\label{ffes3}
\begin{split}
 &\norm{\dt^j  u  }_{2n-2j}^2 + \norm{\nabla \dt^j  p  }_{2n-2j-2}^2+ \abs{  \jump{ \dt^j p } }_{2n-2j-3/2}^2
 \\  &\quad\ls
\norm{\dt^{j+1} u   }_{2n-2j-2 }^2+\norm{  \nabla^2 \dt^{j }\eta}_{2n-2j-2 }^2+ \norm{ \dt^j G^1   }_{2n-2j-2}^2
+ \norm{\dt^j  G^2  }_{2n-2j-1}^2
 \\  &\qquad+\as{\nabla\dt^{j} \eta    }_{ 2n-2j-3/2} +
\as{\dt^{j}\eta_3   }_{2n-2j-3/2} + \as{ \dt^j G^3   }_{2n-2j-3/2}
\\&\quad\ls   \norm{\dt^{j+1} u   }_{2n-2(j+1) }^2 +\norm{   \dt^{j} \eta}_{2n-2j  }^2   +  \x_n.
\end{split}
\end{equation}
A simple induction on \eqref{ffes3} yields, since $\dt\eta=u$,
\begin{equation}\label{ffes4}
\begin{split}
&\sum_{j=0}^{n}\norm{ \dt^j u  }_{2n-2j }^2+ \sum_{j=0}^{n-1}\norm{\nabla \dt^j  p  }_{2n-2j-2}^2+ \sum_{j=0}^{n-1}\as{  \jump{\dt^j  p } }_{2n-2j-3/2}
\\ & \quad \ls \ns{\dt^{n} u}_{0} +\sum_{j=0}^{n-1}\norm{ \dt^j \eta  }_{2n-2j }^2 +  \x_n
\\&  \quad \le \ns{\dt^{n} u}_{0} +\norm{  \eta  }_{2n  }^2 +\sum_{j=1}^{n-1}\norm{ \dt^{j-1} u  }_{2n-2j }^2   + \x_n
\\& \quad \ls  \ns{\dt^{n} u}_{0}+\norm{  \eta  }_{2n  }^2  +\sum_{j=0}^{n-2}\norm{ \dt^{j} u  }_{2n-2j-2}^2      +  \x_n.
\end{split}
\end{equation}
Using the Sobolev interpolation and Young's inequality, we may improve \eqref{ffes4} to be
\begin{equation}\label{claim2'}
\begin{split}
&  \sum_{j=0}^{n}\norm{ \dt^j u  }_{2n-2j }^2+   \sum_{j=0}^{n-1}\norm{\nabla \dt^j  p  }_{2n-2j-2}^2+ \sum_{j=0}^{n-1}\as{ \jump{ \dt^j  p } }_{2n-2j-3/2}
\\&\quad \ls \ns{\dt^{n} u}_{0}+\norm{  \eta  }_{2n  }^2  + \sum_{j=0}^{n-2}\norm{ \dt^{j} u  }_{0}^2  +  \x_n
\ls   \bar{\mathcal{E}}_n   +\mathrm{E}_n +  \x_n.
\end{split}
\end{equation}

Adding $\bar{\mathcal{E}}_n$ to both sides of \eqref{claim2'} implies that
\begin{equation}\label{claim12}
{\mathcal{E}}_{n}  \lesssim
\bar{\mathcal{E}}_n+ \mathrm{E}_{n} +  \x_n.
\end{equation}
Using \eqref{p_G_e_0} to bound $\mathcal{X}_{n} \lesssim  \se{N+2}\se{n}$, we then conclude \eqref{e2n}.
\end{proof}
\begin{remark}
Note that it is crucial that $\mathrm{E}_{n}$ has been controlled in the improvement of dissipation estimates so that we can improve the energy estimates of $(u,p)$ without time derivatives as done in Proposition \ref{e2nic}.
\end{remark}

\subsection{Synthesis}

We now chain all the estimates derived previously to conclude the following.
\begin{prop}
For $n=N+2$ or $2N$, there exists an  energy $\tilde{ \mathcal{E}}_{n}\simeq\mathcal{E}_{n}$ such that
\begin{equation}\label{sys2n}
\frac{d}{dt}\tilde{ \mathcal{E}}_{2N}+ {\mathcal{D}}_{2N}  \ls     \sqrt{ \se{N+2}  }(  \mathcal{J}_{2N} +\f)
\end{equation}
and
\begin{equation}\label{sysn+2}
\frac{d}{dt}\tilde{ \mathcal{E}}_{N+2}+ {\mathcal{D}}_{N+2}  \le 0.
\end{equation}
\end{prop}
\begin{proof}
We let $n$ denote either $2N$ or $N+2$ through the proof, and we use the compact notation
\begin{equation}
\mathcal{Z}_n \text{ with }\mathcal{Z}_{2N}:=\sqrt{ \se{N+2}  }(\sd{ 2N}+  \mathcal{J}_{2N} +\f)\text{ and }\mathcal{Z}_{N+2}:=\sqrt{ \se{2N}  }\sd{ N+2} .
\end{equation}

We deduce from Propositions \ref{conclusion} and \ref{p_upper_evolution  N'} that for $0<\varepsilon\ll 1$,
\begin{equation}
 \frac{d}{dt}\left( \bar{\mathcal{E}}_{n}-\mathfrak{B}_n +\varepsilon \mathrm{E}_{n}\right)
  +  \bar{\mathcal{D}}_{n} +\varepsilon {\mathcal{D}}_{n}
 \ls     \mathcal{Z}_n +\varepsilon  \bar{\mathcal{D}}_{n},
\end{equation}
which implies
\begin{equation} \label{i_te_02n}
 \frac{d}{dt}\left( \bar{\mathcal{E}}_{n}-\mathfrak{B}_n +\varepsilon \mathrm{E}_{n}\right)
  +  \bar{\mathcal{D}}_{n} +\varepsilon {\mathcal{D}}_{n}
 \ls     \mathcal{Z}_n  .
\end{equation}

We now define
\begin{equation}
\tilde{ \mathcal{E}}_{n}:=\bar{\mathcal{E}}_{n}-\mathfrak{B}_n +\varepsilon \mathrm{E}_{n}.
\end{equation}
Recalling $\mathfrak{B}_n$ from \eqref{bnbn}, by \eqref{p_F_e_02}, we have
\begin{equation}
\abs{\mathfrak{B}_n}\ls\norm{\nabla \dt^{n-1} p}_0\norm{Q^{2,n}}_0+\abs{\jump{\dt^{n-1} p}}_0\abs{ Q^{2,n}_3}_0\ls \sqrt{\mathcal{E}_n}\sqrt{\mathcal{E}_{N+2}\mathcal{E}_n}=\sqrt{\mathcal{E}_{N+2}}\mathcal{E}_n.
\end{equation}
This together with Proposition \ref{e2nic} yields
\begin{equation}
\mathcal{E}_{n}\ls \bar{\mathcal{E}}_n  +\mathrm{E}_{n}+ \se{N+2}\se{n}= \tilde{ \mathcal{E}}_{n}+\mathfrak{B}_n+ \se{N+2}\se{n} \ls \tilde{ \mathcal{E}}_{n}+\sqrt{\mathcal{E}_{N+2}}\mathcal{E}_n,
\end{equation}
which implies that $\tilde{ \mathcal{E}}_{n}\simeq \mathcal{E}_{n}$ since $\sqrt{{\mathcal{E}}_{2N}}\le  \delta$ is small. We thus deduce \eqref{sys2n} and \eqref{sysn+2} from \eqref{i_te_02n} by recalling the notation $\mathcal{Z}_n$ and using again the smallness of $\sqrt{{\mathcal{E}}_{2N}}\le \delta$.
\end{proof}

\subsection{Global energy estimates}

In this subsection, we shall conclude our global energy estimates of the solution to \eqref{reformulationic}.

We begin with the estimates of $\f $ and $\mathcal{J}_{2N}$.

\begin{prop}\label{p_f_bound}
There exists a universal constant $0<\delta<1$ so that if $\mathcal{G}_{2N}(T)\le\delta$, then
\begin{equation}\label{grow1}
\f(t)
  \ls  \f(0)+  \sup_{0\le r\le t}\mathcal{E}_{2N}(r)+ \int_0^t \mathcal{D}_{2N} \text{ for all
}0\le t\le T
\end{equation}
and for any $\vartheta>0$,
\begin{equation}\label{grow2}
  \int_0^t \frac{\f+\mathcal{J}_{2N}}{(1+r)^{1+\vartheta}}dr
  \ls  \f(0)+  \sup_{0\le r\le t}\mathcal{E}_{2N}(r)+ \int_0^t \mathcal{D}_{2N} \text{ for all
}0\le t\le T.
\end{equation}
\end{prop}
\begin{proof}
Following the arguments lead to \eqref{eses22} (basically, starting with replacing $\norm{\cdot}_{j,2n-j}^2$ by $\norm{\cdot}_{j,4N+1-j}^2$ in \eqref{hihoh} with $j=2,\dots,4N+1$), we deduce that there exists an energy $\tilde{\mathcal{F}}_{2N}\simeq\f$ such that
\begin{equation}
\begin{split}
 \dtt\tilde{\mathcal{F}}_{2N}
+  \f +\mathcal{J}_{2N}  \ls & \norm{\eta}_{1,4N}^2+ \norm{\p_t u}_{4N-1}^2+\as{u}_{  4N+1/2} +\as{\eta}_{  4N+1/2}
 \\&+ \norm{G^1}_{4N-1}^2+\norm{G^2}_{4N}^2+\norm{\Phi}_{4N}^2.
 \end{split}
\end{equation}
Note that
\begin{equation}
\as{u}_{  4N+1/2}\ls \ns{u}_{1,4N}\ls \bar{\mathcal{D}}_{2N} \text{ and }\as{\eta}_{  4N+1/2}\ls \ns{\eta}_{1,4N}\ls \bar{\mathcal{E}}_{2N}.
\end{equation}
On the other hand, we use \eqref{p_G_e_001} and \eqref{p_G_e_001''} to estimate
\begin{equation}
\norm{G^1}_{4N-1}^2+\norm{G^2}_{4N}^2+\norm{\Phi}_{4N}^2\lesssim  { \se{N+2}  }(\sd{2N} +  \mathcal{J}_{2N} +\f).
\end{equation}
Then we have
\begin{equation}\label{109}
 \dtt\tilde{\mathcal{F}}_{2N}
+  \f +\mathcal{J}_{2N}  \ls  \bar{\mathcal{E}}_{2N}+\bar{\mathcal{D}}_{2N}+ \norm{\p_t u}_{4N-1}^2
 + { \se{N+2}  } (\sd{2N} +  \mathcal{J}_{2N} +\f) .
\end{equation}
This further implies, since $\se{N+2}(t) \le \delta^2$ is small,
\begin{equation}\label{1091}
 \dtt\tilde{\mathcal{F}}_{2N}
+  \f +\mathcal{J}_{2N}  \ls   \mathcal{E}_{2N}+ {\mathcal{D}}_{2N}   .
\end{equation}

We now employ the time weighted analysis on \eqref{1091}. First, applying the Gronwall type analysis on \eqref{1091} yields that
\begin{equation}\label{1001}
\begin{split}
  \f(t)  & \ls  \f(0) e^{-Ct}+\int_0^t e^{-C(t-r)}\left(  \se{2N}(r)+ \sd{2N}(r) \right)dr
  \\& \ls \f(0) e^{-Ct}+\sup_{0\le r\le t}\mathcal{E}_{2N}(r) \int_0^t e^{-C(t-r)}dr+  \int_0^t \sd{2N} ,
  \end{split}
\end{equation}
which in particular yields \eqref{grow1}.

On the other hand, multiplying \eqref{1091} by $(1+t)^{-1-\vartheta}$ for any $\vartheta>0$, we obtain
\begin{equation}\label{1110}
 \dtt \left( \frac{\tilde{\mathcal{F}}_{2N}}{(1+t)^{1+\vartheta}}\right)+(1+\vartheta)\frac{\tilde{\mathcal{F}}_{2N}}{(1+t)^{2+\vartheta}}
+  \frac{\f+\mathcal{J}_{2N}}{(1+t)^{1+\vartheta}} \ls   \frac{\se{2N}}{(1+t)^{1+\vartheta}}+  \frac{\sd{2N}}{(1+t)^{1+\vartheta}} .
\end{equation}
 Integrating \eqref{1110} directly in time yields in particular \eqref{grow2}.
\end{proof}

 Now we show the boundedness of $\mathcal{E}_{2N} +\int_0^t\mathcal{D}_{2N} $.

\begin{prop} \label{Dgle}
There exists a universal constant $0<\delta<1$ so that if $\mathcal{G}_{2N}(T)\le\delta$, then
\begin{equation}\label{Dg}
\mathcal{E}_{2N} (t)+\int_0^t\mathcal{D}_{2N} \lesssim
\mathcal{E}_{2N} (0) + \mathcal{F}_{2N}(0)  \text{ for all
}0\le t\le T.
\end{equation}
\end{prop}
\begin{proof}
Integrating \eqref{sys2n} directly in time, we find that
\begin{equation}
  {\mathcal{E}}_{2N}(t)
+ \int_0^t{\mathcal{D}}_{2N}
\ls {\mathcal{E}}_{2N}(0)+\int_0^t \sqrt{ \se{N+2}  }(  \mathcal{J}_{2N} +\f)
\end{equation}
Since
\begin{equation}
\sqrt{ \se{N+2}(t)}\le \mathcal{G}_{2N}(T) (1+t)^{-N+2}\le \delta (1+t)^{-N+2},
\end{equation}
by the estimates \eqref{grow2} of Proposition \ref{p_f_bound}, we deduce
\begin{equation}
\begin{split}
  {\mathcal{E}}_{2N}(t)
+ \int_0^t{\mathcal{D}}_{2N}
&\ls {\mathcal{E}}_{2N}(0)+ \int_0^t  {\delta}(1+r)^{-N+2}(  \mathcal{J}_{2N}(r) +\f(r))dr
\\& \ls {\mathcal{E}}_{2N}(0)+  { \delta } \left(\f(0)+  \sup_{0\le r\le t}\mathcal{E}_{2N}(r)+ \int_0^t \mathcal{D}_{2N}\right)
.
 \end{split}
\end{equation}
Here we have used the fact that $N-2\ge 1+\vartheta$ by choosing $0<\theta\le N-3$ since $N\ge 4$. This proves \eqref{Dg} since $\delta$ is small.
\end{proof}

It remains to show the decay estimates of
$\mathcal{E}_{N+2}$.

\begin{prop} \label{decaylm}
There exists a universal constant $0<\delta<1$ so that if $\mathcal{G}_{2N}(T)\le\delta$, then
\begin{equation}\label{n+2}
(1+t)^{2N-4} \mathcal{E}_{N+2} (t)\lesssim
\mathcal{E}_{2N} (0)+ \mathcal{F}_{2N}(0) \ \text{for all
}0\le t\le T.
\end{equation}
\end{prop}
\begin{proof}
We will use \eqref{sysn+2} to derive the decay estimates. For this, we shall estimate $\mathcal{E}_{N+2} $ in terms of $\mathcal{D}_{N+2} $. Notice that ${\mathcal{D}}_{N+2} $ can control every term in ${\mathcal{E}}_{N+2} $ except $ \ns{\eta}_{1,2(N+2)}$. The key point is to use the Sobolev interpolation as \cite{RG,GT_per,GT_inf}. Indeed, we first have that, by \eqref{eses2290},
\begin{equation} \label{intep0c}
\begin{split}
\ns{ \eta}_{1,2(N+2)}&\le  \norm{ \eta}_{1,2(N+2)-1} ^{2\theta} \norm{ \eta}_{1,4N}^{2(1-\theta)}
\\&\le ( {\mathcal{D}_{N+2} })^\theta({\mathcal{E}_{2N} })^{1-\theta},\text{
where }\theta=\frac{2N-4}{2N-3}.
\end{split}
\end{equation}
Hence, we may deduce
\begin{equation} \label{intepc}
{\mathcal{E}}_{N+2} \ls({\mathcal{D}}_{N+2} )^\theta({\mathcal{E}}_{2N} )^{1-\theta}+{\mathcal{D}}_{N+2} \ls({\mathcal{D}}_{N+2} )^\theta({\mathcal{E}}_{2N} )^{1-\theta} .
\end{equation}

Now since by Proposition \ref{Dgle},
\begin{equation}
\sup_{0\le r\le t}\mathcal{E}_{2N} (r)\lesssim
\mathcal{E}_{2N} (0) +\f(0):=\mathcal{M}_0,
\end{equation}
we obtain from  \eqref{intepc} that
\begin{equation} \label{u2c}
\tilde{\mathcal{E}}_{N+2} \lesssim\mathcal{E}_{N+2} \lesssim\mathcal{M}_0^{1-\theta}
(\mathcal{D}_{N+2} )^\theta.
\end{equation}
Hence by \eqref{sysn+2} and \eqref{u2c}, there exists some constant $C>0$
such that
\begin{equation}
\frac{d}{dt} \tilde{\mathcal{E}}_{N+2} +\frac{C}{\mathcal{M}_0^s}
(\tilde{\mathcal{E}}_{N+2} )^{1+s}\le 0,\ \text{ where } s =
\frac{1}{\theta}-1 = \frac{1}{2N-4}.
\end{equation}
Solving this differential inequality directly, we obtain
\begin{equation} \label{u3c}
\mathcal{E}_{N+2} (t)\ls \frac{\mathcal{M}_0}{(\mathcal{M}_0^s
+ s C( \mathcal{E}_{N+2} (0))^s t)^{1/s} }
{\mathcal{E}}_{N+2} (0).
\end{equation}
Using that ${\mathcal{E}}_{N+2} (0)\lesssim\mathcal{M}_0 $ and
the fact $1/s=2N-4>1$, we obtain from \eqref{u3c} that
\begin{equation}
{\mathcal{E}}_{N+2} (t)\lesssim
\frac{\mathcal{M}_0}{(1+sCt)^{1/s} }\lesssim
\frac{\mathcal{M}_0}{(1+t^{1/s}) } =
\frac{\mathcal{M}_0}{(1+t^{2N-4}) }.
\end{equation}
This directly implies \eqref{n+2}.
\end{proof}

Now we can arrive at our ultimate energy estimates for
$\mathcal{G}_{2N} $, that is, we present the

\begin{proof}[Proof of Theorem \ref{Ap}]
Theorem \ref{Ap} follows directly from the definition of
$\mathcal{G}_{2N} $ and Propositions
\ref{p_f_bound}--\ref{decaylm}.
\end{proof}

\section{Nonlinear instability for $\abs{\bar B_3}<\mathcal{M}_c$}\label{nonlinear instability}

In this section, we will prove Theorem \ref{maintheorem} for $\abs{\bar B_3}<\mathcal{M}_c$.
Since the linear instability has been established in Section \ref{linear theory}, the remaining in the proof of Theorem
\ref{maintheorem} is the passage from linear instability to nonlinear instability.

\subsection{Growth of solutions to the linear inhomogeneous equations}\label{growth}

In this subsection, we will show that $\lam$ given in Theorem \ref{growingmode} is the sharp growth rate of arbitrary solutions to the linearized problem \eqref{perturb_linear}. Since the spectrum of the linear operator is complicated, it is hard to obtain the largest growth rate of the solution operator in ``$L^2\rightarrow L^2$'' in the usual way. Instead, motivated by \cite{3GT2}, we can use careful energy estimates to show that $e^{\Lam t}$ is the sharp growth rate in a slightly weaker sense, say, for instance ``$H^2\rightarrow L^2$''. However, this will be done for strong solutions to the problem, and it may be difficult to apply directly to the nonlinear problem due to the issue of compatibility conditions of the initial and boundary data since the problem is defined in a domain with boundary.  We overcome this obstacle by proving the estimates for the growth in time of arbitrary solutions to the linear inhomogeneous equations:
\begin{equation}\label{linear ho}
\begin{cases}
 \dt\eta=u&\hbox{in }\Omega
\\\rho\partial_t u-\mu\Delta u+\nabla p- (\bar B\cdot\nabla)^2 \eta=G^1\quad&\hbox{in }\Omega
\\ \diverge u=G^2&\hbox{in }\Omega
\\ \Lbrack u\Rbrack=0,\quad \Lbrack pI-\mu\mathbb{D}u\Rbrack e_3-\jump{\bar B_3(\bar B\cdot\nabla)\eta}= \rj g\eta_3  e_3+G^3&\hbox{on }\Sigma
\\ u  =0 &\hbox{on }\Sigma_{m,\ell},
\end{cases}
\end{equation}
where $G^1,G^2$ and $G^3$ are given functions. Note that by the boundary conditions of $u$,
\begin{equation}\label{qstru1}
\int_\Omega G^2=\int_\Omega \diverge u=0.
\end{equation}
This allows us to adjust the divergence of $u$ to reduce the problem \eqref{linear ho} to be a divergence-free problem. Indeed, by Lemma \ref{div11}, for $G^2\in {H}^{r-1}(\Omega), r\ge 1$, thanks to \eqref{qstru1}, there exists $\bar u\in H_0^1(\Omega)\cap {H}^{r}(\Omega)$ so that $\diverge{\bar u}=G^2$ in $\Omega$ and
\begin{equation}\label{dives}
\norm{\bar u}_{r}\lesssim\norm{G^2}_{r-1}.
\end{equation}
Hence, defining $v=u-\bar{u}$, we have $\diverge v=0$. Furthermore, differentiating the second equation with respect to time and then eliminating
the $\eta$ terms by using the first equation, we may then switch the problem \eqref{linear ho} to  the following second-order formulation for $(v,\dt p)$:
\begin{equation}\label{second_order}
\begin{cases}
\rho \dt^2 v-\mu\Delta \dt v+\nabla \dt p- (\bar B\cdot\nabla)^2 v=\mathfrak{G}^1\quad&\hbox{in }\Omega
\\ \diverge \dt v=0&\hbox{in }\Omega
\\ \Lbrack \dt v\Rbrack=0,\quad \Lbrack \dt pI-\mu\mathbb{D}\dt v\Rbrack e_3-\jump{\bar B_3(\bar B\cdot\nabla)v}= \rj gv_3  e_3+\mathfrak{G}^3&\hbox{on }\Sigma
\\ \dt v  =0 &\hbox{on }\Sigma_{m,\ell},
\end{cases}
\end{equation}
where
\begin{equation}\label{vG1_def}
\mathfrak{G}^1=\dt G^1
- \rho \dt^2 \bar u+\mu\Delta \dt \bar u + (\bar B\cdot\nabla)^2 \bar u
\end{equation}
and
\begin{equation}\label{vG3_def}
\mathfrak{G}^3=\dt G^3+\Lbrack \mu\mathbb{D}\dt \bar u\Rbrack e_3+\jump{\bar B_3(\bar B\cdot\nabla)\bar u}+\rj g\bar u_3  e_3 .
\end{equation}

Our first result gives an energy and its evolution equation for solutions to \eqref{second_order}.
\begin{lem}\label{lin_en_evolve le}
Let $ v $ solve \eqref{second_order}. Then
\begin{equation}\label{energyidentity}
\begin{split}
 &\hal \frac{d}{dt}\left(  \int_\Omega \left( \rho \abs{ \dt v }^2  +\abs{(\bar B\cdot \nabla)   v }^2\right)- \int_{\Sigma} \rj g \abs{   v_3}^2\right)+ \int_\Omega \frac{\mu}{2}  \abs{\sg   \dt v }^2
 \\&\quad= \int_\Omega   \dt v \cdot   \mathfrak{G}^1 + \int_\Sigma \dt v \cdot \dt \mathfrak{G}^3 .
\end{split}
\end{equation}
\end{lem}
\begin{proof}
We multiply the first equation of $\eqref{second_order}$ by $\partial_{t} v$ and  then integrate by parts over $\Omega$. By using the other equations in $\eqref{second_order}$, we obtain \eqref{energyidentity}.
\end{proof}

The variational characterization of $\lam$ gives rise to the next result.
\begin{lem}\label{lin_en_bound}
Let $v\in \Hs$. Then
\begin{equation}\label{variation}
\hal\left(\int_\Omega\abs{(\bar B\cdot \nabla)   v }^2- \int_{\Sigma} \rj g \abs{   v_3}^2\right)
\ge -\frac{\lam^2}{2}\int_\Omega \rho \abs{ v }^2
- \frac{\lam}{2} \int_\Omega \frac{\mu}{2}  \abs{\sg    v }^2 .
\end{equation}
 \end{lem}
 \begin{proof}
Recalling the notations \eqref{E_def}--\eqref{E1_def} with $s=\lam$, we may rewrite
 \begin{equation}
\hal\left(\int_\Omega\abs{(\bar B\cdot \nabla)   v }^2- \int_{\Sigma} \rj g \abs{   v_3}^2\right)
=E_0(v)=E(v;\lam)-\lam E_1(v).
\end{equation}
Since $v\in \Hs$, by the variational characterization for $\lam$ of \eqref{mu_def} with $s=\lam$, we have
 \begin{equation}
E(v;\lam)\ge -\lam^2J(v).
\end{equation}
These two yield \eqref{variation}.
 \end{proof}

We may first show the estimates for the growth of solutions to \eqref{second_order}.
\begin{lem} \label{lineargrownth0}
Let  $v$ solve \eqref{second_order}. Then
\begin{equation}\label{result10}
\norm{v(t)}_{1}
  \ls  e^{ \Lam
t}\left(\norm{ v(0)}_1 +\norm{\partial_t v(0)}_0   \right)
+ \mathfrak{N}_\mathfrak{G}(t),
\end{equation}
where
\begin{equation}\label{result20}
\begin{split}
\mathfrak{N}_\mathfrak{G}(t) &=  e^{ \Lam
t}\left(\norm{ \mathfrak{G}^1(0)}_0 + \abs{\mathfrak{G}^3(0)}_{0} \right)
+ \sup_{0\le s\le t}  \left(\norm{ \mathfrak{G}^1(s)}_0 + \abs{\mathfrak{G}^3(s)}_{0} \right)
\\&\quad+\sqrt{ \int_0^t e^{2\Lam
(t-s)}   \left(\norm{\mathfrak{G}^1(s)}_0+\norm{\dt \mathfrak{G}^1(s)}_0+  \abs{\mathfrak{G}^3(s)}_{0}+\abs{\dt \mathfrak{G}^3(s)}_{0}\right)\norm{v(s)}_{1}}.
\end{split}
\end{equation}
\end{lem}

\begin{proof}
Integrating the result of Lemma \ref{lin_en_evolve le} in time, and then applying Lemma \ref{lin_en_bound}, we find that
\begin{equation}\label{j111}
\begin{split}
  & \hal\int_\Omega   \rho  \abs{\dt v(t)}^2+ \int_0^t\int_\Omega \frac{\mu}{2}  \abs{\sg   \dt v }^2
\\ &\quad\le K_0+\int_0^t\left( \int_\Omega \mathfrak{G}^1\cdot\dt v+ \int_\Sigma \mathfrak{G}^3\cdot \dt v\right)
-\hal\left(\int_\Omega\abs{(\bar B\cdot \nabla)   v }^2- \int_{\Sigma} \rj g \abs{   v_3}^2\right)
\\ &\quad\le K_0+ \int_0^t\left( \int_\Omega \mathfrak{G}^1\cdot\dt v+ \int_\Sigma \mathfrak{G}^3\cdot \dt v\right)
+\frac{\lam^2}{2}\int_\Omega \rho \abs{ v }^2
+ \frac{\lam}{2} \int_\Omega \frac{\mu}{2}  \abs{\sg    v }^2 ,
\end{split}
\end{equation}
where
\begin{equation}
  K_0 =  \hal\left(  \int_\Omega\left(  \rho \abs{ \dt v(0) }^2  +\abs{(\bar B\cdot \nabla)   v(0) }^2\right)- \int_{\Sigma} \rj g \abs{   v_3(0)}^2\right).
\end{equation}
For notational simplicity we introduce the norms
\begin{equation}
\norm{v}_\star^2:=\int_\Omega\bar\rho | v|^2\hbox{ and }\norm{v}_{\star\star}^2:= \int_\Omega \frac{\mu}{2}  \abs{\sg    v }^2
\end{equation}
and the corresponding inner-products given by  $\langle \cdot,\cdot\rangle_{\star}$ and $\langle \cdot,\cdot\rangle_{\star\star}$, respectively.  We may then compactly rewrite the inequality \eqref{j111} as
\begin{equation}\label{j4}
\frac{1}{2}\norm{\partial_t  v(t)}_\star^2 + \int_0^t\norm{\partial_t  v(s)}_{\star\star}^2 ds \le K_0
+\frac{\Lam^2}{2}\norm{ v(t) }_\star^2 +  \frac{\Lam}{2}\norm{ v(t)}_{\star\star}^2 + \mathfrak{H}(t),
\end{equation}
where
\begin{equation}
\mathfrak{H}(t)=\int_0^t \left( \int_\Omega \mathfrak{G}^1\cdot\dt v+ \int_\Sigma \mathfrak{G}^3\cdot \dt v\right).
\end{equation}

Integrating in time and using Cauchy's inequality, we may bound
\begin{equation}\label{j222}
\begin{split}
\Lam\norm{v(t)}_{\star\star}^2
&=\Lam\norm{v(0)}_{\star\star}^2+ \Lam\int_0^t 2 \langle
v(s),\partial_t v(s) \rangle_{\star\star}\,ds
\\&\le
\Lam\norm{v(0)}_{\star\star}^2+\int_0^t\norm{\partial_tv(s)}_{\star\star}^2\,ds
+\Lam^2\int_0^t\norm{v(s)}_{\star\star}^2\,ds.
\end{split}
\end{equation}
On the other hand,
\begin{equation}\label{j333}
\Lam\partial_t\norm{v(t)}_{\star}^2=2 \Lam \langle
v(t),\partial_t v(t) \rangle_{\star}\le \norm{\partial_tv(t)}_{\star}^2
+\Lam^2\norm{v(t)}_{\star}^2.
\end{equation}
We may combine \eqref{j222} and \eqref{j333} with \eqref{j4} to derive the differential inequality
\begin{equation}\label{j5}
\partial_t\norm{v(t)}_\star^2+\norm{v(t)}_{\star\star}^2\le K_1+2\Lam \left( \norm{v(t)}_\star^2+\int_0^t\norm{v(s)}_{\star\star}^2 \,ds \right)+\frac{2}{\Lam}\mathfrak{H}(t),
\end{equation}
where
\begin{equation}
K_1=\frac{2 K_0}{\Lam} +2\norm{v(0)}_{\star\star}^2.
\end{equation}
An application of Gronwall's theorem then shows that
\begin{equation}\label{j6}
\norm{v(t)}_\star^2+\int_0^t\norm{v(s)}_{\star\star}^2
\le e^{2\Lam t}\norm{v(0)}_\star^2+\frac{K_1}{2\Lam}( e^{2\Lam
t}-1)+\frac{2}{\Lam}\int_0^t e^{2\Lam
(t-s)}\mathfrak{H}(s)ds.
\end{equation}
Now plugging \eqref{j6} and \eqref{j222} into \eqref{j4}, we find that
\begin{align}\label{j000}
\nonumber\frac{1}{\Lam}\norm{\partial_tv(t)}_\star^2+\norm{v(t)}_{\star\star}^2
&\le K_1+\Lam \norm{v(t)}_\star^2+2\Lam \int_0^t
\norm{v(s)}_{\star\star}^2\,ds+\frac{2}{\Lam}\mathfrak{H}(t)
\\ &\le  e^{2\Lam
t}(2\Lam\norm{v(0)}_\star^2+K_1)+\frac{2}{\Lam}\mathfrak{H}(t)+4\int_0^t e^{2\Lam
(t-s)}\mathfrak{H}(s)ds.
\end{align}

Notice that by integrating by parts in time,
\begin{align}
\nonumber 4\int_0^t e^{2\Lam
(t-s)}\mathfrak{H}(s)ds
&=-\frac{2}{\Lam}\int_0^t \p_t\left(e^{2\Lam
(t-s)}\right)\mathfrak{H}(s)ds\\&=-\frac{2}{\Lam}\mathfrak{H}(t)+e^{2\Lam
 t }\mathfrak{H}(0)+\frac{2}{\Lam}\int_0^t e^{2\Lam
(t-s)} \p_t \mathfrak{H}(s)ds
\end{align}
and
\begin{equation}
\mathfrak{H}(0)=0 \text{ and }\p_t \mathfrak{H} =\int_\Omega \mathfrak{G}^1\cdot\dt v+\int_\Sigma \mathfrak{G}^3\cdot \dt v.
\end{equation}
We then have
\begin{equation}\label{j0001}
\frac{2}{\Lam}\mathfrak{H}(t)+4\int_0^t e^{2\Lam
(t-s)}\mathfrak{H}(s)ds
=\frac{2}{\Lam}\int_0^t  e^{2\Lam
(t-s)} \left(\int_\Omega \mathfrak{G}^1(s)\cdot\dt v (s)+\int_\Sigma \mathfrak{G}^3(s)\cdot \dt v(s) \right) ds.
\end{equation}
Integrating by parts in time again, by the trace theory, we obtain
\begin{equation}\label{j0002}
\begin{split}
&\int_0^t  e^{2\Lam
(t-s)} \left(\int_\Omega \mathfrak{G}^1(s)\cdot\dt v (s)+\int_\Sigma \mathfrak{G}^3(s)\cdot \dt v(s) \right) ds
\\&\quad= \int_\Omega  \mathfrak{G}^1(t)\cdot v (t) - e^{2\Lam
t}  \int_\Omega\mathfrak{G}^1(0)\cdot v (0)+\int_\Sigma \mathfrak{G}^3(t)\cdot v (t)-e^{2\Lam
t}   \int_\Sigma \mathfrak{G}^3(0)\cdot v (0)
\\&\qquad -\int_0^t e^{2\Lam
(t-s)} \left( \int_\Omega\left(\dt \mathfrak{G}^1(s)-2\Lam \mathfrak{G}^1(s)\right)\cdot  v (s)+\int_\Sigma\left(\dt \mathfrak{G}^3(s)-2\Lam \mathfrak{G}^3(s)\right)\cdot  v (s)\right)  ds
\\ &\quad\ls \norm{ \mathfrak{G}^1(t)}_0\norm{v(t)}_0 +\abs{ \mathfrak{G}^3(t)}_{0}\abs{v(t)}_{0}
+e^{2\Lam
t}\left(\norm{ \mathfrak{G}^1(0)}_0\norm{v(0)}_0+\abs{ \mathfrak{G}^3(0)}_{0}\abs{v(0)}_{0}\right)
\\&\qquad + \int_0^t   e^{2\Lam
(t-s)}   \left(\left(\norm{\dt \mathfrak{G}^1(s)}_0+ \norm{\mathfrak{G}^1(s)}_0\right) \norm{v(s)}_0+  \left(\abs{\dt \mathfrak{G}^3(s)}_{0}+ \abs{\mathfrak{G}^3(s)}_{0}\right) \abs{v(s)}_{0}\right)
\\ &\quad\ls \left(\norm{ \mathfrak{G}^1(t)}_0  +\abs{ \mathfrak{G}^3(t)}_{0}\right)\norm{v(t)}_{1}
+e^{2\Lam
t} \left(\norm{ \mathfrak{G}^1(0)}_0 +\abs{ \mathfrak{G}^3(0)}_{0}\right)\norm{v(0)}_{1}
\\&\qquad +  \int_0^t   e^{2\Lam
(t-s)}    \left(\norm{\dt \mathfrak{G}^1(s)}_0+  \norm{\mathfrak{G}^1(s)}_0+\abs{\dt \mathfrak{G}^3(s)}_{0}+  \abs{\mathfrak{G}^3(s)}_{0}\right)\norm{v(s)}_{1} ds .
\end{split}
\end{equation}
Recalling the definitions of $K_0$ and $K_1$, by the trace theory, we have
\begin{equation}\label{j0003}
K_1\ls
\norm{ v(0)}_1^2+\norm{\partial_t v(0)}_0^2 .
\end{equation}
Thus, plugging \eqref{j0001}--\eqref{j0003} into \eqref{j000} and by Korn's and Cauchy's inequalities, we deduce
\begin{equation}
\begin{split}
\norm{v(t)}_{1}^2
  &\ls  e^{2\Lam
t}\left(\norm{ v(0)}_1^2+\norm{\partial_t v(0)}_0^2 +\norm{ \mathfrak{G}^1(0)}_0^2+ \abs{\mathfrak{G}^3(0)}_0^2\right)+ \sup_{0\le s\le t}  \left(\norm{ \mathfrak{G}^1(s)}_0^2+ \abs{\mathfrak{G}^3(s)}_0^2\right)
\\&\quad+ \int_0^t e^{2\Lam
(t-s)}   \left(\norm{\dt \mathfrak{G}^1(s)}_0+\norm{\mathfrak{G}^1(s)}_0+\abs{\dt \mathfrak{G}^3(s)}_{0}+ \abs{\mathfrak{G}^3(s)}_{0}\right)\norm{v(s)}_{1} .
\end{split}
\end{equation}
This yields \eqref{result10} by taking the square root.
\end{proof}

We can now show the estimates for the growth of solutions to  \eqref{linear ho}, which is the main result of this subsection.
\begin{thm} \label{lineargrownth}
Let  $u$ solve \eqref{linear ho}. Then
\begin{equation}\label{result1}
\begin{split}
\norm{u(t)}_{1}
  \ls&  e^{ \Lam
t}\left(\norm{ u(0)}_1 +\norm{\partial_t u(0)}_0 +\mathfrak{N}_G(0)\right)
\\&+ \sup_{0\le s\le t}\mathfrak{N}_G(s)+\sqrt{\int_0^t e^{2\Lam
(t-s)}\mathfrak{N}_G(s)(\norm{u(s)}_{1}+\mathfrak{N}_G(s))ds}
\end{split}
\end{equation}
and
\begin{equation}\label{result123}
\begin{split}
\norm{\eta(t)}_{1}
  \ls & \norm{ \eta(0)}_1 +e^{ \Lam
t}\left(\norm{ u(0)}_1 +\norm{\partial_t u(0)}_0 +\mathfrak{N}_G(0)\right)
\\&+\int_0^t  \sup_{0\le s\le \tau}\mathfrak{N}_G(s)d\tau+\int_0^t\sqrt{\int_0^\tau e^{2\Lam
(\tau-s)}\mathfrak{N}_G(s)(\norm{u(s)}_{1}+\mathfrak{N}_G(s)) ds}d\tau,
\end{split}
\end{equation}
where
\begin{equation}\label{result2}
\mathfrak{N}_G  =  \norm{\dt^2 G^1}_0+\norm{\dt G^1}_0+\norm{
 \dt^3 G^2 }_0+\norm{  \dt^2 G^2 }_1 +\norm{\dt G^2 }_1+\norm{  G^2 }_1+\abs{\dt^2 G^3}_0+\abs{  \dt  G^3}_0.
\end{equation}
\end{thm}
\begin{proof}
Since $u= \bar u+v$, by \eqref{result20}, we have
\begin{equation}\label{result1'}
\norm{u(t)}_{1}\le \norm{\bar u(t)}_{1}+\norm{v(t)}_{1}
  \ls  \norm{\bar u(t)}_{1}+e^{ \Lam
t}\left(\norm{ u(0)}_1 +\norm{\partial_t u(0)}_0+\norm{\bar u(0)}_1 +\norm{\partial_t \bar u(0)}_0      \right)
+ \mathfrak{N}_\mathfrak{G}(t),
\end{equation}
where $\mathfrak{N}_\mathfrak{G}(t)$ is defined by \eqref{result20}.
We then estimate $\mathfrak{N}_\mathfrak{G}(t)$. By \eqref{vG1_def}, \eqref{vG3_def} and \eqref{dives}, we have
\begin{equation}
\norm{\mathfrak{G}^1}_0\ls \norm{\dt G^1}_0+\norm{
 \dt^2 \bar u}+\norm{  \dt \bar u}_0 +\norm{ \bar u}_2
 \ls \norm{\dt G^1}_0+\norm{
 \dt^2 G^2 }_0+\norm{  \dt G^2 }_0 +\norm{G^2 }_1,
\end{equation}
and by using additionally the trace theory,
\begin{equation}
\begin{split}
\abs{\mathfrak{G}^3}_{0}&\ls\abs{\dt G^3}_0+\abs{\nabla \dt \bar u}_0+\abs{\nabla \bar u}_0+\abs{\bar u_3}_0.
\\&\ls \abs{\dt G^3}_0+\norm{  \dt \bar u}_2+\norm{  \bar u}_2
\ls \abs{\dt G^3}_0+\norm{  \dt G^2}_1+\norm{  G^2}_1.
\end{split}
\end{equation}
Similarly,
\begin{equation}
\norm{\dt \mathfrak{G}^1}_0\ls \norm{\dt^2 G^1}_0+\norm{
 \dt^3 G^2 }_0+\norm{  \dt^2 G^2 }_0 +\norm{\dt G^2 }_1
\end{equation}
and
\begin{equation}
\abs{\dt\mathfrak{G}^3}_{0}  \ls \abs{\dt^2 G^3}_0+\norm{  \dt^2 G^2}_1+\norm{ \dt G^2}_1.
\end{equation}
This implies
\begin{equation}
\norm{\mathfrak{G}^1 }_0+ \norm{\dt \mathfrak{G}^1 }_0+  \abs{\mathfrak{G}^3 }_{0}+\abs{\dt \mathfrak{G}^3 }_{0}
  \ls \mathfrak{N}_G,
\end{equation}
where $\mathfrak{N}_G$ is defined by \eqref{result2}. To save up the notations, we deduce  from \eqref{result1'} that
\begin{equation}\label{hhjj}
\norm{u(t)}_{1}
  \ls  e^{ \Lam
t}\left(\norm{ u(0)}_1 +\norm{\partial_t u(0)}_0 +\mathfrak{N}_G(0)\right)
+ \sup_{0\le s\le t}\mathfrak{N}_G(s)+\int_0^t e^{2\Lam
(t-s)}\mathfrak{N}_G(s)\norm{v(s)}_{1} ds
\end{equation}
Note that
\begin{equation}
\norm{v}_{1}\le \norm{u}_{1}+\norm{\bar u}_{1}\ls \norm{u}_{1}+\norm{G^2}_{0}\le \norm{u}_{1}+\mathfrak{N}_G,
\end{equation}
then \eqref{result1} follows from \eqref{hhjj}. Note that \eqref{result123} follows by using $\dt\eta=u$ and \eqref{result1}.
\end{proof}

\subsection{Nonlinear energy estimates}

This subsection, the most technical part in the proof of Theorem \ref{maintheorem}, is devoted to the nonlinear energy estimates for the system \eqref{reformulationic} when $\abs{\bar B_3}<\mc$. We recall again the structure of $\eta$ from \eqref{imp1}, which follows from the assumption \eqref{eta00} of $\eta_0$. The analysis here is similar to that of the stable regime when $\abs{\bar B_3}>\mc$ in Section \ref{nonlinear stability}.  The primary difference is that we will use  slightly modified versions of the energy and dissipation functionals in order to handle the fact that the internal interface makes a negative contribution to the original energy and dissipation. For the integer $N\ge 4$, we define the modified energy as
\begin{equation}\label{p_energy_defi}
 \fe{2N} := \sum_{j=0}^{2N}  \ns{\dt^j u}_{4N-2j} + \sum_{j=0}^{2N-1}  \ns{\nabla \dt^j p}_{4N-2j-2}+ \sum_{j=0}^{2N-1}  \as{  \jump{\dt^j p}}_{2n-2j-3/2}
+\ns{ \eta}_{1,4N}
\end{equation}
and the dissipation as
\begin{equation}\label{p_dissipation_defi}
\begin{split}
 \fd{2N} :=&  \sum_{j=0}^{2N}   \ns{\dt^j u}_{4N-2j+1}
+ \sum_{j=0}^{2N-1}  \ns{\nabla \dt^j p}_{4N-2j-1}
 +  \sum_{j=0}^{2N-1}  \as{  \jump{\dt^j p}}_{4N-2j-1/2}+\ns{(\bar B\cdot\nabla)\eta}_{0,4N}.
\end{split}
\end{equation}
Note that
\begin{equation}\label{equivalent}
 \fe{2N}+\f \simeq \se{2N}+\f\text{ and }\fd{2N}+\f \simeq \sd{2N}+\mathcal{J}_{2N}+\f
 \end{equation}

We will derive a priori estimates for solutions $(\eta,u,p)$ to \eqref{reformulationic} in our functional framework, i.e. for solutions satisfying $\fe{2N}$, $\fd{2N}$, $\f<\infty$.  Throughout this section we will assume that
\begin{equation}
\fe{2N}(t)+\f(t)\le \delta^2\le 1
\end{equation}
for some sufficiently small $\delta>0$ and for all $t \in [0,T]$ where $T>0$ is given. We will implicitly allow $\delta$ to be made smaller in each result, but we will reiterate the smallness of $\delta$ in our main result. The main result for the case $\bar B_3\neq 0$ is stated in Theorem \ref{engver}, and the main result for the case $\bar B_3= 0$ is stated in Theorem \ref{enghor}.

\subsubsection{Energy evolution}
In this subsection we derive energy evolution estimates for temporal and horizontal spatial derivatives by using the energy-dissipation structure of the system \eqref{reformulationic}. It follows from modified versions of those in Section \ref{stability evolution} by shifting the negative interface energy onto the right hand side of the estimates. Note that the estimates derived in this subsection holds for any $\bar B$.

 We first estimate the energy evolution of the pure temporal derivatives.
\begin{prop}\label{i_temporal_evolution  Ni}
It holds that
\begin{equation} \label{tem en 2Ni}
\begin{split}
&   \sum_{j=0}^{2N}\left(\norm{\dt^j u(t)}_0^2+\norm{(\bar B\cdot\nabla)\dt^j\eta(t)}_0^2\right)
+ \int_0^t\sum_{j=0}^{2N}\norm{    \dt^j u}_1^2
\\&\quad\ls \fe{2N}(0)+(\fe{2N}(t)+\f(t))^{3/2}+\int_0^t \sqrt{\fe{2N}+\f} (\fd{2N}+\f)+\int_0^t \abs{
 \eta_3}_{0}^2.
\end{split}
\end{equation}
\end{prop}
\begin{proof}
Following the proof of \eqref{i_te_0} in Proposition \ref{i_temporal_evolution N}, in light of \eqref{equivalent}, we can deduce that for $j=0,\dots,2N$,
\begin{equation}\label{iies1}
\begin{split}
&  \int_\Omega \left( \rho\abs{\dt^j u(t)}^2+\abs{(\bar B\cdot\nabla)\dt^j\eta(t)}^2\right)
+ \int_0^t\int_\Omega  \mu \abs{ \sg  \dt^j u}^2
\\&\quad\ls \fe{2N}(0)+(\fe{2N}(t)+\f(t))^{3/2}+\int_0^t \sqrt{\fe{2N}+\f} (\fd{2N}+\f)+\int_0^t\int_{\Sigma} \rj g\dt^j\eta_{3 }  \cdot \dt^j u_{3}.
\end{split}
\end{equation}
For the last term, by the trace theory and Cauchy's inequality, we have
 \begin{equation}\label{i_te_1m'}
 \int_{\Sigma}\rj g
  \partial_t^j \eta_3    \partial_t^j u_3
 \ls\abs{
\partial_t^j \eta_3}_{0} \abs{\partial_t^j u }_{0}
\ls C_\varepsilon  \abs{
\partial_t^j \eta_3}_{0}^2 +\varepsilon\norm{\partial_t^j u}_1^2
 \end{equation}
for any $\varepsilon>0$. Hence, employing Korn's inequality in  together with the estimates \eqref{i_te_1m'}, taking  $\varepsilon$  sufficiently small, we deduce from \eqref{iies1} that
\begin{equation}\label{tempor5}
\begin{split}
&  \norm{\dt^j u(t)}_0^2+\norm{(\bar B\cdot\nabla)\dt^j\eta(t)}_0^2
+ \int_0^t\norm{    \dt^j u}_1^2
\\&\quad\ls \fe{2N}(0)+(\fe{2N}(t)+\f(t))^{3/2}+\int_0^t \sqrt{\fe{2N}+\f} (\fd{2N}+\f)+\int_0^t \abs{
\partial_t^j \eta_3}_{0}^2.
\end{split}
\end{equation}

Now taking $j=0$ in \eqref{tempor5}, we have
\begin{equation}\label{tempor7}
\begin{split}
&  \norm{  u(t)}_0^2+\norm{(\bar B\cdot\nabla) \eta(t)}_0^2
+ \int_0^t\norm{    u}_1^2
\\&\quad\ls \fe{2N}(0)+(\fe{2N}(t)+\f(t))^{3/2}+\int_0^t \sqrt{\fe{2N}+\f} (\fd{2N}+\f)+\int_0^t \abs{
 \eta_3}_{0}^2.
\end{split}
\end{equation}
For $j=1,\dots,2N$, since $\dt\eta=u$, the trace theory shows that
\begin{equation}\label{tempor8}
\abs{
\partial_t^j \eta_3}_{0}^2=\norm{
\partial_t^{j-1} u_3}_{0}^2 \ls \norm{
\partial_t^{j-1} u_3}_{ 1}^2.
\end{equation}
Plugging \eqref{tempor8} into \eqref{tempor5}, we obtain
\begin{equation}\label{tempor9}
\begin{split}
&  \norm{\dt^j u(t)}_0^2+\norm{(\bar B\cdot\nabla)\dt^j\eta(t)}_0^2
+ \int_0^t\norm{    \dt^j u}_1^2
\\&\quad\ls \fe{2N}(0)+(\fe{2N}(t)+\f(t))^{3/2}+\int_0^t \sqrt{\fe{2N}+\f} (\fd{2N}+\f)+\int_0^t \norm{
\partial_t^{j-1} u}_{ 1}^2.
\end{split}
\end{equation}
Hence, by chaining together \eqref{tempor7} and \eqref{tempor9}, we get \eqref{tem en 2Ni}.
\end{proof}

Next, we estimate the energy evolution of the horizontal spatial derivatives.

\begin{prop}\label{p_upper_evolution  N12i}
It holds that
\begin{equation}\label{energy 2Ni}
\begin{split}
&
 \norm{\nabla_\ast u(t)}_{0,4N-1 }^2+\norm{(\bar B\cdot\nabla) \nabla_\ast \eta(t)}_{0,4N-1 }^2
  +\int_0^t \norm{\nabla_\ast u }_{1,4N-1 }^2
   \\
&\quad\lesssim  \fe{2N}(0) +\int_0^t \sqrt{\fe{2N}+\f} (\fd{2N}+\f)+
\int_0^t \abs{ \eta_3}_{4N- 1/2}^2 .
  \end{split}
  \end{equation}
 \end{prop}
\begin{proof}
Following the proof of \eqref{p_u_e_00} in Proposition \ref{p_upper_evolution  N12}, we can deduce that for $\al\in \mathbb{N}^{2}$ so that $\al_1+\al_2\ge 1$ and $|\al|\le 4N$,
\begin{equation}\label{es_00i}
\begin{split}
 &  \int_\Omega \rho \abs{\p^\al  u(t) }^2  +\abs{(\bar B\cdot \nabla)\partial^\alpha  \eta(t) }^2 + \int_0^t\int_\Omega \mu  \abs{\sg \p^\al  u }^2
 \\&\quad\ls \fe{2N}(0) +\int_0^t\sqrt{\fe{2N}+\f} (\fd{2N}+\f)+\int_0^t\int_{\Sigma} \rj g  \partial^\alpha  \eta_3 \partial^\alpha  u_3.
\end{split}
\end{equation}
For the last term in \eqref{es_00i}, by the trace theory and Cauchy's inequality, we have
\begin{equation}\label{m6}
\int_{\Sigma} \rj g\pa^\al \eta_3 \pa^\al u_3\ls\abs{\partial^\alpha\eta_3}_{-1/2}\abs{\partial^\alpha u_{3}}_{1/2}
\ls
C_\varepsilon \abs{ \eta_3}_{4N- 1/2}^2+\varepsilon\norm{
\partial^\alpha u }_{1}^2
\end{equation}
for any $\varepsilon>0$. Hence, employing Korn's inequality   in  together with the estimates \eqref{m6}, taking  $\varepsilon$  sufficiently small, we deduce from \eqref{es_00i} that
\begin{align}\label{ls0}
 \nonumber& \norm{\partial^\alpha   u (t)}_0^2+\norm{(\bar B\cdot \nabla)\partial^\alpha  \eta(t)}_0^2
  +\int_0^t  \norm{\partial^\alpha   u}_1^2\,ds
 \\
&\quad\lesssim  \fe{2N}(0) +\int_0^t\sqrt{\fe{2N}+\f} (\fd{2N}+\f)+
  \int_0^t \abs{ \eta_3}_{4N -1/2}^2\,ds.
 \end{align}
Hence, summing over such $\al$, we get \eqref{energy 2Ni}.
\end{proof}

We now estimate the energy evolution that recovers the estimates of $\eta$.

\begin{prop}\label{p_upper_evolution  N'132i}
It holds that
\begin{equation}\label{p_u_e_00'132i}
\begin{split}
& \norm{\eta(t)}_{1,4N}^2+\int_0^t\norm{(\bar B\cdot \nabla)  \eta }_{0,4N}^2
\\&\quad\ls \fe{2N}(0) +\norm{u(t)}_{0,4N}^2+\int_0^t\sqrt{\fe{2N}+\f} (\fd{2N}+\f)+
  \int_0^t\left(\norm{u}_{0,4N}^2+\abs{\eta_3 }_{ 4N }^2\right).
\end{split}
\end{equation}
\end{prop}
\begin{proof}
Following the proof of \eqref{p_u_e_00'132} in Proposition \ref{p_upper_evolution  N'132}, we can deduce that for $\al\in \mathbb{N}^{2}$ so that $ |\al|\le 4N$,
\begin{equation} \label{p1p111123}
\begin{split}
  & \int_\Omega \mu   \abs{\sg  \p^\al  \eta(t)}^2 +\int_0^t\int_{\Omega}  \abs{(\bar B\cdot \nabla)\partial^\alpha  \eta }^2
\\&\quad  \lesssim  \fe{2N}(0)
   +\int_0^t \sqrt{\fe{2N}+\f} (\fd{2N}+\f)
   \\&\qquad-\int_\Omega  \rho  \partial^\alpha  u (t)\cdot \partial^\alpha  \eta(t)+
  \int_0^t\int_\Omega  \rho  \abs{\partial^\alpha  u}^2+ \int_0^t\int_{\Sigma} \rj g \abs{\partial^\alpha  \eta_3}^2   .
\end{split}
\end{equation}
Employing Korn's and Cauchy's inequalities, and then summing over such $\al$, we obtain \eqref{p_u_e_00'132i}.
\end{proof}

To conclude the energy evolution estimates, we define the horizontal energy
\begin{equation}
 \feb{2N}:= \sum_{j=0}^{2N}\left(\norm{\dt^j u}_{0}^2+\norm{(\bar B\cdot\nabla)\dt^j\eta}_{0}^2\right)
+\norm{\nabla_\ast u}_{0,4N-1 }^2+\norm{(\bar B\cdot\nabla) \nabla_\ast \eta}_{0,4N-1 }^2+ \norm{\eta}_{1,4N}^2
\end{equation}
and the corresponding dissipation
\begin{equation}
 \fdb{2N}:=\sum_{j=0}^{2N}\norm{    \dt^j u}_{1}^2+\norm{\nabla_\ast u}_{1,4N-1 }^2+\norm{(\bar B\cdot \nabla)  \eta }_{0,4N}^2.
\end{equation}
Note that
\begin{equation}\label{equivalent12}
 \feb{2N} \simeq \seb{2N} \text{ and }\fdb{2N}=\sdb{2N}.
 \end{equation}
We have the following.
\begin{prop}\label{conclusioni}
It holds that
\begin{equation}\label{ededed}
\begin{split}
  \feb{2N}(t)+\int_0^t\fdb{2N}
 \ls &\fe{2N}(0) +(\fe{2N}(t)+\f(t))^{3/2}
 \\&+\int_0^t \sqrt{\fe{2N}+\f} (\fd{2N}+\f) +
  \int_0^t \abs{\eta_3 }_{4N}^2 .
  \end{split}
\end{equation}
\end{prop}
\begin{proof}
A suitable linear combination of the results of Propositions \ref{i_temporal_evolution Ni}--\ref{p_upper_evolution  N'132i} gives the desired estimates.
\end{proof}

\subsubsection{Full energy estimates when $\bar B_3\neq0$}

We now improve the energy evolution estimates to derive the energy-dissipation estimates for the case $\bar B_3\neq0$. The conclusion is the following.
\begin{thm}\label{engver}
Assume $\bar B_3\neq 0$. For any $\varepsilon>0$, there exists $C_\varepsilon>0$ so that
\begin{equation}\label{non-horizontal es}
\begin{split}
&\feb{2N}(t)
+ \varepsilon( \fe{2N}(t)+ \f(t))+  \int_0^t\left( \fdb{2N}
+ \varepsilon( \fd{2N}+ \f)\right)
 \\&\quad\le C_\varepsilon \left( \fe{2N}(0) +  \f(0)\right) + \varepsilon \int_0^t \feb{2N} +
 C_\varepsilon \int_0^t\as{\eta_3 }_{0} .
  \end{split}
\end{equation}
\end{thm}
\begin{proof}
Following the proof of the estimates \eqref{109} in Proposition \ref{p_f_bound}, we can deduce that
\begin{equation}\label{109i}
\begin{split}
  \f(t)
+  \int_0^t\left(\f +\mathcal{J}_{2N}\right)  \ls  \f(0)+\int_0^t& \left(\feb{2N}+ \fdb{2N}+\norm{\p_t u}_{4N-1}^2\right.
 \\&\ \ + {( \fe{2N}+\f  )}(\fd{2N} +\f)\Big).
 \end{split}
\end{equation}
Combining \eqref{109i} and the estimates \eqref{eses11} with $n=2N$ in Proposition \ref{p_upper_evolution  N'}, yields that
\begin{equation}\label{109i2}
\begin{split}
  \f(t)
+  \int_0^t\left(\fd{2N}+\f\right) \ls  \f(0)+\int_0^t &\left(\feb{2N}   +\bar{\mathcal{D}}_{2N}+  \norm{  u   }_{4N-1 }^2\right.
 \\&\ \ +  {( \fe{2N}+\f  )}(\fd{2N} +\f)\Big).
 \end{split}
\end{equation}
Using the Sobolev interpolation and Young's inequality, and since $\fe{2N}+\f \le \delta^2$ is small, we may improve \eqref{109i2} to be
\begin{equation}\label{109i222}
\begin{split}
  \f(t)
+  \int_0^t\left(\fd{2N}+\f\right)  &\ls  \f(0)+\int_0^t \left(\feb{2N}   +\bar{\mathcal{D}}_{2N}+  \norm{  u   }_{0}^2 \right)
\\&\ls  \f(0)+\int_0^t \left(\feb{2N}+ \bar{\mathcal{D}}_{2N}\right) .
\end{split}
\end{equation}

By \eqref{109i222}, we may deduce from the estimates \eqref{ededed} in Proposition \ref{conclusioni} that
\begin{equation}
\begin{split}
  \feb{2N}(t)+\int_0^t\fdb{2N}
 \ls &\fe{2N}(0) +\delta(\fe{2N}(t)+\f(t))
 \\&+\delta \left(\f(0)+\int_0^t \left(\feb{2N}+ \bar{\mathcal{D}}_{2N}\right) \right) +
  \int_0^t \as{\eta_3 }_{ 4N } ,
\end{split}
\end{equation}
which implies
\begin{equation}\label{ver3}
  \feb{2N}(t)+\int_0^t\fdb{2N}
 \ls \fe{2N}(0) +  \f(0)+\delta(\fe{2N}(t)+\f(t)) +\delta\int_0^t \feb{2N}  +
  \int_0^t \as{\eta_3 }_{ 4N} .
\end{equation}

A suitable linear combination of \eqref{109i222} and \eqref{ver3} yields that, since $\as{\eta_3 }_{ 4N}\ls \ns{\eta_3 }_{1, 4N} \ls \feb{2N}$,
\begin{equation}\label{109i23}
 \feb{2N}(t)+ \f(t)
+  \int_0^t\left(\fd{2N} +\f\right)   \ls  \fe{2N}(0) +\f(0)+\delta(\fe{2N}(t)+\f(t)) +\int_0^t \feb{2N}   .
\end{equation}
On the other hand, by Proposition \ref{e2nic}, we have
\begin{equation}\label{e2ni}
\fe{2N}  \lesssim
\feb{2N}  +\f+ \delta(\fe{2N}+\f),
\end{equation}
which implies
\begin{equation}\label{e2ni2}
\fe{2N}  \lesssim
\feb{2N}  +\f .
\end{equation}
We then deduce from \eqref{109i23} that
\begin{equation}\label{109i2345}
 \fe{2N}(t)+ \f(t)
+  \int_0^t\left(\fd{2N}+\f \right)  \ls  \fe{2N}(0) +\f(0)+\int_0^t \feb{2N} .
\end{equation}

Consequently, \eqref{ver3}+\eqref{109i2345}$\times\varepsilon$ implies that
\begin{equation}\label{hhqq}
\begin{split}
&\seb{2N}(t)
+ \varepsilon( \fe{2N}(t)+ \f(t))+  \int_0^t\left( \fdb{2N}
+ \varepsilon( \fd{2N}+ \f)\right)
 \\&\quad\ls \fe{2N}(0) +  \f(0)+\delta(\fe{2N}(t)+\f(t))+(\delta+\varepsilon)\int_0^t \feb{2N} +
  \int_0^t \as{\eta_3 }_{ 4N} .
  \end{split}
\end{equation}
By the Sobolev interpolation, the trace theory and Young's inequality, we have
\begin{equation}\label{e3bound}
\as{\eta_3 }_{ 4N}  \ls \varepsilon \as{\eta_3 }_{4N+1/2}   +
  {\varepsilon^{-1}} \as{\eta_3 }_{0}  \ls \varepsilon \norm{\eta_3 }_{1,4N}^2  +
  {\varepsilon^{-1}} \as{\eta_3 }_{0}\ls \varepsilon \feb{2N}  +
  {\varepsilon^{-1}} \as{\eta_3 }_{0}.
\end{equation}
We can then think $\delta\ll \varepsilon$, by \eqref{e3bound}, to deduce from \eqref{hhqq} that
\begin{equation}\label{non-horizontal es22}
\begin{split}
&\seb{2N}(t)
+ \varepsilon( \fe{2N}(t)+ \f(t))+  \int_0^t\left( \fdb{2N}
+ \varepsilon( \fd{2N}+ \f)\right)
 \\&\quad\ls \fe{2N}(0) +  \f(0) + \varepsilon \int_0^t \feb{2N}  +
 {\varepsilon^{-1}} \int_0^t\as{\eta_3 }_{0} .
  \end{split}
\end{equation}
Thus,  \eqref{non-horizontal es} follows from \eqref{non-horizontal es22}.
\end{proof}

\subsubsection{Full energy estimates when $\bar B_3=0$}
We now improve the energy evolution estimates to derive the energy-dissipation estimates for the case $\bar B_3=0$. The conclusion is the following.
\begin{thm}\label{enghor}
Assume $\bar B_3=0$. For any $\varepsilon>0$, there exist $C_\varepsilon>0$ and an energy $\mathcal{F}_{2N}^\varepsilon$ with $\f\le \mathcal{F}_{2N}^\varepsilon\le  C_\varepsilon \f$ such that
\begin{equation}\label{horizontal es}
\begin{split}
 &{\mathfrak{E}}_{2N}(t)  +\mathcal{F}_{2N}^\varepsilon(t)+ \varepsilon\int_0^t\fd{2N}
       \\&\quad\le C_\varepsilon \left(\fe{2N}(0)+\mathcal{F}_{2N}^\varepsilon(0)\right)
       + \varepsilon \int_0^t \left(\fe{2N}  +\mathcal{F}_{2N}^\varepsilon\right)
  + C_\varepsilon
  \int_0^t \as{\eta_3 }_{0}  .
   \end{split}
\end{equation}
\end{thm}
\begin{proof}
We divide the proof into several steps.

{\bf Step 1: control $\fd{2N}$.}

Since $\bar B_3=0$, the equations \eqref{perturb} yields
\begin{equation}\label{stokesp12}
\begin{cases}
-\mu\Delta  u+\nabla   p =-\rho\dt  u+(\bar B_\ast\cdot\nabla_\ast)^2  \eta+ G^1& \text{in }
\Omega
\\ \diverge  u=  G^2 & \text{in }
\Omega
\\ \Lbrack \dt^j u\Rbrack =0,\ \Lbrack pI-\mu\mathbb{D}  u \Rbrack e_3= \rj g  \eta_3 e_3+  G^3 &\hbox{on }\Sigma
\\  u=0 &\text{on }\Sigma_{m,\ell}.
\end{cases}
\end{equation}
Applying the time derivatives $\dt^{j},\ j=1,\dots,2N-1$ to \eqref{stokesp12} and then applying the elliptic estimates \eqref{cSresult} of Lemma \ref{cStheorem} with $r=4N-2j+1 \ge 3$, following the proof of the estimates \eqref{eses11} in Proposition \ref{p_upper_evolution  N'}, we may deduce
\begin{equation}\label{eses00i}
\begin{split}
 & \sum_{j=1}^{2N}\norm{ \dt^j u  }_{4N-2j+1}^2+  \sum_{j=1}^{2N-1}\norm{\nab \dt^j  p  }_{4N-2j-1}^2+\sum_{j=1}^{2N-1}\abs{  \jump{\dt^j p}  }_{4N-2j-1/2}^2
 \\&\quad \ls \norm{  u   }_{4N- 1 }^2    +\fdb{2N}+ {( \fe{2N}+\f  )}(\fd{2N} +\f).
\end{split}
\end{equation}
Also, applying the elliptic estimates \eqref{cSresult} of Lemma \ref{cStheorem} with $r=4N- 1 $ to \eqref{stokesp12}, we obtain
\begin{equation}\label{eses00i22}
\begin{split}
&\norm{ u  }_{ 4N+1}^2+   \norm{\nab   p  }_{4N-1}^2+\abs{  \jump{ p}  }_{4N -1/2}^2
 \\&\quad \ls \norm{   \eta }_{4N-1,2}^2+ \norm{\p_t u}_{4N-1}^2
 +  {( \fe{2N}+\f  )}(\fd{2N} +\f) ,
   \end{split}
\end{equation}
Combining \eqref{eses00i} and \eqref{eses00i22} yields
\begin{equation}
\fd{2N} \ls \norm{  u   }_{4N- 1 }^2    +\norm{   \eta }_{4N-1,2}^2+ \fdb{2N}
+ {( \fe{2N}+\f  )}(\fd{2N} +\f),
\end{equation}
which implies, since $ \fe{2N}+\f \le \delta^2$ is small,
\begin{equation}\label{109i223}
\fd{2N} \ls \norm{  u   }_{4N- 1 }^2    +\norm{   \eta }_{4N-1,2}^2+ \fdb{2N}
+ \delta^2  \f
 \ls \norm{  u   }_{4N- 1 }^2    + \fdb{2N}
+   \f.
\end{equation}
Using the Sobolev interpolation and Young's inequality, we may improve \eqref{109i223} to be
\begin{equation}\label{deses}
\fd{2N} \ls \norm{  u   }_{0 }^2       + \fdb{2N}
+   \f\ls   \fdb{2N}
+   \f.
\end{equation}

{\bf Step 2: improve the energy evolution estimates.}

It follows from the estimates \eqref{ededed} in Proposition \ref{conclusioni} and \eqref{deses} that
\begin{equation}
  \feb{2N}(t)+\int_0^t\fdb{2N}
 \ls \fe{2N}(0) +\delta(\fe{2N}(t)+\f(t)) +\delta\int_0^t   (\fdb{2N}+\f) +
  \int_0^t \abs{\eta_3 }_{4N}^2,
\end{equation}
which implies
\begin{equation} \label{bareseses}
  \feb{2N}(t)+\int_0^t\fdb{2N} \ls \fe{2N}(0) +\delta(\fe{2N}(t)+\f(t)) +\delta\int_0^t    \f  +
  \int_0^t \abs{\eta_3 }_{4N}^2.
\end{equation}

{\bf Step 3: control $\eta$.}

Note that we can not hope to improve estimates of vertical derivatives of $\eta$ in either dissipation or energy as the non-horizontal case by employing the Stokes regularity for the quantity $(w,p)$. We will need to use mutually the Stokes system for $(u,p)$ and the equation $\dt\eta=u$. First, we directly obtain, by Cauchy's inequality,
\begin{equation}\label{eses11i20-}
\dt \norm{ \eta  }_{1,4N}^2\ls \norm{\eta}_{1 ,4N}\norm{ u  }_{1,4N}\ls \varepsilon\norm{\eta}_{1,4N}^2+ {\varepsilon^{-1}}
 \fdb{2N},
\end{equation}
and for $j=1,\dots,2N$,
\begin{equation}\label{eses11i20}
\dt \norm{ \eta  }_{2j+1,4N-2j}^2\ls \norm{ u  }_{2j+1,4N-2j}\norm{\eta}_{2j+1,4N-2j}\ls \varepsilon\norm{\eta}_{2j+1,4N-2j}^2+ {\varepsilon^{-1}}\norm{ u  }_{2j+1,4N-2j}^2.
\end{equation}
On the other hand, for $j=1,\dots,2N$, we apply the elliptic estimates \eqref{cSresult} of Lemma \ref{cStheorem} to \eqref{stokesp12} to find that
\begin{equation}\label{eses11i2}
\begin{split}
&\norm{ u  }_{2j+1,4N-2j}^2+   \norm{\nab   p  }_{2j-1,4N-2j}^2+\abs{  \jump{ p}  }_{4N -1/2}^2
 \\&\quad\ls \norm{(\bar B_\ast\cdot\nabla_\ast)^2 \eta }_{2j-1,4N-2j}^2+ \norm{\p_t u}_{4N-1}^2
 + \abs{  \eta_3 }_{2j-1/2+4N-2j}^2+  {( \fe{2N}+\f  )}(\fd{2N} +\f)
      \\&\quad\ls \norm{   \eta }_{2j-1,4N-2(j-1)}^2+ \norm{  u   }_{4N- 1 }^2   +\fdb{2N}+  \delta^2 \f.
   \end{split}
\end{equation}
Here in the last inequality we have used \eqref{eses00i} and \eqref{deses}. In light of \eqref{eses11i2},  \eqref{eses11i20} implies that
\begin{equation}\label{eio}
\begin{split}
\dt \norm{ \eta  }_{2j+1,4N-2j}^2& \ls\varepsilon\norm{ \eta  }_{2j+1,4N-2j}^2+ {\varepsilon^{-1}}\norm{   \eta }_{2j-1,4N-2(j-1)}^2 \\&\quad+{\varepsilon^{-1}} \left(\norm{  u}_{4N-1}^2
   +\fdb{2N}+  \delta^2 \f\right) .
   \end{split}
\end{equation}
Multiplying \eqref{eio} by $\varepsilon^{2j}$ and summing over $j=1,\dots,2N$, we get
\begin{equation}\label{eses11i4}
\begin{split}
\dt \sum_{j=1}^{2N}\varepsilon^{2j}  \norm{ \eta  }_{2j+1,4N-2j}^2 &\ls \varepsilon \sum_{j=1}^{2N}\varepsilon^{2j} \norm{ \eta  }_{2j+1,4N-2j}^2+  \sum_{j=1}^{2N} \varepsilon^{2j-1} \norm{   \eta }_{2j-1,4N-2(j-1)}^2
\\&\quad+ \sum_{j=1}^{2N} \varepsilon^{2j-1} \left(\norm{  u}_{4N-1}^2
   +\fdb{2N}+  \delta^2 \f\right)
   \\&\ls \varepsilon \sum_{j=1}^{2N}\varepsilon^{2j} \norm{ \eta  }_{2j+1,4N-2j}^2+  \varepsilon\sum_{j=1}^{2N} \varepsilon^{2(j-1)} \norm{   \eta }_{2j-1,4N-2(j-1)}^2
\\&\quad+  \varepsilon  \left(\norm{  u}_{4N-1}^2
   +\fdb{2N}+  \delta^2 \f\right)
   \\&\ls \varepsilon \sum_{j=0}^{2N}\varepsilon^{2j} \norm{ \eta  }_{2j+1,4N-2j}^2 +   \varepsilon  \left(\norm{  u}_{4N-1}^2
   +\fdb{2N}+  \delta^2 \f\right).
   \end{split}
\end{equation}
We now define
\begin{equation}
 \mathfrak{F}_{2N}^\varepsilon:= \sum_{j=0}^{2N}\varepsilon^{2 j  }  \norm{ \eta  }_{2j+1,4N-2j}^2.
\end{equation}
Then \eqref{eses11i20-} and \eqref{eses11i4} yields
\begin{equation}\label{jmj}
\dt \mathfrak{F}_{2N}^\varepsilon\ls \varepsilon \mathfrak{F}_{2N}^\varepsilon+ {\varepsilon^{-1}}
 \fdb{2N} +  \varepsilon   \left(\norm{  u}_{4N-1}^2
 +  \delta^2 \f\right).
\end{equation}
Integrating \eqref{jmj} in time, by \eqref{bareseses}, we obtain
\begin{equation}\label{bareseses2}
\begin{split}
 \mathfrak{F}_{2N}^\varepsilon(t) &\ls\mathfrak{F}_{2N}^\varepsilon(0)+ \varepsilon \int_0^t \mathfrak{F}_{2N}^\varepsilon+ {\varepsilon^{-1}}\fe{2N}(0)+ \frac{\delta}{\varepsilon}(\fe{2N}(t)+\f(t))
  \\&\quad+  \frac{\delta}{\varepsilon}
 \int_0^t  \f +  {\varepsilon^{-1}}
  \int_0^t \abs{\eta_3 }_{4N}^2  +  \varepsilon \int_0^t   \norm{ u}_{4N-1}^2 .
   \end{split}
\end{equation}

{\bf Step 4: control $\fe{2N}$.}

By Proposition \ref{e2nic}, we have
\begin{equation}
\fe{2N}  \lesssim
\feb{2N}  +\f+ \delta(\fe{2N}+\f) ,
\end{equation}
which implies
\begin{equation}\label{bareseses3}
\fe{2N}  \lesssim
\feb{2N}  +\f .
\end{equation}

{\bf Step 5: chaining estimates.}

Summing up \eqref{bareseses} and \eqref{bareseses2}, by \eqref{bareseses3} and noting that $\f\le \varepsilon^{-4N} \mathfrak{F}_{2N}$, we deduce that
\begin{equation}\label{eweq}
\begin{split}
 &\mathfrak{F}_{2N}^\varepsilon(t)+\varepsilon^{4N}\fe{2N}(t) + \feb{2N}(t)+ \int_0^t\fdb{2N}
\\&\quad\ls\mathfrak{F}_{2N}^\varepsilon(0)+ \varepsilon \int_0^t \mathfrak{F}_{2N}^\varepsilon+ {\varepsilon^{-1}}\fe{2N}(0)+ \frac{\delta}{\varepsilon^{4N+1}}\left(\mathfrak{F}_{2N}^\varepsilon(t)+\varepsilon^{4N}\fe{2N}(t)\right)
  \\&\qquad+  \frac{\delta}{\varepsilon^{4N+1}}
 \int_0^t  \mathfrak{F}_{2N}^\varepsilon +  {\varepsilon^{-1}}
  \int_0^t \abs{\eta_3 }_{4N}^2  +  \varepsilon \int_0^t   \norm{ u}_{4N-1}^2 .
   \end{split}
\end{equation}
We can think $ \delta\le \varepsilon^{4N+2}$ and hence \eqref{eweq} reduces to
\begin{equation}\label{pppewe}
\begin{split}
 &\mathfrak{F}_{2N}^\varepsilon(t)+\varepsilon^{4N}\fe{2N}(t) + \feb{2N}(t)+ \int_0^t\fdb{2N}
     \\&\quad\ls\mathfrak{F}_{2N}^\varepsilon(0)+ {\varepsilon^{-1}}\fe{2N}(0)+ \varepsilon \int_0^t \mathfrak{F}_{2N}^\varepsilon
   +  {\varepsilon^{-1}}
  \int_0^t \abs{\eta_3 }_{4N}^2  +  \varepsilon \int_0^t   \norm{ u}_{4N-1}^2 .
   \end{split}
\end{equation}
Using the Sobolev interpolation and Young's inequality, we have
\begin{equation}
\norm{ u}_{4N-1}^2\ls \varepsilon^{4N}\norm{u}_{4N}^2+\varepsilon^{-4N}\norm{u}_{0}^2\ls \varepsilon^{4N}\fe{2N}+\varepsilon^{-4N}\fdb{2N}.
\end{equation}
We may then improve \eqref{pppewe} to be, by using \eqref{bareseses} again,
\begin{equation}\label{eiieo}
\begin{split}
 &\mathfrak{F}_{2N}^\varepsilon(t)+\varepsilon^{4N}\fe{2N}(t) + \feb{2N}(t)+ \int_0^t\fdb{2N}
     \\&\quad\ls\mathfrak{F}_{2N}^\varepsilon(0)+ {\varepsilon^{-1}}\fe{2N}(0)+ \varepsilon \int_0^t \left(\mathfrak{F}_{2N}^\varepsilon
     +\varepsilon^{4N}\fe{2N}\right)
  +  {\varepsilon^{-1}}
  \int_0^t \as{\eta_3 }_{ 4N}  +  \varepsilon^{1-4N}\int_0^t \fdb{2N}
  \\&\quad\ls\mathfrak{F}_{2N}^\varepsilon(0)+\varepsilon^{1-4N}\fe{2N}(0)+ \varepsilon \int_0^t \left(\mathfrak{F}_{2N}^\varepsilon
     +\varepsilon^{4N}\fe{2N}\right)
  \\&\qquad+\frac{\delta}{\varepsilon^{8N-1}}\left(\mathfrak{F}_{2N}^\varepsilon(t)+\varepsilon^{4N}\fe{2N}(t)\right) +  \frac{\delta}{\varepsilon^{8N-1}}\int_0^t  \mathfrak{F}_{2N}^\varepsilon
  + \varepsilon^{1-4N}
  \int_0^t \as{\eta_3 }_{ 4N}  .
   \end{split}
\end{equation}
We can think further $\delta\le \varepsilon^{8N}$ and hence \eqref{eiieo} reduces to
\begin{equation}\label{lfaf}
\begin{split}
 &\mathfrak{F}_{2N}^\varepsilon(t)+\varepsilon^{4N}\fe{2N}(t) + \feb{2N}(t)+ \int_0^t\fdb{2N}
       \\&\quad\ls\mathfrak{F}_{2N}^\varepsilon(0)+\varepsilon^{1-4N}\fe{2N}(0)+ \varepsilon \int_0^t \left(\mathfrak{F}_{2N}^\varepsilon
     +\varepsilon^{4N}\fe{2N}\right)
  + \varepsilon^{1-4N}
  \int_0^t \as{\eta_3 }_{ 4N}   .
   \end{split}
\end{equation}
By the Sobolev interpolation, the trace theory and Young's inequality, we have
\begin{equation}
\begin{split}
\as{\eta_3 }_{ 4N} &\le \varepsilon^{-8N}\as{\eta_3 }_{0} +\varepsilon^{8N}\as{\eta_3 }_{ 4N+1/2}
\\&\le \varepsilon^{-8N}\as{\eta_3 }_{0} +\varepsilon^{8N}\norm{\eta_3 }_{4N+1}^2
\le \varepsilon^{-8N}\as{\eta_3 }_{0} +\varepsilon^{4N}\mathfrak{F}_{2N}^\varepsilon.
\end{split}
\end{equation}
Hence, we deduce from \eqref{lfaf} that
\begin{equation}\label{gesr}
\begin{split}
 &\mathfrak{F}_{2N}^\varepsilon(t)+\varepsilon^{4N}\fe{2N}(t) + \feb{2N}(t)+ \int_0^t\fdb{2N}
       \\&\quad\ls\mathfrak{F}_{2N}^\varepsilon(0)+\varepsilon^{1-4N}\fe{2N}(0)+ \varepsilon \int_0^t \left(\mathfrak{F}_{2N}^\varepsilon
     +\varepsilon^{4N}\fe{2N}\right)
  + \varepsilon^{1-12N}
  \int_0^t \as{\eta_3 }_{0}  .
   \end{split}
\end{equation}
Defining $\mathcal{F}^\varepsilon_{2N}:=\varepsilon^{-4N}\mathfrak{F}_{2N}^\varepsilon(t)$, we know that $\f\le \mathcal{F}^\varepsilon_{2N}\le C_\varepsilon \f$ and we deduce \eqref{horizontal es} from \eqref{gesr} by using \eqref{deses} again.
\end{proof}

\subsection{Bootstrap argument}\label{proof}
In this subsection, we will combine the estimates derived in the previous sections to prove Theorem \ref{maintheorem}, by employing the Guo-Strauss bootstrap argument \cite{GS}.

For notational convenience, we denote
\begin{equation}\label{Uvector}
U:= ( \eta,  u,  p),
\end{equation}
and we define the triple norm $\Lvert3\cdot\Rvert3_{00}$  by
\begin{equation}\label{norm3}
 \Lvert3  U \Rvert3_{00} := \sqrt{\mathcal{E}_{2N} +\mathcal{F}_{2N}}\simeq \sqrt{\fe{2N} +\mathcal{F}_{2N}}
\end{equation}
for an integer $N\ge 4$.

We now restate the main results of the previous sections in our new notation.
\begin{prop}\label{prop1}
Assume $\abs{\bar B_3}<\mc$.
Let the norm $\Lvert3\cdot\Rvert3_{00}$ be given by \eqref{norm3}. Then we have the following.
\begin{enumerate}
\item There is a growing mode $U^\star:=(\eta^\star, u^\star, p^\star)$
 satisfying $\abs{ \eta^\star_3}_0= 1$,
$\Lvert3 U^\star \Rvert3_{00}=C_1<\infty$, and
$e^{\lam t}U^\star$ is the solution to \eqref{perturb_linear}.

\item Suppose that $U(t)$ is the solution to \eqref{reformulationic}. There exists $C_2> 0$ so that for any $\iota>0$,
\begin{align}\label{bbbf}
\nonumber \abs{ \eta_3(t)-\iota e^{\lam t}\eta^\star_3}_{0} & \le
C_2 e^{ \Lam t}\left(\Lvert3U(0)-\iota U^\star  \Rvert3_{00}+\Lvert3U(0)  \Rvert3_{00}^2\right) +C_2\int_0^t  \sup_{0\le s\le \tau}\Lvert3U(s)  \Rvert3_{00}^2 d\tau
 \\ &\quad+C_2\int_0^t\sqrt{\int_0^\tau  e^{2\Lam
(\tau-s)} \Lvert3 U (s)  \Rvert3_{00}^2 (\Lvert3U (s)-\iota e^{\lam s} U^\star   \Rvert3_{00}+\Lvert3 U (s)  \Rvert3_{00}^2)ds}d\tau.
\end{align}

\item  Suppose that $U(t)$ is the solution to \eqref{reformulationic}. There exists a small constant $\delta$ such that  if $\Lvert3U(t) \Rvert3_{00}\le\delta$ for all $t \in [0,T]$, then there exists $C_\delta> 0$ so
that the following inequality holds for $t \in [0,T]$:
\begin{equation}\label{energyes}
  \Lvert3 U(t) \Rvert3_{00}^2 \le  C_\delta\Lvert3 U(0)  \Rvert3_{00}^2
   +  \lambda\int_0^t\Lvert3U(s)  \Rvert3_{00}^2\,ds
   + C_\delta\int_0^t \abs{  \eta(s)}_{0}^2\,ds.
 \end{equation}
\end{enumerate}
 \end{prop}
\begin{proof}
Statement $(1)$ follows from Theorem \ref{growingmode}.

We next prove Statement $(2)$ by using Theorem \ref{lineargrownth}. We observe that $U(t)-\iota e^{\lam t}U^\star$ solves the problem \eqref{linear ho} with initial data $U(0)-\iota U^\star$ and the force terms $G^i$ given by \eqref{G1_def}--\eqref{G3_def}. Then Statement $(2)$ follows from \eqref{result123} by noticing that $
\mathfrak{N}_{G}\ls  \Lvert3U  \Rvert3_{00}^2.$

We now prove Statement $(3)$. For the case $\bar B_3\neq 0$, taking $\varepsilon=\lam $ in the estimate \eqref{non-horizontal es} of Theorem \ref{engver}, since $\Lvert3  U \Rvert3_{00} \simeq \feb{2N}
+ \lam ( \fe{2N}+ \f)$, we obtain \eqref{energyes}; for the case $\bar B_3= 0$, taking $\varepsilon=\lam $ in the estimate \eqref{horizontal es} of Theorem \ref{enghor}, since $\Lvert3  U \Rvert3_{00} \simeq   \fe{2N}+ \f^{\lam }$, we obtain \eqref{energyes}.
\end{proof}

We now  construct a curve of small initial data that satisfy the compatibility conditions for the nonlinear problem \eqref{reformulationic}, which are close to the linear growing mode in Proposition \ref{prop1}. Moreover, the initial data satisfy the structure condition \eqref{eta00}, which was used in the derivation of the nonlinear energy estimates.

\begin{prop}\label{intialle}
  Let $U^\star$ be the linear growing mode stated in Proposition \ref{prop1}. Then there exists  a number $\iota_0>0$ and a family of initial data
  \begin{equation}\label{initial0}
U_0^\iota
=\iota U^\star+\iota^2
\tilde{U}(\iota)
\end{equation}
for $\iota\in [0,\iota_0)$ so that the followings  hold.

1. $U_0^\iota$ satisfy the nonlinear compatibility conditions required for a solution to the nonlinear problem \eqref{reformulationic} to exist in the norm $\Lvert3\cdot\Rvert3_{00}$. Moreover, $\eta_0^\iota$ satisfy \eqref{eta00}.

2. There exist $C_3, C_4>0$ independent of $\iota$ so that
\begin{equation}\label{initial1}
\Lvert3\tilde{U}(\iota) \Rvert3_{00} \le
C_3
\end{equation}
and
\begin{equation}\label{initial2}
\Lvert3U_0^\iota\Rvert3_{00}^2\le C_4\iota^2.
\end{equation}
 \end{prop}
 \begin{proof}
Note that the nonlinear problem is slightly
perturbed from the linearized problem and so their compatibility conditions for the small initial data should be close to each other. On the other hand, by \eqref{phhicon}, the structure condition \eqref{eta00} for the nonlinear problem is also close to the divergence free condition that is satisfied by the linear growing mode. We may employ the abstract argument before  Lemma 5.3 of \cite{JT} that uses the implicit function theorem to have our conclusion.
\end{proof}

With Propositions \ref{prop1} and \ref{intialle} in hand, we can now present the
\begin{proof}[Proof of Theorem \ref{maintheorem}]
Recall the notation \eqref{Uvector}. First, we restrict to $0<\iota<\iota_0\le \theta_0$, where $\iota_0$ is as small as in Proposition \ref{intialle}  and  the value of $\theta_0$ is   sufficiently small to be determined later. For $0<\iota\le \iota_0$, we let $U_0^\iota$ be the initial data given in Proposition \ref{intialle}.  By further restricting $\iota$ the local well-posedness allows us find  $U^\iota(t)$,  solutions to \eqref{reformulationic} with
 \begin{equation}
 \left.U^\iota \right|_{t=0}=
U_0^\iota
=\iota U^\star+\iota^2
\tilde{U}(\iota).
\end{equation}

Fix $\delta>0$ as small as in Proposition \ref{prop1}, and let $C_\delta>0$ be the constant appearing in Proposition \ref{prop1} for this fixed choice of $\delta$.   We then define $\tilde{\delta}=\min\{\delta,\frac{\lam}{2C_\delta}\}$.  Denote
 \begin{equation}
 T^\ast=\sup\left\{ s : \Lvert3U^\iota(t) \Rvert3_{00}   \le\tilde{\delta}, \text{ for } 0\le t\le s\right\}
 \end{equation}
 and
 \begin{equation}
 T^{\ast\ast}=\sup\left\{ s: \abs{ \eta^\iota(t)}_{0}\le 2 \iota e^{\lam t}, \text{ for } 0\le t\le s\right\}.
 \end{equation}
 With $\iota_0$ small enough, \eqref{initial2} and the local well-posedness guarantee that $T^\ast$ and $T^{\ast\ast}>0$.  Recall that $T^\iota$ is defined by \eqref{escape_time}.   Then for all $t\le \min\{T^\iota,T^\ast,T^{\ast\ast}\}$, we deduce from the estimate \eqref{energyes} of Proposition \ref{prop1}, the definitions of $T^\ast$ and $T^{\ast\ast}$, and \eqref{initial2} that
 \begin{equation}\label{ins1}
 \begin{split}\Lvert3U^\iota(t)\Rvert3_{00}^2
&\le  C_\delta\Lvert3U^\iota_0 \Rvert3_{00}^2
   +  \lam\int_0^t\Lvert3U^\iota(s) \Rvert3_{00}^2\,ds
   + C_\delta\int_0^t \abs{ \eta^\iota(s)}_{0}^2\,ds
 \\&\le \lam \int_0^t\Lvert3 U^\iota(s) \Rvert3_{00}^2\,ds+C_\delta C_4\iota^2+\frac{C_\delta(2\iota)^2}{2\lam}e^{2\lam t}
  \\&\le \lam\int_0^t\Lvert3U^\iota(s) \Rvert3_{00}^2\,ds+C_5 \iota^2e^{2\lam t}.
\end{split}
\end{equation}
for some constant $C_5>0$ independent of $\iota$. We may view \eqref{ins1} as a differential inequality.  Then Gronwall's theorem implies that
\begin{equation}\label{ins11}
 \begin{split}\Lvert3U^\iota(t)\Rvert3_{00}^2
&\le   C_5 \iota^2e^{2\lam t}+C_5\iota^2e^{\lam t}\int_0^t \lam e^{\lam s}\,ds
 \\&\le   C_5 \iota^2e^{2\lam t}+ C_5\iota^2 e^{2\lam t}  =  2C_5\iota^2e^{2\lam t}.
\end{split}
\end{equation}
We then deduce from the estimates \eqref{bbbf} of Proposition  \ref{prop1} and \eqref{ins11} that
\begin{equation}\label{ins12}
 \begin{split}
\abs{\eta_3^\iota(t)-\iota e^{\lam t}\eta^\star_3}_{0}  &\le
C_2 e^{ \Lam t}\left(\Lvert3\iota^2
\tilde{U}(\iota) \Rvert3_{00}+\Lvert3U^\iota(0)  \Rvert3_{00}^2\right) +C_2\int_0^t  \Lvert3U^\iota(s)  \Rvert3_{00}^2 ds
 \\&\quad+C_2\int_0^t\sqrt{\int_0^\tau  e^{2\Lam
(\tau-s)} \Lvert3 U^\iota(s)  \Rvert3_{00}^2 (\Lvert3U^\iota(s)-\iota e^{\lam s} U^\star   \Rvert3_{00}+\Lvert3 U^\iota(s)  \Rvert3_{00}^2)ds}d\tau
 \\&\le
C_2 e^{ \Lam t}(C_3\iota^2+C_4\iota^2)+C_2\int_0^t  2C_5\iota^2e^{2\lam s}
  \\&\quad+C_2\int_0^t\sqrt{\int_0^\tau  e^{2\Lam
(\tau-s)} 2C_5\iota^2e^{2\lam s} (\sqrt{2C_5}\iota e^{ \lam s}+C_1\iota e^{ \lam s}+2C_5\iota^2e^{2\lam s})ds}d\tau
 \\&\le
C_6 e^{ \Lam t}\iota^2+ C_6\iota^2 e^{2\lam t}+C_6\iota^2 e^{2\lam t}+C_6\iota^\frac{3}{2} e^{\frac{3}{2}\lam t}
 \\&\le 3C_6  \iota^2 e^{2\lam t}+C_6\iota^\frac{3}{2} e^{\frac{3}{2}\lam t}.
\end{split}
\end{equation}

Now we claim that
\begin{equation}\label{Tmin}
T^\iota= \min\{T^\iota,T^\ast,T^{\ast\ast}\}
\end{equation}
by fixing $\theta_0$ small enough, namely, setting
 \begin{equation}
 \theta_0=\min\left\{\frac{\tilde{\delta}}{2\sqrt{2C_5}},\frac{1}{12C_6},\frac{1}{16C_6^2}\right\}.
 \end{equation}
 Indeed, if $T^\ast= \min\{T^\iota,T^\ast,T^{\ast\ast}\}$, then by \eqref{ins11}, we have
  \begin{equation}
  \begin{split}
  \Lvert3U^\iota(T^\ast)\Rvert3_{00}
\le    \sqrt{2C_5}\iota e^{ \lam T^\ast} \le   \sqrt{2C_5}\iota e^{ \lam
T^\iota}=\sqrt{2C_5}\theta_0\le \frac{\tilde{\delta}}{2} < \tilde{\delta},
\end{split}
\end{equation}
which contradicts to the definition of $T^\ast$.  If $T^{\ast\ast} =\min\{T^\iota,T^\ast,T^{\ast\ast}\}$, then by \eqref{ins12}, we have
\begin{equation}
\begin{split}
\abs{\eta^\iota_3(T^{\ast\ast})}_{0}   &\le\iota e^{\lam T^{\ast\ast} }\abs{\eta^\star_3}_{0}
  +\abs{ \eta_3^\iota(T^{\ast\ast})-\iota e^{\lam T^{\ast\ast}}\eta^\star_3}_{0}\\&\le \iota e^{\lam T^{\ast\ast} }\abs{\eta^\star_3}_{0}
+ 2C_6\iota^2 e^{2\lam T^{\ast\ast}}+C_6\iota^\frac{3}{2} e^{\frac{3}{2}\lam t}
\\
& \le  \iota e^{\lam T^{\ast\ast}} (1 +3C_6\iota e^{ \lam T^{\iota}} +C_6\sqrt{\iota}  e^{\hal \lam T^\iota})
 \\&\le \iota e^{\lam T^{\ast\ast}}(1 +3C_6\theta_0
+C_6\sqrt{\theta_0})<2\iota e^{\lam T^{\ast\ast}},
\end{split}
\end{equation}
which contradicts to the definition of $T^{\ast\ast}$.  Hence \eqref{Tmin} must hold, proving the claim.

Now we use \eqref{ins12} again  to find that
\begin{equation}
\begin{split}
\abs{\eta^\iota_3(T^\iota)}_{0}   &\ge \iota e^{\lam T^\iota }\abs{\eta^\star_3}_{0}
  -\abs{ \eta_3^\iota(T^\iota)-\iota e^{\lam T^\iota}\eta^\star_3}_{0}
  \\&\ge\iota e^{\lam T^\iota } -3C_6\iota^2 e^{ 2\lam T^\iota}-C_6\iota^\frac{3}{2} e^{\frac{3}{2}\lam t}
  \\&\ge\theta_0- 3C_6\theta_0^2- C_6\theta_0^{\frac{3}{2}} \ge
 \frac{\theta_0}{2}.
\end{split}
\end{equation}
This completes the proof of Theorem \ref{maintheorem}.
\end{proof}

\appendix

\section{Analytic tools}\label{section_appendix}

\subsection{Analytic inequalities}\label{app_po}

We will need some estimates of the product of functions in Sobolev spaces.

\begin{lem}\label{sobolev}
Let $U$ denote a domain either of the form $\Omega_\pm$ or of the form $\Sigma$.
\begin{enumerate}
 \item Let $0\le r \le s_1 \le s_2$ be such that  $s_1 > n/2$.  Let $f\in H^{s_1}(U)$, $g\in H^{s_2}(U)$.  Then $fg \in H^r(U)$ and
\begin{equation}\label{i_s_p_01}
 \norm{fg}_{H^r} \lesssim \norm{f}_{H^{s_1}} \norm{g}_{H^{s_2}}.
\end{equation}

\item Let $0\le r \le s_1 \le s_2$ be such that  $s_2 >r+ n/2$.  Let $f\in H^{s_1}(U)$, $g\in H^{s_2}(U)$.  Then $fg \in H^r(U)$ and
\begin{equation}\label{i_s_p_02}
 \norm{fg}_{H^r} \lesssim \norm{f}_{H^{s_1}} \norm{g}_{H^{s_2}}.
\end{equation}
\end{enumerate}
\end{lem}
\begin{proof}
These results are standard and may be derived, for example, by use of the Fourier characterization of the $H^s$ spaces and extensions.
\end{proof}

We will need the following version of Korn's inequality.

\begin{lem}\label{korn}
It holds that for all $w\in H_0^1(\Omega)$,
\begin{equation}\label{korneq}
\norm{w}_{1} \lesssim \norm{\mathbb{D}w }_{0}.
\end{equation}
\end{lem}
\begin{proof}
 See Lemma 2.7 of \cite{B1}.
\end{proof}

The following provides the $H^{-1/2}(\Sigma)$ boundary estimates of $w$ knowing $w,\diverge w\in L^2(\Omega)$.
\begin{lem}\label{h-1/2}
If $w,\diverge w\in L^2(\Omega)$, then
\begin{equation}
\abs{w}_{-1/2} \lesssim \norm{w}_{0}+\norm{\diverge w }_{0}.
\end{equation}
\end{lem}
\begin{proof}
These result is classical \cite{T} and can be derived by use of the duality and extension.
\end{proof}

We have the following Poincar\'e type inequality.
\begin{lem}\label{app pb}
For any constant vector $\bar B \in \mathbb{R}^3$ with $\bar B_3\neq 0$, it holds that for all $w\in H_0^1(\Omega)$,
\begin{equation}\label{poB}
\ns{ w}_0  \le \frac{m^2+\ell^2}{\bar B_3^2}\ns{(\bar B\cdot\nabla)w }_0 .
\end{equation}
\end{lem}
\begin{proof}
For any $w\in H^1_0(\Omega)$, by the fundamental theory of calculous and the Cauchy-Schwartz inequality, since $\bar B_3\neq 0$, we deduce that for any  $x=(x_\ast,x_3)\in \Omega_+$,
\begin{equation}\label{aineq1}
\begin{split}
 w(x_\ast,x_3)&=w\left(x_\ast+\frac{\ell-x_3 }{\bar B_3}\bar B_\ast,\ell\right)-\int_0^{\frac{\ell-x_3}{\bar B_3}}\frac{d}{ds}\left(w(x_\ast+s\bar B_\ast, s\bar B_3)\right)ds
 \\&=-\int_0^{\frac{\ell-x_3}{\bar B_3}}(\bar B\cdot\nabla)w(x_\ast+s\bar B_\ast, s\bar B_3)ds
 \\&\le \left(\frac{\ell-x_3}{\bar B_3}\right)^\hal\left(\int_0^{\frac{\ell-x_3}{\bar B_3}}\abs{(\bar B\cdot\nabla)w(x_\ast+s\bar B_\ast, s\bar B_3)}^2ds\right)^\hal
 \\&\le \left(\frac{\ell}{\bar B_3}\right)^\hal\left(\int_0^{\frac{\ell}{\bar B_3}}\abs{(\bar B\cdot\nabla)w(x_\ast+s\bar B_\ast, s\bar B_3)}^2ds\right)^\hal.
\end{split}
\end{equation}
By taking the square of \eqref{aineq1} and then integrating over $x_\ast\in \mathbb{R}^2$, using the Fubini theorem and the change of variables, we have
\begin{equation}\label{pbbb}
\begin{split}
\int_{\mathbb{R}^2}  w^2(x_\ast,x_3)dx_\ast& \le \frac{\ell}{\bar B_3}\int_{\mathbb{R}^2} \int_0^{\frac{\ell}{\bar B_3}}\abs{(\bar B\cdot\nabla)w(x_\ast+s\bar B_\ast, s\bar B_3)}^2ds
\\&=\frac{\ell}{\bar B_3}  \int_0^{\frac{\ell}{\bar B_3}}\int_{\mathbb{R}^2}\abs{(\bar B\cdot\nabla)w(x_\ast+s\bar B_\ast, s\bar B_3)}^2ds   \\&=\frac{\ell}{\bar B_3}  \int_0^{\frac{\ell}{\bar B_3}}\int_{\mathbb{R}^2}\abs{(\bar B\cdot\nabla)w(x_\ast , s\bar B_3)}^2ds dx_\ast
=\frac{\ell}{\bar B_3^2} \int_{\Omega_+} \abs{(\bar B\cdot\nabla)w }^2.
\end{split}
\end{equation}
Integrating \eqref{pbbb} in $x_3$ over $(0,\ell)$ yields
\begin{equation}\label{pbbb2}
\int_{\Omega_+}  w^2\le \frac{\ell^2}{\bar B_3^2} \int_{\Omega_+} \abs{(\bar B\cdot\nabla)w }^2.
\end{equation}
Similarly, we have
\begin{equation}
\int_{\Omega_-}  w^2  \le \frac{m^2}{\bar B_3^2} \int_{\Omega_-} \abs{(\bar B\cdot\nabla)w }^2.
\end{equation}
Combining these two shows \eqref{poB}.
\end{proof}

We have the following trace estimates.
\begin{lem}\label{tracethb}
 For any constant vector $\bar B \in \mathbb{R}^3$ with $\bar B_3\neq 0$, it holds that for all $w\in H^1_0(\Omega)$,
\begin{equation}\label{tra2b}
\abs{w}_{0} \lesssim \frac{1}{\sqrt{\abs{\bar B_3}}}\norm{(\bar B\cdot\nabla)w}_{0}^{1/2}\norm{w }_{0}^{1/2}.
\end{equation}
\end{lem}
\begin{proof}
Note that for any $w\in H^1_0(\Omega)$,  since $\bar B_3\neq 0$, we deduce that for any  $ x_\ast \in  \mathbb{R}^2$,
\begin{equation}\label{aineq122}
\begin{split}
 w^2(x_\ast,0)&=w^2\left(x_\ast+\frac{\ell}{\bar B_3}\bar B_\ast,\ell\right)-\int_0^{\frac{\ell}{\bar B_3}}\frac{d}{ds}\left(w^2(x_\ast+s\bar B_\ast, s\bar B_3)\right)ds
 \\&=-2\int_0^{\frac{\ell }{\bar B_3}}((\bar B\cdot\nabla)w w)(x_\ast+s\bar B_\ast, s\bar B_3)ds
 \\&\ls \left(\int_0^{\frac{\ell }{\bar B_3}}\abs{(\bar B\cdot\nabla)w(x_\ast+s\bar B_\ast, s\bar B_3)}^2ds\right)^\hal
  \left(\int_0^{\frac{\ell }{\bar B_3}}\abs{ w(x_\ast+s\bar B_\ast, s\bar B_3)}^2ds\right)^\hal.
\end{split}
\end{equation}
The estimates \eqref{tra2b} follows by integrating \eqref{aineq122} over $x_\ast\in \mathbb{R}^2$, using the Cauchy-Schwarz inequality and then employing the change of variables as in Lemma \ref{app pb}.
\end{proof}

\begin{remark}\label{traceth}
It follows similarly as Lemma \ref{tracethb} that for all $w\in H^1(\Omega)$,
\begin{equation}\label{tra2}
\abs{w}_{0} \lesssim \norm{w}_{0}^{1/2}\norm{w}_{1}^{1/2}.
\end{equation}
\end{remark}
\subsection{Stokes regularity}\label{app stokes}

Let $U$ denote a domain of the form $\Omega_\pm$.
We first recall the classical regularity theory for the one-phase Stokes system with Dirichlet boundary conditions:
\begin{equation}\label{stokes eq}
\begin{cases}
-\mu\Delta u +\nabla p =f  \quad &\hbox{in }U
\\\diverge{u} =g  \quad  &\hbox{in }U
\\u=h\quad &\hbox{on }\pa U.
\end{cases}
\end{equation}
\begin{lem}\label{i_linear_elliptic2}
Let $r\ge 2$. If $f\in H^{r-2}(U)$, $g\in H^{r-1}(U)$ and $h\in H^{r-1/2}(\pa U)$ be given such that
\begin{equation}
\int_U g =\int_{\pa U} h\cdot \nu,
\end{equation}
where $\nu$ is the outer-normal to $\pa U$, then there exists unique $u\in H^r(U),  \nabla p\in H^{r-2}(U)$ solving \eqref{stokes eq}. Moreover,
\begin{equation}\label{stokes es}
\norm{u}_{r}+\norm{\nabla p}_{r-2}\lesssim\norm{f}_{r-2}+\norm{g}_{r-1}+\abs{h}_{r-1/2}.
\end{equation}
\end{lem}
\begin{proof}
 See \cite{L,T}.
\end{proof}

Now considering the two-phase stationary Stokes problem
\begin{equation}\label{cS}
\begin{cases}
-\mu \Delta u +\nabla {p} =F^1  &\hbox{ in }\Omega
\\ \diverge{u} =F^2    &\hbox{ in}\ \Omega
\\
 \Lbrack u\Rbrack=0,\quad \Lbrack(pI-\mu\mathbb{D}(u))e_3\Rbrack=F^3   &\hbox{ on }\Sigma
\\  u =0 &\hbox{ on }\Sigma_{m,\ell},
\end{cases}
\end{equation}
we have the following elliptic regularity theory.
\begin{lem}\label{cStheorem}
Let $r\ge 2$.  If $F^1\in  {H}^{r-2}(\Omega),\ F^2\in
 {H}^{r-1}(\Omega)$ and $F^3\in  {H}^{r-3/2}(\Sigma)$, then there exists unique $u\in H^r(\Omega),  \nabla p\in H^{r-2}(\Omega)$ (and $\jump{p}\in H^{r-3/2}(\Sigma)$) solving \eqref{cS}.
 Moreover,
 \begin{equation} \label{cSresult}
 \norm{u}_{r}+\norm{\nabla p}_{r-2}+\abs{\jump{ p}}_{r-3/2} \lesssim \norm{F^1}_{r-2}+\norm{F^2}_{r-1}+\abs{F^3}_{r-3/2}.
 \end{equation}
\end{lem}
\begin{proof}
It follows by  using first the flatness of the interface $\Sigma$ to get the   estimates of $u$ on $\Sigma$ and then applying Lemma
\ref{i_linear_elliptic2} to $\Omega_\pm$, respectively. We may refer to Theorem 3.1 of \cite{WTK} for a proof with slight modifications.
\end{proof}

\subsection{Lemmas related to pressure}\label{app pressure}
The following lemma allows us to adjust the divergence of the velocity to avoid estimating the pressure term.
\begin{lem}\label{div11}
Let $ r\ge 1$. If $p\in {H}^{r-1}(\Omega)$ satisfies $\int_\Omega p =0$, then there exists $v\in H_0^1(\Omega)\cap {H}^{r}(\Omega)$ so that $\diverge{v}=p$ in $\Omega$ and
\begin{equation}\label{dives11}
\norm{v}_{r}\lesssim\norm{p}_{r-1}.
\end{equation}
\end{lem}
\begin{proof}
For $r=1$, the result is classical \cite{L,T,SS}; for $r\ge 2$, we may refer to  Lemma 3.2 of \cite{WTK} for a proof with slight modifications.
\end{proof}

The following lemma allows us to introduce the pressure as a Lagrange multiplier.

\begin{lem}\label{Pressure}
If $\Lambda \in (H_0^{1}(\Omega))^\ast$ is such that $\Lambda(\varphi) = 0$ for all $\varphi \in  H_{0,\sigma}^1(\Omega)$, then there exists a unique $p \in  L^2_{loc}(\Omega)$ (up to constants) so that
\begin{equation}
\int_\Omega p \diverge \varphi= \Lambda(\varphi) \text{ for all } \varphi\in  H_0^{1}(\Omega).
\end{equation}
\end{lem}
\begin{proof}
 See \cite{L,T,SS}.
\end{proof}




\begin{thebibliography}{99}




\bibitem{B1} J. Beale. The initial value problem for the Navier-Stokes equations with a free surface. \emph{Comm. Pure Appl. Math.} \textbf{34} (1981), no. 3, 359--392.


\bibitem{Ca}
H. Cabannes. {\it Theoretical Magnetofludynamics}. Academic Press, New York, 1970.

\bibitem{3C} S. Chandrasekhar.  {\it Hydrodynamic and Hydromagnetic Stability}.  The International Series of Monographs on Physics, Clarendon Press, Oxford, 1961.



\bibitem{Co}
T. G. Cowling. {\it Magnetohydrodynnamics}. Institute of Physics Publishing, 1976.




\bibitem{3E} D. Ebin. Ill-posedness of the Rayleigh-Taylor and Helmholtz problems for incompressible fluids. \emph{Comm. Partial Differential Equations} \textbf{13} (1988), no. 10, 1265--1295.

\bibitem{GLL}

J.-F. Gerbeau, C. Le Bris, T. Leli$\grave{e}$vre. Mathematical Methods for the Magnetohydrodynamics of Liquid Metals. Numerical Mathematics and Scientific Computation. Oxford University Press, Oxford, 2006.


\bibitem{GS} Y. Guo, W. Strauss. Instability of periodic BGK equilibria. \emph{Comm. Pure Appl. Math.} \textbf{48} (1995), no. 8, 861--894.

\bibitem{3GT1} Y. Guo, I. Tice. Compressible, inviscid Rayleigh-Taylor instability. \emph{Indiana Univ. Math. J.} \textbf{60} (2011), 677--712.

\bibitem{3GT2} Y. Guo, I. Tice.  Linear Rayleigh-Taylor instability for viscous, compressible fluids. \emph{SIAM J. Math. Anal.} \textbf{42} (2010), no. 4, 1688--1720.

\bibitem{GT_lwp} Y. Guo, I. Tice. Local well-posedness of the viscous surface wave problem without surface tension.  \emph{Anal. PDE} \textbf{6} (2013), no. 2, 287--369.

\bibitem{GT_per}
Y. Guo, I. Tice.  Almost exponential decay of periodic viscous surface waves without surface tension.  \emph{Arch. Ration. Mech. Anal.} \textbf{207} (2013), no. 2, 459--531.

\bibitem{GT_inf} Y. Guo, I. Tice. Decay of viscous surface waves without surface tension in horizontally infinite domains.   \emph{Anal. PDE} \textbf{6} (2013), no. 6, 1429--1533.

\bibitem{hw_guo} H. Hwang, Y. Guo. On the dynamical Rayleigh-Taylor instability. \emph{Arch. Ration. Mech. Anal.} \textbf{167} (2003), no. 3, 235--253.

\bibitem{JJ}F. Jiang, S. Jiang. On instability and stability of three-dimensional gravity driven viscous flows in a bounded domain. \emph{Adv. Math.}  \textbf{264}  (2014), 831--863.

\bibitem{JJ2}F. Jiang, S. Jiang. On linear instability and stability of the Rayleigh-Taylor problem in magnetohydrodynamics. \emph{J. Math. Fluid Mech.}  \textbf{17}  (2015),  no. 4, 639--668.

\bibitem{JJW}F. Jiang, S. Jiang, Y. J. Wang. On the Rayleigh-Taylor instability for the incompressible viscous magnetohydrodynamic equations.
\emph{Comm. Partial Differential Equations} \textbf{39} (2014), no. 3, 399--438.

\bibitem{JT} J. Jang, I. Tice.  Instability theory of the Navier--Stokes--Poisson equations. \emph{Anal. PDE} \textbf{6} (2013), no. 5, 1121--1181.


\bibitem{JTW_NRT} J. Jang, I. Tice, Y. J. Wang. The compressible viscous surface-internal wave problem: nonlinear Rayleigh-Taylor instability. \emph{Arch. Ration. Mech. Anal.} \textbf{221}  (2016),  no. 1, 215--272.


\bibitem{JTW_GWP} J. Jang, I. Tice, Y. J. Wang.  The compressible viscous surface-internal wave problem: stability and vanishing surface tension limit. \emph{Comm. Math. Phys.} \textbf{343} (2016),  no. 3, 1039--1113.







\bibitem{3K} H. Kull. Theory of the Rayleigh-Taylor instability. \emph{Phys. Rep.} \textbf{206} (1991), no. 5, 197--325.


\bibitem{L}
O. A. Ladyzhenskaya. \emph{The Mathematical Theory of Viscous Incompressible Flows.}  Gordon and Breach, New York, 1969.

\bibitem{LL}
L. D. Landau, E. M. Lifshitz. \emph{Electrodynamics of Continuous Media}. 2nd ed., Pergamon, New York, 1984.


\bibitem{Lin}
F. H. Lin. Some analytical issues for elastic complex fluids. \emph{Comm. Pure Appl. Math.} \textbf{65}  (2012),  no. 7, 893--919.









\bibitem{PSO}
M. Padula, V. A. Solonnikov. On the free boundary problem of magnetohydrodynamics. \emph{J. Math.
Sci.} \textbf{178} (2011), 313--344.

\bibitem{PS2} J. Pr\"{u}ss, G. Simonett. On the Rayleigh-Taylor instability for the two-phase Navier-Stokes equations. \emph{Indiana Univ. Math. J.} \textbf{59} (2010), 1853--1872.

\bibitem{Ra} L. Rayleigh. Analytic solutions of the Rayleigh equation for linear density profiles. \emph{Proc. London. Math. Soc.} \textbf{14} (1883), 170--177.

\bibitem{SS} V. A. Solonnikov, V. E. Skadilov.  On a boundary value problem for a stationary system of Navier-Stokes equations.  \emph{Proc. Steklov Inst. Math.} \textbf{125} (1973), 186--199.

\bibitem{RG}
R. M. Strain, Y. Guo. Almost exponential decay near Maxwellian. \emph{Comm. Partial
Differential Equations} \textbf{31} (2006), no. 1-3, 417--429.


\bibitem{Ta} G. I. Taylor. The instability of liquid surfaces when accelerated in a direction perpendicular to their planes.  \emph{Proc. Roy. Soc. London} Ser. A. \textbf{201} (1950), 192--196.

\bibitem{T}
R. Temam. \emph{Navier-Stokes Equations: Theory and Numerical Analysis}. Third edition. North-Holland, Amsterdam, 1984.

\bibitem{W}
Y. J. Wang. Critical magnetic number in the magnetohydrodynamic Rayleigh-Taylor instability. \emph{J. Math. Phys.} \textbf{53} (2012),  no. 7, 073701.

\bibitem{WT} Y. J. Wang, I. Tice. The viscous surface-internal wave problem:  nonlinear Rayleigh-Taylor instability. \emph{Comm. Partial Differential Equations} \textbf{37} (2012), no. 11, 1967--2028.

\bibitem{WTK} Y. J. Wang, I. Tice, C. Kim. The viscous surface-internal wave problem: global well-posedness and decay.  \emph{Arch. Ration. Mech. Anal.} \textbf{212} (2014), no. 1, 1--92.

\bibitem{WL} J. Wehausen, E. Laitone. Surface waves. \emph{Handbuch der Physik} Vol. 9, Part 3, pp. 446--778. Springer-Verlag, Berlin, 1960.



\end{thebibliography}
    \end{document}